\documentclass[bibliography=totoc, 12pt]{amsart}
\usepackage[utf8]{inputenc}
\usepackage[margin=1in]{geometry} 
\usepackage{amsmath,amsthm,amssymb,amsfonts, dsfont, graphicx,mathtools,mathrsfs,tikz,hyperref,physics, stmaryrd, bm, tikz-cd,tkz-euclide,comment, enumerate, appendix, todonotes}
\usepackage{subcaption}
\usepackage[backend=biber,bibencoding=utf8,style=numeric,url=false,
    isbn=false,     
    doi=false,           
    abbreviate=true,            
    giveninits=true, 
    natbib=true,
    hyperref=true]{biblatex}
\usetikzlibrary{decorations.pathreplacing,decorations.markings}

\DeclareFieldFormat{postnote}{#1}
\newcommand{\psum}{\sideset{}{^*}\sum}

\newcommand{\ZZ}{\mathbb{Z}}
\newcommand{\QQ}{\mathbb{Q}}
\newcommand{\CC}{\mathbb{C}}

\newcommand{\FF}{\mathbb{F}}
\newcommand{\A}{\mathbb{A}}
\newcommand{\PP}{\mathbb{P}}

\DeclareMathOperator{\vol}{vol}

\DeclareMathOperator{\GL}{GL}

\DeclareMathOperator{\cha}{char}

\DeclareMathOperator{\rk}{rank}

\DeclareMathOperator{\Pic}{Pic}
\def\O{\mathcal{O}}

\def\p{\mathfrak{p}}
\def\m{\mathfrak{m}}
\def\RR{\mathbb{R}}

\theoremstyle{plain}
\newtheorem{theorem}{Theorem}[section] 
\newtheorem{corollary}[theorem]{Corollary}
\newtheorem{lemma}[theorem]{Lemma}
\newtheorem{prop}[theorem]{Proposition}

\newtheorem{conjecture}[theorem]{Conjecture}

\numberwithin{equation}{section}

\theoremstyle{definition}
\newtheorem{definition}[theorem]{Definition}

\theoremstyle{remark}
\newtheorem*{remark}{Remark}

\title{Rational points on del Pezzo surfaces of low degree}
\address{(J.G) IST Austria, Am Campus 1, 3400 Klosterneuburg, Austria.}
\address{(L.H.) McAllister Building, Pennsylvania State University, 16802 State College, USA}

\email{jakob.glas@ist.ac.at, hochfilzer@psu.edu}

\author{Jakob Glas, Leonhard Hochfilzer}
\begin{document}
\maketitle
\begin{abstract}
    We give upper bounds for the number of rational points of bounded anti-canonical height on del Pezzo surfaces of degree at most five over any global field whose characteristic is not equal to two or three. For number fields these results are conditional on a conjecture relating the rank of an elliptic curve to its conductor, while they are unconditional in positive characteristic. For quartic or quintic del Pezzo surfaces with a conic bundle structure, we establish even stronger estimates unconditionally as long as the characteristic is not two.
\end{abstract}
\tableofcontents
\section{Introduction}
Understanding the set of rational points of a variety $X$ over a global field $K$ constitutes one of the cornerstones of modern number theory. When $X$ is a smooth Fano variety, Manin and his collaborators \cite{batyrev1990manin} put forward a conjecture for the counting function 
\[
N_U(B)=\#\{x\in U(K)\colon H(x)<B\},
\]
where $H\colon X(K)\to \RR_{>0}$ is a height function associated to the anti-canonical divisor $-K_X$ of $X$ and $U\subset X$. More precisely, following a refined version of the conjecture due to Peyre~\cite{peyreThin}, one expects the  existence of a thin subset $Z\subset X(K)$ such that for $U=X\setminus Z$, one has 
\[
N_U(B)\sim cB(\log B)^{\rk \Pic(X)-1},
\]
where $c\geq 0$ is Peyre's constant. While Manin's conjecture for curves is well understood, the case of surfaces is already much more mysterious. A surface $X$ that is Fano is called a \emph{del Pezzo surface} and is classified by the degree $d=K_X^2$, which satisfies $1\leq d\leq 9$. Moreover, it is generally believed that for $d\geq 2$ one can take $U$ to be the complement of the exceptional curves in Manin's conjecture. 

If the degree satisfies $6 \leq d \leq 9$ then any del Pezzo surface is a toric variety.
Thanks to work of Batyrev and Tschinkel \cite{BatTschinkToric} for number fields and Bourqui \cite{BourquiI, BourquiII} in positive characteristic, Manin's conjecture is therefore known for all del Pezzo surfaces of degree $6\leq d\leq 9$. If $d \leq 5$, much less is known. One of the notable exceptions is de la Brétèche's work \cite{Bretechedp5}, in which he proved Manin's conjecture for split del Pezzo surfaces of degree 5 over $\QQ$. Recently, with a different approach Browning~\cite{Browningdp5Improved} used different methods to obtain the same result with a better error term. In addition, de la Brétèche and Fouvry \cite{BretFouvry} verified Manin's conjecture for del Pezzo surfaces of degree 5 over $\QQ$ that are the blow-up of a pair of points that are defined over $\QQ$ and a pair of conjugate points over $\QQ(i)$. When $d=4$, the tour de force \cite{browningBretechedp4} of de la Brétèche and Browning provides us with the only example of a del Pezzo surface of degree 4 over $\QQ$ for which we know Manin's conjecture. These results already reflect the guiding principle that the arithmetic of del Pezzo surfaces becomes harder to understand the smaller the degree is. In particular, for $1\leq d \leq 3$ we do not know the truth of Manin's conjecture for any single example of a del Pezzo surface and in fact, if the ground field is not $\QQ$, we do not even know it for any del Pezzo surface of degree $1\leq d \leq 5$. 

While an asymptotic formula remains largely elusive for small degrees, even providing upper bounds remains a substantial challenge in itself.  For $2\leq d\leq 5$, currently the best upper bounds are found in work of Salberger \cite{TheSalberger}, in which he showed that ${N_U(B)\ll_X B^{3/\sqrt{d}+\varepsilon}}$ when the ground field is $\QQ$. Again working over $\QQ$, it follows from combining work of Bhargava et al.~\cite{bhargavaEll} on pointwise bound for the ranks of elliptic curves with work of Helfgott and Venkatesh~\cite{HelfVenkatesh} on integral points of elliptic curves that $N_U(B)\ll_X B^{2.87}$ when $X$ is a del Pezzo surface of degree 1. Moreover, when $d=3$ Heath Brown \cite{heathbrownDiagCubic} succeeded in showing that $N_U(B)\ll_X B^{3/2+\varepsilon}$ and the underlying cubic forms is diagonal conditional on certain conjectures for Hasse-Weil $L$-functions associated to a family of cubic threefolds. The authors~\cite{glas2022question} showed that the same upper bounds holds unconditionally over $\FF_q(t)$ when $\cha(\FF_q)>3$. 

Let us now consider the following conjecture and its consequences for Manin's conjecture for del Pezzo surfaces.
\begin{conjecture}\label{Conj: RGH}
    Let $E$ be an elliptic curve over a global field $K$ with $\cha(K)\neq 2,3$ and $\mathscr{C}_E$ its conductor. Then
    \[
    \rk E = o(\log \mathrm{N}(\mathscr{C}_E))\quad\text{as }\mathrm{N}(\mathscr{C}_E)\to\infty, 
    \]
where $\mathrm{N}(\mathscr{C}_E)$ denotes the norm of $\mathscr{C}_E$.
\end{conjecture}
 Mestre \cite{MestreBoundConductor} showed that Conjecture~\ref{Conj: RGH} is implied by the Birch and Swinnerton-Dyer conjecture. Moreover, 2-descent shows that we always have $\rk E=O(\log \mathrm{N}(\mathscr{C}_E))$. In fact, when $\cha(K)>3$, the conjecture was proven by Brumer~\cite{Brumerelliptic}. The relevance of this conjecture is that the pullback of a generic hyperplane under the birational map $X\to \PP^d$ induced by the anti-canonical divisor is generically a smooth genus 1 curve. Assuming Conjecture \ref{Conj: RGH}, Heath-Brown \cite{heath1998counting} obtained \emph{uniform} upper bounds for the number of rational points of bounded height on planar elliptic curves and showed that $N_U(B)\ll_X B^{4/3+\varepsilon}$ for any cubic surface over $K=\QQ$. Our first main result extends this to any del Pezzo surface of degree at most 5 and to any global field whose characteristic exceeds 3 if it is positive.
\begin{theorem}\label{Th: TheTheorem}
    Let $X$ be a del Pezzo surface of degree $1\leq d\leq 5$ over a global field $K$ with $\cha(K)\neq 2,3$. Then 
    \[
    N_U(B)\ll_X B^{1+1/d +\varepsilon},
    \]
    unconditionally when $\cha(K)>3$ and assuming Conjecture \ref{Conj: RGH} when $\cha(K)=0$. Moreover, when $d=1$ the implied constant is independent of $X$.
\end{theorem}
Using an approach based on exponential sums, Bonolis and Browning~\cite{bonolis2020uniform} obtained the estimate $N_U(B)\ll B^{3-1/20}(\log B)^2$ for $d=1$ over $\QQ$. While this estimate is weaker than the one obtained by Bhargava et al.~\cite{bhargavaEll}, it has the advantage that it is uniform with respect to the underlying surface. The upper bound from Theorem~\ref{Th: TheTheorem} for $d=1$ shares the same uniformity.

Recall that if there is a dominant $K$-morphism $X\to \PP^1$ such that all fibers are plane conics, we say that $X$ admits a \emph{conic bundle structure}. When a del Pezzo surface comes with such extra structure, one can use it to get better control over the number of rational points. In particular, Heath-Brown \cite{HeathBrownCubicConicBundle} has shown that $N_U(B)\ll_X B^{4/3+\varepsilon}$ for $d=3$ over $\QQ$ when $X$ has three coplanar lines defined over $\QQ$, which give rise to three inequivalent conic bundle structures. This result was later generalised to number fields by Broberg~\cite{broberg2001rational}. Moreover, Browning and Sofos \cite{BrowningSofosDP4Conic} proved that 
\[{B(\log B)^{\rk \Pic(X)-1}\ll_X N_U(B)\ll_X B(\log B)^{\rk \Pic(X)-1}}
\]
when $d=4$ for $K=\QQ$ assuming that $X(K)\neq \emptyset$. Building on ideas of Salberger announced at the conference ``Higher
dimensional varieties and rational points'' at Budapest in 2001,  work of Browning an Swarbrick-Jones \cite{browning_S-J} gives $N_U(B)\ll_X B^{1+\varepsilon}$ when $d=4$ and $K$ is a number field. When $d=2$ and $K=\QQ$, Salberger announced at the conference ``G\'eom\`etrie arithm\'etique et vari\'et\'es rationnelles'' at Luminy in 2007 the result that $N_U(B)\ll_X B^{11/6+\varepsilon}$ provided $X$ is split. We are now ready to reveal our second main result.
\begin{theorem}\label{Th: TheConicBundleTheorem}
    Let $X$ be a del Pezzo surface of degree $4$ or $5$ over a global field $K$ of characteristic $\cha(K)\neq 2$ with a conic bundle structure. Then $N_U(B)\ll_X B^{1+\varepsilon}$ for an effectively computable implied constant.
\end{theorem}
This result is new when $d=4$ and $\cha(K)>0$ and new for any global field when $d=5$. We also note that the constants in Theorem~\ref{Th: TheConicBundleTheorem} are all effectively computable, which is in contrast to  the result due to Browning and Swarbrick-Jones~\cite{browning_S-J}. This is because we avoid an application of the Thue--Siegel--Roth theorem, which we note is in general not true in positive characteristic (cf.~\cite{osgood_thue_siegel_roth}).

The results listed so far are by no means exhaustive. In particular, when one considers \emph{singular} del Pezzo surfaces, one can even obtain an asymptotic formula when the degree is $2$ or $3$, but we restrict to the more difficult case of smooth surfaces in this work. Moreover, in \cite{frei_loughran_sofos} Frei, Loughran and Sofos studied lower bounds for del Pezzo surfaces with a conic bundle structure and showed that $N_U(B)\gg B(\log B)^{\rk \Pic (X)-1}$ for del Pezzo surfaces over number fields whose rank of the Picard group is sufficiently large with respect to $d$. 

\subsection{Outline.} The basic idea underlying the proof of Theorem \ref{Th: TheTheorem} is simple: A generic hyperplane section of a del Pezzo surface is a non-singular genus 1 curve and the rank growth hypothesis allows us to obtain uniform upper bounds for the number of rational points of bounded height on elliptic curves. In fact, in Section 3 we establish the following result. 
\begin{prop}\label{Prop.UpperBoundEllCurve}
    Let $E\subset \PP^n$ be a non-singular genus 1 curve of degree $d$ over a global field $K$ with $\cha(K)\neq 2,3$. Assuming that Conjecture \ref{Conj: RGH} holds when $\cha(K)=0$, we have 
    \[
    \#\{x\in E(K)\colon H(x)<B\}\ll B^\varepsilon,
    \]
    where the implied constant only depends on $d$, $n$, $K$ and $\varepsilon$ and $H\colon \PP^n(K)\to \RR_{>0}$ is the usual height on projective space.
\end{prop}
Heath-Brown \cite{heath1998counting} previously proved Proposition~\ref{Prop.UpperBoundEllCurve} for plane elliptic curves over $\QQ$. However, his results still had a dependence on the height of the elliptic curve, which we were able to remove. Proposition~\ref{Prop.UpperBoundEllCurve} alone is not sufficient to prove Theorem \ref{Th: TheTheorem}. The hyperplane sections can also be singular and we need uniform upper bounds for the number of rational points of bounded height on curves. The determinant method developed by Heath-Brown \cite{HeathBrownCurvesSurfaces} has been used by many authors over the last two decades to establish such estimates. Moreover, the recent work of Paredes and Sasyk~\cite{paredes2021uniform} extended these results using the global determinant method due to Salberger~\cite{TheSalberger} to any global field, which enables us to work over arbitrary global fields. These results will also prove vital when bounding the number of degenerate hyperplane sections, which correspond to rational points on the dual variety of $X$. 

When $d=4$ or $d=5$, even the bounds coming from the determinant method are not strong enough when the hyperplane section contains an irreducible component of degree $d-1$, as we need an additional saving with respect to the height of the hyperplane. To overcome this difficulty, we show that such curves are in fact rational and provide upper bounds by exhibiting a uniform parameterisation of the rational points. 

Finally, to prove Theorem \ref{Th: TheConicBundleTheorem}, we use uniform upper bounds for rational points on conics of Browning and Heath-Brown \cite{BrowningHeathBrownRatPtsHypersurfaces}, that we transfer to the setting of positive characteristic. For quintic del Pezzo surface we combine them with techniques from the geometry of numbers in a similar fashion as Bonolis, Browning and Huang \cite{bonolis2022density}. In particular, in Section~\ref{sec.lattices} we prove a result regarding the number of lattice points in a box depending on the successive minima of the lattice for all global fields, which is likely to be useful for applications outside the context of this paper.
\subsection{Conventions} We will make use of the big-$O$ and Vinogradov's notation $O, o, \ll \gg, \asymp,\dots $. Moreover, we will indicate a dependence of the implied constant on certain parameters by subscripts, unless specified otherwise. The letter $\varepsilon$ will denote an arbitrarily small real number whose exact value may change from one occurrence  to the next. Its appearance together with one of the symbols $\ll, \gg, O$ indicates that the implied constant depends on $\varepsilon$. Similarly, in Sections~\ref{Sec.conicbundles} and \ref{Sec.dp345} the letter $A$ denotes an arbitrarily large constant, whose value can change from line to the next. In particular, we may write expressions like $B^{2A}\ll B^A$, which has the advantage of avoiding introducing notation like $A', A'', A''', \dots$. Finally, we shall use the notation $\alpha \sim R$ to indicate that $\alpha$ lies in the dyadic interval $(R, 2R]$. 

\subsection*{Acknowledgements}
We would like to thank Tim Browning for many useful discussions and comments. We would also like to thank Dan Loughran and Mihran Papikian for useful conversations.
\section{Background} \label{sec.background}
\subsection{Geometry} Throughout this work a \emph{variety} over a field $K$ is a separated $K$-scheme of finite type. In this section $K$ denotes an arbitrary field. \subsubsection{del Pezzo surfaces} In this subsection we review some of the basic geometric properties of del Pezzo surfaces, which can for example be found in the book by Manin~\cite{maninCubicFormsBook}. A del Pezzo surface over a field $K$ is a smooth, projective and geometrically integral surface over $K$ whose anti-canonical divisor $-K_X$ is ample.  Let $\bar{X}=X\times K^{\text{sep}}$, where $K^{\text{sep}}$ denotes a separable closure of $K$. The geometric Picard group $\Pic(\bar{X})$ is a finitely generated $\ZZ$-module and since $X$ is smooth, we can identify $\Pic(\bar{X})$ with the class group of Weil divisors on $X$. Therefore, $\Pic(\bar{X})$ comes with a symmetric bilinear intersection pairing $(~\cdot~) \colon \Pic(\bar{X})\times \Pic(\bar{X})\to \ZZ$. If $C\in \Pic(\bar{X})$, then by abuse of notation we shall write $C^2$ for $C\cdot C$. The degree of $X$ is defined to be $K_X^2$ and satisfies $1\leq d\leq 9$.
\begin{definition}
    Let $C\subset \bar{X}$ be an irreducible curve. Then we say that $C$ is \emph{exceptional} if $C^2=C\cdot K_X=-1$. 
\end{definition}
By the adjunction formula, an exceptional curve has arithmetic genus 0 and hence is $\bar{K}$-isomorphic to $\PP^1$. There are at most finitely many exceptional curves on a del Pezzo surface and their precise number is given in Table~\ref{Ta: Exc.curves}. The anti-canonical divisor induces a birational map $X\to \PP^d$, which is in fact a morphism for $d\geq 2$. When $d\geq 3$ the map is an embedding and realises $X$ as a non-degenerate surface of degree $d$ in $\PP^d$. 
\begin{table}[b]
\begin{tabular}{c|c|c|c|c|c|c|c|c|c}
degree& 9 & 8 & 7 & 6 & 5  & 4  & 3  & 2  & 1   \\ \hline
      & 0 & 0 or 1 & 3 & 6 & 10 & 16 & 27 & 56 & 240
\end{tabular}
\caption{Number of exceptional curves}
\label{Ta: Exc.curves}
\end{table}
If $C\subset \bar{X}$ is a geometrically connected curve of arithmetic genus 0 with $C\cdot K_X = -2$, we say that $C$ is a \emph{conic}.
\begin{definition}
    We say that $X$ admits a conic bundle structure over $K$ if there is a dominant $K$-morphism $X\to \PP^1$ all of whose fibers are plane conics.
\end{definition}
It follows from Lemma~5.1 in \cite{frei_loughran_sofos} that $X$ admits a conic bundle structure over $K$ if and only if contains a conic defined over $K$. In~\cite{frei_loughran_sofos} the authors assume  the ground field to be perfect in their statement; however, an inspection of the proof reveals that this assumption is not used.
\subsubsection{Dual varieties} Let $X\subset \PP^n$ be a variety and denote by $\widehat{\PP}^n$ the dual projective space parameterising hyperplanes in $\PP^n$. We define the \emph{conormal variety} of $X$ to be 
\[
Z(X)= \overline{\{(x,H)\in X_{sm}\times \widehat{\PP}^n\colon T_xX\subset H\}},
\]
equipped with the reduced scheme structure, where $T_xX\subset {\PP}^n$ denotes the embedded tangent space of $X$ at $x$ and $X_{sm}$ denotes the smooth locus of $X$.  Let $\phi\colon Z(X)\to \widehat{\PP}^n$ denote the projection onto the second factor. Then the image $\phi(Z(X))\subset \widehat{\PP}^n$ is a variety, which is called the \emph{dual variety} of $X$ and denoted by $X^*$. We always have the inequality $\dim(X^*)\leq n-1$. We say that $X$ is \emph{reflexive} if $Z(X)=Z(X^*)$ under the natural identification $\PP^n=\widehat{\widehat{\PP^n}}$.

If $X$ is smooth, then a hyperplane $H\in \widehat{\PP}^n$ has singular intersection with $X$ if and only if $H\in X^*$, while if $X$ is not smooth and $X\cap H$ is singular, then $H\in X^*$ or $H$ intersects the singular locus of $X$. 

Our goal of this section is to compute $\deg(X^*)$ when $X$ is a del Pezzo surface. To do so, we need to recall the properties of Chern classes as found in \cite{Fulton}. For a smooth variety $X$ over $K$, let $CH(X)$ denote the Chow ring whose $k$th graded piece $A^k$ is the group of cycles of codimension $k$ modulo rational equivalence and multiplication is given by the intersection pairing. Associated to any vector bundle $\mathcal{F}$ over $X$ are the Chern classes $c_i(\mathcal{F})\in A^i$ satisfying 
\begin{enumerate}
    \item $c_0(\mathcal{F})=1$,
    \item $c_i(\mathcal{F})=0$ if $i>\rk(\mathcal{F})$,
    \item if $\mathcal{F}$ is of rank $r$, then $c_1(\mathcal{F})=c_1(\wedge^r \mathcal{F})$.
\end{enumerate}
Moreover, if $X$ is a smooth surface and $\mathcal{L}$ is a line bundle on $X$ corresponding to a Weil divisor $D$, then $\deg(c_1(\mathcal{L})\cap E)=D\cdot E$ for any Weil divisor $E$ of $X$, where $\cdot $ is the usual intersection pairing on the class group of Weil divisors. 

Let $\mathcal{T}_X$ be the tangent bundle on a smooth variety $X\subset \PP^n$ of dimension $k$. Then we define the \emph{delta invariants} of $X$ to be 
\[
\delta_i(X)\coloneqq \sum_{j=i}^k (-1)^{k-j}{j+1 \choose i+1}\deg(c_{k-j}(X)),
\]
where $c_i(X)=c_i(\mathcal{T}_X)$ and $\deg(c_i(X))=\deg(c_i(X)\cap H^{k-i})$ with $H$ the class of a hyperplane section on $X$. We then have the following result due to Holme~\cite[Theorem 3.4]{Holme}.
\begin{prop}\label{Prop: DualReflexive}
    Let $X\subset \PP^n$ be a smooth variety and suppose that $r_0$ is such that \break ${\delta_0(X)=\cdots = \delta_{r_0-1}(X)=0}$, but $\delta_{r_0}(X)\neq 0$. Then $\dim(X^*)=n-1-r_0$ and $X$ is reflexive if and only if $\deg(X^*)=\delta_{r_0}(X)$. 
\end{prop}
\begin{remark}
    In fact Holme's result is even valid for singular varieties. However, in this case the delta invariants may differ from the definition above.
\end{remark}
We now have everything at hand to compute the degree of the dual variety of a del Pezzo surface. 
\begin{prop}
    Let $X\subset \PP^d$ be a del Pezzo surface of degree $d \geq 3$ embedded anti-canonically over a field $K$ with $\cha(K)\neq 2,3$. Then $X^*\subset \widehat{\PP}^d$ is a hypersurface with $\deg(X^*)=12$. 
\end{prop}
\begin{proof}
    Let us first compute $\delta_0(X)$. Since $X$ is embedded anti-canonically,
 the definition of $\delta_0(X)$ gives 
    \begin{equation}\label{Eq: Delta.dP}
    \delta_0(X)= \deg(c_2(X))-2\deg(c_1(X)\cap (-K_X))+3 (-K_X)^2.
    \end{equation}
    We have by definition $(-K_X)^2=d$. Moreover, as $\mathcal{T}_X$ has rank $2$, we have $c_1(X)=c_1(\wedge^2 \mathcal{T}_X)= c_1(-K_X)$ and hence $\deg(c_1(X)\cap (-K_X))= d$ as well. It thus remains to compute $\deg(c_2(X))$. Let $\chi(X,\O_X)$ be the Euler characteristic of the structure sheaf on $X$. By Lemma 3.2.1 of Koll\'{a}r~\cite{kollar}, we have $\chi(X,\O_X)=1$. Moreover, Noether's formula (see Example 15.2.2 of Fulton~\cite{Fulton}) gives 
    \[
    \chi(X,\O_X)=((-K_X)^2 +\deg (c_2(X))/12,
    \]
    which implies $\deg(c_2(X))=12-d$. Once combined with \eqref{Eq: Delta.dP}, we obtain $\delta_0(X)=12$. 

    In the light of Proposition~\ref{Prop: DualReflexive} it thus remains to show that $X$ is reflexive. To do so, we make use of the Monge-Segre-Wallace criterion in the form given by Kleiman \cite[(4) Theorem]{Kleiman}, which asserts that $X$ is reflexive if and only if the map $\phi\colon Z(X)\to X^*$ is separable. The map $\phi$ is known to be finite and by definition its degree $\deg(\phi)$ is the degree of the induced extension of function fields $K(X^*)/K(Z(X))$. By the last equation on page 152 of Holme~\cite{Holme}, we have $\deg(\phi)\deg(X^*)=\delta_0(X)$. In particular, $\deg(\phi)\mid 12$. As we assume that $\cha(K)\neq 2,3$, this automatically implies that $K(X^*)/K(Z(X))$ is separable and hence $X$ is reflexive by the Monge-Segre-Wallace criterion. 
\end{proof}
\begin{remark}
    At least the condition $\cha(K)\neq 2$ is necessary in the last proposition. Indeed, it follows for example from Lemma~4.1 of the authors' work \cite{glas2022question} that if $K=\FF_q(t)^{\text{sep}}$ with $\cha(\FF_q)=2$, then the dual variety of the Fermat cubic surface is again a Fermat cubic surface.
\end{remark}
\subsection{Algebraic number theory} We call $K$ a \emph{global field} if it is a finite extension of $\QQ$ or the function field of a curve over a finite field. When $K$ is a function field of positive characteristic $p$, suppose $\FF_q$ is the field of constants of $K$. We can then always find an element $t\in K$ that is transcendental over $\FF_q$ and such that $K/\FF_q(t)$ is separable. Note that the choice of $t$ is not unique, but this is irrelevant to us. We shall then write $k=\FF_q(t)$ and if $K$ is a number field, we write $k=\QQ$ and in either case define $d_K=[K\colon k]$.

The ring of integers $\O_K$ of $K$ is  by definition the integral closure of $\ZZ$ and $\FF_q[t]$ respectively. By definition, a non-archimedean place $\p$ of $K$ is a discrete valuation ring $\O_{(\p)}\subset K$ with field of fractions $K$ with the additional constraint $\FF_q\subset \O_{(\p)}$ when $\cha(K)>0$. A non-archimedean place corresponds to an embedding $K\to \CC$ and only exists in characteristic $0$. We then define $\Omega_{K,\infty}$ to be the set of places lying above the infinite place $t^{-1}$ in $\FF_q(t)$ when $\cha(K)>0$ and to be the set of archimedean places when $\cha(K)=0$. In either case $\Omega_K$ denotes the set of all places of $K$ and $\Omega_{K,f}=\Omega_K\setminus \Omega_{K,\infty}$ the set of finite places. In addition, we let $s_K$ be the cardinality of $\Omega_{K,\infty}$. 

If $\nu$ is a place of $K$ and $\mu$ the corresponding place of $k$ lying below it, we let $K_\nu$ and $k_\mu$ be the completions of $K$ and $k$ with respect to $\nu$ and $\mu$ respectively and define the local degree $d_\nu\coloneqq [K_\nu\colon k_\mu]$. If $\nu$ is non-archimedean, we let $\O_\nu \subset K_\nu$ be the ring of integers and $\m_\nu$ its maximal ideal. For any place $\nu$, we define an absolute value on $K$ via 
\[
|x|_\nu \coloneqq \begin{cases}
    |x|_\infty^{d_\nu} &\text{if }\nu \text{ is archimedean,}\\
    \#(\O_{\nu}/\m_\nu)^{-v_\nu(x)}&\text{if }\nu \text{ is non-archimedean,}
\end{cases}
\]
where $|\cdot|_\infty$ denotes the usual absolute value on $\CC$ and $v_\nu$ the normalised valuation on $K$ induced by $\nu$. Via the embedding $K\to K_\nu$ corresponding to the place $\nu$ this gives rise to an absolute value on $K$.
\begin{remark}
    Note that if $\nu$ is a complex place, then strictly speaking $|\cdot |_\nu$ is not an absolute value in the usual sense, as it does not satisfy the triangle inequality. However, it still satisfies $|x+y|_\nu \leq 4 \max\{|x|_\nu, |y|_\nu\}$.
\end{remark}
When $\p$ is a prime ideal of $\O_K$, we define its norm to be $\mathrm{N}(\p)\coloneqq \#(\O_K/\mathfrak{p})$ and extend it multiplicatively to all fractional of $\O_K$. If $\alpha_1,\dots, \alpha_m\in K$, we let $\langle \alpha_1,\dots, \alpha_m\rangle$ be the fractional ideal generated by $\alpha_1,\dots, \alpha_m$. By abuse of notation, we shall then write $\mathrm{N}(\alpha_1,\dots, \alpha_m)$ instead of $\mathrm{N}(\langle \alpha_1,\dots,  \alpha_m\rangle)$. Similarly, if $\nu$ is a non-archimedean place we also write $\mathrm{N}(\cdot )$ for the ideal norm on $\O_\nu$. When $K$ is a function field, a divisor is by definition a formal sum $\mathfrak{a}=\sum_{\nu \in \Omega_K}e_\nu \nu$, where $e_\nu \in \ZZ$ is non-zero for at most finitely many $\nu$. We then define $\deg (\mathfrak{a})=\sum e_\nu \deg (\nu)$, where $\deg(\nu) = [(\O_\nu / \mathfrak{m}_\nu)\colon \FF_q]$ is the degree of the residue field extension. We shall then also write $\mathrm{N}(\mathfrak{a})=q^{\deg (\mathfrak{a})}$ and refer to it as the norm of $\mathfrak{a}$.

Further, given $K_\nu$ there exists a standard additive character $\psi_\nu \colon K_\nu \rightarrow \CC^\times$ as defined in~\cite[Chapter 7]{ramakrishnan_valenza}. In particular, these have the property that when $\nu$ is a finite place then $\O_\nu \subset K_\nu$ is the maximal subgroup on which $\psi_\nu$ acts trivially. As a result we obtain the following character orthogonality relation. The proof is standard, so we omit it here.
\begin{lemma} \label{lem.character_orthogonality}
    Let $\nu$ be a finite place of $K$, let $\pi$ be a uniformizer of $K_\nu$ and let $r \geq 1$ be an integer. If $a \in \O_\nu$ we have
    \[
    \frac{1}{\mathrm{N}(\pi)^r} \sum_{x \in \O_\nu/(\pi^r)} \psi_\nu \left(\frac{ax}{\pi^r} \right) = \begin{cases}
        1 \quad &\text{if $a \in (\pi)^r$,} \\
        0 &\text{otherwise.}
    \end{cases}
    \]
\end{lemma}

\subsection{Height functions} We can construct an exponential height on $\PP^n(K)$ via 
\[
H(\bm{x})= \prod_{\nu \in \Omega_K}\max_{0\leq i \leq n}|x_i|_\nu
\]
whenever $\bm{x}=[x_0, \hdots , x_n]\in \PP^n(K)$. Note that our normalisation of the absolute values implies that the product formula $\prod_{\nu \in \Omega_K}|x|_\nu=1$ holds for any $x\in K$, so that the height on $\PP^n(K)$ is indeed well defined. Observe that if $\bm{x}=[x_0,\hdots , x_n] \in \PP^n(K)$ with $(x_0,\hdots, x_n)\in \O_K^{n+1}$, then 
\[
H(\bm{x})= \frac{1}{\mathrm{N}(x_0,\dots, x_n)}\prod_{\nu\in\Omega_{K,\infty}}\max_{0\leq i \leq n}|x_i|_\nu.
\]
In addition, for $\bm{x}\in\O_K^{n+1}$ we define the norms 
\[
\norm{\bm{x}}\coloneqq \max_{0\leq i \leq n} \max_{\nu\mid \infty} |x_i|_\nu \quad \text{and}\quad \norm{\bm{x}}_\infty \coloneqq \prod_{\nu \mid \infty}\max_{0\leq i \leq n} |x_i|_\nu.
\]
Given $\bm{R}=(R_\nu)_{\nu\mid\infty} \in \RR_{>0}^{s_K}$ we define $|\bm{R}|=\prod_{\nu\mid\infty}R_\nu$ and 
\[
L(\bm{R})=\{x\in \O_K\colon |x|_\nu \leq R_\nu \text{ for all }\nu\mid \infty\}.
\]
The following results are all standard over number fields and should be well known over function fields. Due to a lack of statements in the literature, we provide full proofs in the case of function fields.
\begin{lemma}\label{Le: Number.OK.Points}
Let $\bm{R}=(R_\nu)_{\nu\mid\infty} \in \RR_{>0}^{s_K}$ and let $\mathfrak{a}\subset \O_K$ be an integral ideal. We have
\[
|\bm{R}|\mathrm{N}(\mathfrak{a})^{-1}\ \ll_K \#(L(\bm{R})\cap\mathfrak{a})\ll_K \max\{1, |\bm{R}|\mathrm{N}(\mathfrak{a})^{-1}\}.
\]
\end{lemma}
\begin{proof}
    If $K$ is a number field, this is Theorem 0 in Chapter V\S1 of Lang~\cite{LangAlgNT}. We may therefore assume that $K$ is a function field. Without loss of generality, we may assume that $R_\nu =q^{r_\nu}$  for some $r_\nu\in \ZZ$, where $q^{r_\nu}\in |K_\nu|_\nu$. In particular, we may choose elements $a_\nu\in K_\nu$ such that $|a_\nu|_\nu=R_\nu$ for all $\nu\mid\infty$. The condition $x\in (L(\bm{R})\cap \mathfrak{a})$ is then equivalent to 
    \begin{equation}\label{Eq: RRineq}
    v_\nu(x)\geq -v_\nu(a_\nu) \text{ for all }\nu\mid\infty \quad\text{and}\quad v_{\mathfrak{p}}(x)\geq v_{\mathfrak{p}}(\mathfrak{a}) \text{ for all }\mathfrak{p}\mid\mathfrak{a}.
    \end{equation}
    Let $C$ be the smooth projective curve associated to the function field $K$ and let $g$ be its genus. We then define the divisor $D\in \Pic(C)$ to be $D=\sum_{\nu\mid\infty}v_\nu(a_\nu)\nu-\sum_{\mathfrak{p}\nmid \infty}v_{\mathfrak{p}}(\mathfrak{a})\mathfrak{p}$. For $x\in K$, denote by $(x)=\sum_\nu v_\nu(x)\nu$ the associated divisor inside $\Pic(C)$ and define the Riemann--Roch space
    \[
    L(D)=\{x\in K^*\colon (x)+D\geq 0\}\cup\{0\},
    \]
    where the notation $(x)+D\geq 0$ means that the divisor is effective. It is then clear from \eqref{Eq: RRineq} that $L(D)=L(\bm{R})\cap\mathfrak{a}$. The Riemann-Roch space defines an $\FF_q$-vector space and we denote its dimension by $\ell(D)$. 

    Let $\deg\colon \Pic(C)\to \ZZ$ be the degree map, which for the divisor $D$ defined above satisfies $\deg(D)=\sum_{\nu\mid \infty}\log_q(R_\nu) - \log_q(\mathrm{N}(\mathfrak{a}))$. If $\deg(D)\leq 2g-2$, then as the degree map has finite fibers whose cardinality is given by the class number of $K$ and $L(D)$ only depends on the class of $D$ in $\Pic(C)$, there are only $O(1)$ possibilities for $\ell(D)$. We may therefore assume that $\deg(D)>2g-2$ from now on. In this case the Riemann-Roch Theorem in the form given by Rosen~\cite[Corollary 4]{rosen} tells us that $\ell(D)=\deg(D)-g+1$, so that 
    \[
    \#(L(\bm{R})\cap \mathfrak{a})=q^{\ell(D)}=q^{1-g}|\bm{R}|\mathrm{N}(\mathfrak{a})^{-1},
    \]
    which completes the proof.
\end{proof}

Fix $\mathfrak{a}_1,\dots, \mathfrak{a}_h\subset \O_K$ to be a full set of representatives in the class group of $K$ once and for all. Whenever we indicate that an implied constant depends on $K$, then it is implicitly allowed to also depend on our choice of the ideal class group representatives.

Unless $K$ is a principal ideal domain one cannot longer ensure that an element in $\mathbb{P}^n(K)$ has a representative in $\O_K^{n+1}$ such that the coordinate entries generate $\O_K$. Instead in this more general setting the set of primitive vectors is replaced by the set
\[
Z_n' \coloneqq \left\{\bm{x} \in \O_K^{n+1} \setminus \{\bm{0}\} \colon (x_0, \hdots, x_n) = \mathfrak{a}_i \text{ for some } i=1, \hdots, h \right\}.
\]
Note that if $\bm{x} \in Z_n'$ is a representative for an element $x \in \PP^n(K)$ then we have $H(x) \asymp \norm{\bm{x}}_\infty$. Since $\norm{u}_\infty = 1$ for any unit $u \in \O_K$ the set $\{ \bm{x} \in Z_n \colon \norm{\bm{x}}_\infty \leq B\}$ is potentially infinite.
\begin{lemma} \label{lem.good_units}
    Let $\lambda_\nu\in\RR_{>0}$ be given for $\nu\mid \infty$. Then there exists a unit $u\in \O_K^\times$ and $t\in \RR$ such that 
    \[
    \lambda_\nu \asymp_K \frac{|u|_\nu}{t}.
    \]
\end{lemma}
\begin{proof}
    If $K$ is a number field, this is proved in the Lemma on p. 187 of \cite{SerreLecturesMordellWeil}, so that we may assume that $K$ is a function field from now on. Upon defining $\lambda_\nu'=\log_q\lambda_\nu$, the statement of the lemma is equivalent to 
    \[
    |\lambda_\nu'+\log_q(t)-\log_q(u)|\leq C,
    \]
    for some fixed constant $C=C(K)>0$. If we define 
    \[
    \phi\colon \O_K^\times \to \RR^{s_K}, \quad u\mapsto (\log_q|u|_\nu)_{\nu\mid \infty},
    \]
    then $\Gamma=\phi(\O_K^\times)\oplus \bm{1}\ZZ$ is a lattice of rank $s_K$, where $\bm{1}=(1,\dots, ,1)$. Thus the quotient $\RR^{s_K}/\Gamma$ is compact and we can find a unit $u\in \O_K^\times$ and $t'\in \ZZ$ such that 
    \[
    |\lambda_\nu'+t'-\log_q(u)|\leq C,
    \]
    where $C$ only depends on $K$. The statement now follows upon setting $t=q^{t'}$. 
\end{proof}

\begin{lemma}\label{Le: GoodReps}
    There exists constants $c_1,c_2>0$ depending on $K$ such that every member of $\PP^n(K)$ has a representative $\bm{x}\in Z_n'$ such that $c_1\norm{\bm{x}} \leq \norm{\bm{x}}_\infty^{1/s_K}\leq c_2 \norm{\bm{x}}$. 
\end{lemma}
\begin{proof}
    Let $\bm{y}\in \PP^n(K)$ and suppose $\bm{x}'=(x_0',\dots, x'_n)\in Z_n'$ is a representative for $\bm{y}$. Define $\lambda_\nu = \max\{|x'_0|_\nu,\dots, |x'_n|_\nu\}$. Then by Lemma \ref{lem.good_units} we can find a unit $u\in \O_K^\times$ and $t\in \RR$ such that $\lambda_\nu\asymp_K |u|_\nu/t$ for all $\nu\mid \infty$. Taking the product over all infinite places, we deduce that
    \[
    1\asymp\prod_{\nu \mid \infty}t\max\{|x'_0|_\nu,\dots, |x'_n|_\nu\}=t^{s_K}\norm{\bm{x}'}_\infty,
    \]
    which implies that $t^{-1}\asymp\norm{\bm{x}'}_\infty^{1/s_K}$. Moreover, if we define $\bm{x}=u^{-1}\bm{x}'$, then we have 
    \[
    \max\{|x_0|_\nu,\dots, |x_n|_\nu\} \asymp t^{-1}\asymp\norm{\bm{x}}^{1/s_K}_\infty,
    \]
    since $\norm{\bm{x}}_\infty=\norm{\bm{x}'}_\infty$. As $\bm{x}$ is also a representative for $\bm{y}$ in $\PP^n(K)$, the result follows.
\end{proof}
Let $c_1,c_2$ be the constants from Lemma~\ref{Le: GoodReps}. We then define
\[
Z_n\coloneqq \{\bm{x}\in Z_n'\colon c_1\norm{\bm{x}}\leq \norm{\bm{x}}_\infty^{1/s_K}\leq c_2 \norm{\bm{x}}\},
\]
so that in particular every element in $\PP^n(K)$ has a representative in $Z_n$. 

For $x\in K$, we define the affine height 
\[
h(x)\coloneqq H(1,x). 
\]
\begin{lemma}\label{Le: NumberUnits}
    Let $K$ be a global field. Then 
    \[
    \{u\in \O_K^\times\colon h(u)\leq B\} \ll_{d_K} (\log B)^{s_K}.
    \]
\end{lemma}
\begin{proof}
   If $K$ is a number field this is proved by Broberg \cite[Proposition 4]{broberg2001rational} and so we may assume that $K$ is a function field. Let $\FF_q$ be the field of constants of $K$. If we define
   \[
   \phi\colon \O_K^\times\to \ZZ^{s_K},\quad u\mapsto (\log_q|u|_\nu)_{\nu\mid \infty},
   \]
   then $\phi(\O_K^\times)$ is a lattice of rank $s_K-1$ inside $\ZZ^{s_K}$ and $\phi$ is a group homorphism with kernel $\FF_q^\times$. Note that for any $u\in \O_K^\times$ we have $h(u)=h(u^{-1})$, so that if $h(u)\leq B$ holds, then we must have $\max_{\nu\mid \infty}|\log_q|u|_\nu|\leq \log_q B$. It follows that the number of units in question is $O((\log B)^{s_K})$ as desired.  
   \end{proof}

\begin{lemma}\label{Le: Divisor.bound.elements}
    Let $L/K$ be an extension of degree $d$ and let $y\in \O_K\setminus\{0\}$. Then
    \[
    \#\{(y_1,y_2) \in \O_L^2\colon  y_1y_2=y \text{ and }|y_i|_\nu \leq R \text{ for all }\nu\in\Omega_{K,\infty}, i=1,2\}\ll_{d, K}(R\mathrm{N}(y))^\varepsilon,
    \]
    where $|\cdot|_\nu$ is extended uniquely to $L$.
\end{lemma}
\begin{proof}
    Note that by the usual divisor bound, there are $O(\mathrm{N}(y)^\varepsilon)$ ideals in $\O_K$ that divide $(y)$. Moreover, if $\mathfrak{p}$ is a prime ideal of $\O_K$ and we have a factorisation $\mathfrak{p}=\mathfrak{p}_1^{e_1}\cdots \mathfrak{p}_r^{e_1}$ into prime ideals of $\O_L$, then we must have $\sum_{i=1}^r e_i\leq d$. It follows that there $O(\mathrm{N}(y)^{\varepsilon O_d(1)})=O_d(\mathrm{N}(y)^\varepsilon)$ divisors of the ideal $(y)$ in $\O_L$. 

    Now suppose that $y_i$ and $z_i$ generate the same ideal and $y_1y_2=z_1z_2$. This implies that $z_1=uy_1$ and $z_2=u^{-1}z_2$ for some unit $u\in \O_L$. Thus if $|z_i|_\nu \leq R$, we have $|u|_\nu \leq R/ |y_i|_\nu$ for $i=1,2$. Moreover, if $\omega$ is a place of $L$ lying above $\nu$, then $|u|_\omega=|u|_\nu^{d_\omega}$, where $d_\omega$ is the degree of the extension $K_\omega/K_\nu$. In particular,
    \[
    h(u)=\prod_{\omega\in \Omega_{L,\infty}}\max\{1, |u|_\omega\}\ll \prod_{\omega\in \Omega_{L,\infty}} R^{d_\omega}|y_i|_{\omega}^{-1} \ll R^{d}N_L(y_i)^{-1} \leq R^d. 
    \]
If $\cha(K)>0$ and $\FF_q$ is the field of constants of $K$, then the field of constants of $L$ is an extension of $\FF_q$ of degree at most $d$. Therefore, $s_L\leq d s_K$ and Lemma~\ref{Le: NumberUnits} implies that there are $O_d((\log R)^{d s_K})=O_d(R^\varepsilon)$ available $u$, which completes the proof.  
\end{proof}

   Suppose we are given a morphism $\phi\colon \PP^n\to \PP^m$ of the form $\phi(\bm{x})=(\phi_0(\bm{x}),\dots, \phi_m(\bm{x}))$ where $\phi_0,\dots, \phi_m\in \O_K[x_0,\dots, x_n]$ are homogeneous forms of degree $e$ without a common zero in $\overline{K}$. Then functoriality of heights implies that 
\begin{equation*}
H(\phi(\bm{x}))\asymp H(\bm{x})^e,
\end{equation*}
where the implied constant depends on $n, m, K$ and $\phi$. 

Given $f\in K[x_1,\dots,x_n]$ we define $\norm{f}=\norm{\bm{f}}$, where $\bm{f}$ is the coefficient vector of $f$.  Similarly, we can extend it to vectors $F=(f_1,\dots, f_r)\in K[x_1,\dots, x_n]^r$ by setting $\norm{F}=\max \norm{f_i}$. We require a version of the functoriality of heights for morphisms $\PP^1\to  \PP^n$ with an explicit dependence on the height of the morphism.
\begin{lemma}\label{Le: FunctorialityHeights}
    Let $\psi\colon \PP^1\to \PP^n$ be a morphism over $K$ given by 
    \[
    \psi([u,v])= [\psi_0(u,v), \hdots , \psi_n(u,v)],
    \]
    where $\psi_i\in\O_K[u,v]$ is homogeneous of degree $d$ for $i=0,\dots, n$ and the $\psi_i$ do not share a common non-constant factor. Then 
    \[
    \norm{\psi}^{-B}H([u,  v])^d\ll_{K, n, d} H(\psi([u,v]))\ll_{K, n, d}\norm{\psi}^BH([u,v])^d,
    \]
    where the implied constants and $B$ only depend on $K, n$ and $d$.  
\end{lemma}
\begin{proof}
    The upper bound is easy: For any $(u,v)\in Z_1$, the triangle inequality implies 
    \[
    H(\psi([u, v])) =\mathrm{N}(\psi(u,v))^{-1}\prod_{\nu \mid \infty}\max\{|\psi_i(u,v)|_\nu\}\ll \prod_{\nu\mid \infty}\max\{|\psi_{ij}|_\nu\}\max\{|u|_\nu, |v|_\nu\}^d,
    \]
    from which the claim follows. 

    For the lower bound, a repeated application of the Euclidean algorithm gives 
    \[
    Cu^e=\sum_{i=0}^nf_i\psi_i(u,v) \quad\text{and}\quad Cv^e = \sum_{i=0}^n f_i'\psi_i(u,v)
    \]
    for some homogeneous forms $f_i,f_i'\in\O_K[u,v]$ of degree $e-d$ and $C\in \O_K$. Moreover, an inspection of the algorithm reveals that $e$ only depends on $n$; that the $f_i$ and $f_i'$ can be taken to satisfy $\norm{f_i}, \norm{f_i'}=O(\norm{\psi}^B)$ and $\norm{C}=O(\norm{\psi}^{B'})$ for some absolute constants $B, B'>0$. It follows that if $\mathfrak{a}\mid (\psi(u,v))$ for $(u,v)\in Z_1$, then $\mathfrak{a}\mid (C)(\mathfrak{a}_1\cdots \mathfrak{a}_h)^e$, where $\mathfrak{a}_1,\dots ,\mathfrak{a}_h$ are a set of representatives for $CL_K$. Therefore, for any $\nu\mid \infty$,
    \begin{align*}
        |C|_\nu \max\{|u|_\nu^e,|v|_\nu^e\} &=\max\left\{\left|\sum f_i \psi(u,v)\right|_\nu, \left|\sum f_i'\psi_i(u,v)\right|_\nu\right\} \\
        &\ll \norm{\psi}^B\max\{|u|^{e-d}_\nu, |v|^{e-d}_\nu\}\max\{|\psi_i(u,v)|\}, 
    \end{align*}
which implies 
\begin{align*}
    H(\psi(u,v)) &=\mathrm{N}(\psi(u,v))^{-1}\prod_{\nu \mid \infty}\max\{|\psi_i(u,v)|_\nu\}\\
    &\gg \mathrm{N}(C)^{-1}\norm{\psi}^{-B}H([u,v])^d\\
    &\gg \norm{\psi}^{-B''}H([u,v])^d
\end{align*}
for some constant $B''>0$.
\end{proof}

\subsection{Rational points on varieties}
For $\bm{c}\in\PP^n(K)$, we shall write $H_{\bm{c}}\subset \PP^n$ for the hyperplane defined by $\bm{c}\cdot \bm{x}=0$. We then have the following result, which follows for number fields from work of Bombieri and Vaaler \cite{bombieri_vaaler} and for function fields from work of Thunder \cite{ThunderSiegel}.
\begin{lemma}[Siegel's Lemma] \label{lem.siegel}
Let $K$ be a global field  and $\bm{x}\in \PP^n(K)$ with $H(\bm{x})\leq R$. Then there exists $\bm{c}\in \PP^n(K)$ with $H(\bm{c}) \ll_{n,k} R^{1/n}$ and $\bm{x}\in H_{\bm{c}}$.
\end{lemma}
The next result \cite[Theorem 1.8]{paredes2021uniform} will be helpful when dealing with singular hyperplane sections. 
\begin{lemma}\label{lem.irreducible_curve_uniform}
Let $C\subset \PP^n$ be an irreducible curve of degree $e\geq 1$. Then 
\[
\#\{\bm{x}\in C(K) \colon H(\bm{x})\leq  B\}\ll_{e,n,K} B^{2/e}.
\]
\end{lemma}
We will need strong upper bounds for the number of rational points on irreducible varieties of large degree to control the number of degenerate hyperplane sections. Let $X\subset \PP^n$ be an irreducible variety of degree $d$. Then the estimate 
\begin{equation}\label{Eq:UnifVariety}
\#\{\bm{x}\in X(K)\colon \PP^n(K)\}\ll_{d,n,K}B^{\dim(X)+1}
\end{equation}
is proved via a simple inductive process of taking hyperplane sections, see for example Lemma~5.6 of \cite{BrowningBookCircle} for the case $K=\QQ$. We omit a proof here. Using the determinant method, one can do significantly better. The principal input is the following result \cite[Theorem 1.11]{paredes2021uniform} building on work of Salberger \cite{TheSalberger} and Heath-Brown \cite{HeathBrownCurvesSurfaces}, that demonstrates the truth of the dimension growth conjecture for all global fields. 
\begin{prop}\label{Prop.DimGrowth}
Let $X\subset \PP^n$ be an integral variety of degree $d\geq 5$. Then 
\[
\#\{\bm{x}\in X(K)\colon H(\bm{x})\leq B\} \ll_{d, n,K}B^{\dim(X)}.
\]
\end{prop}
We shall also make us of the following result about representations by binary quadratic forms.
\begin{lemma}\label{Le: QuadRepI}
    Let $Q\in \O_K[s,t]$ be a binary quadratic form without repeated roots and let $\bm{R}\in \RR_{>1}^{s_K}$. If $\cha(K)\neq 2$, then for any $\gamma \in \O_K$ we have
    \[
    r_Q(\gamma,\bm{R})\coloneqq \#\{(s,t)\in \O_K^2\colon |(s,t)|_\nu\leq R_v\text{ for all }\nu\mid\infty, Q(s,t)=\gamma\}\ll_K (|\bm{R}|\norm{Q}\mathrm{N}(\gamma))^\varepsilon.
    \]
\end{lemma}
\begin{proof}
As $\cha(K)\neq 2$, we can always find $P\in \GL_2(K)$ such that $Q(P(s,t))$ is diagonal. Moreover, it is clear that $P$ can be chosen in such a way that its entries $p_{ij}$ satisfy ${\norm{p_{ij}}\ll \norm{Q}^B}$. Setting $(x,y)=P(s,t)$, the equation then becomes $ax^2+by^2=\gamma$, where $\norm{(a,b)}\ll \norm{Q}^2$ and $|(x,y)|_\nu \ll \norm{Q}^B R_\nu$ for all $\nu\mid\infty$. Multiplying both sides by a suitable element in $\O_K$, we may assume that $a$ is a square in $\O_K$ and $a,b \in \O_K$, so that after applying another change of variables, it transpires that 
\[
r_Q(k,\bm{R})\leq \#\{(x,y)\in \O_K^2\colon |(x,y)|_\nu\ll \norm{Q}^BR_\nu, \, x^2+dy^2=\gamma'\},
\]
where $\norm{d}\ll \norm{Q}^B$ and $|\gamma'|_\nu\ll \norm{Q}^B|\gamma|_\nu$ for all $\nu\mid\infty$. Let $L=K(\sqrt{d})$ and note that in $L$ we have the factorisation of integral ideals 
\[
(x-\sqrt{d}y)(x+\sqrt{d}y)=(\gamma').
\]
Let $z_1=x-\sqrt{d}y$ and $z_2=x+\sqrt{d}y$. Since $Q$ is square-free, the assignment $(x,y)\mapsto (z_1,z_2)$ is injective. Moreover, if we also denote by $|\cdot |_\nu$ the extension of $|\cdot |_\nu$ to $L$, then we have 
\[
|z_i|_\nu \ll |\sqrt{d}|_\nu \max\{|x|_\nu, |y|_\nu\}\ll \norm{Q}^B R_\nu \leq \norm{Q}^B |\bm{R}|
\]
for $i=1,2$. Therefore, Lemma~\ref{Le: Divisor.bound.elements} implies that the number of available $(z_1,z_2)\in \O_L$ is $O((\norm{Q}^B |\bm{R}|\mathrm{N}(\gamma'))^\varepsilon)=O((\norm{Q}|\bm{R}|\mathrm{N}(\gamma))^\varepsilon))$ as desired.
\end{proof}
\begin{corollary}\label{Cor: QuadRepII}
    Let $Q\in\O_K[s,t]$ be a square-free binary quadratic form. Suppose we are given $\bm{R}\in \RR_{>1}^{s_K}$ and $S>0$. If $\mathfrak{a}\subset \O_K$ is an ideal, then 
    \[
    \#\{(s,t)\in Z_1\colon |(s,t)|_\nu\leq R_\nu, \norm{Q(s,t)}_\infty \leq S, Q(s,t)\in \mathfrak{a}\}\ll_K \max\left\{1,\frac{|\bm{R}|}{\mathrm{N}(\mathfrak{a})}\right\}(|\bm{R}|S\norm{Q})^{\varepsilon}.
    \]
\end{corollary}
\begin{proof}
There are $O(1)$ possible $(s,t)\in Z_1$ for which $Q(s,t)=0$, and so it suffices to bound the contribution from those with $Q(s,t)\neq 0$. Fix a place $\omega \mid \infty$ of $K$ and let us split $|Q(s,t)|_\nu$ into dyadic intervals for $\nu\neq \omega$, say $|Q(s,t)|_\nu \sim w_\nu$. The quantity we want to estimate is then at most
\[
\sum_{w_\nu}\#\Big\{(s,t)\in\O_K^2\colon |(s,t)|\leq R_\nu, Q(s,t)\in \mathfrak{a}, |Q(s,t)|_\omega \ll S\prod_{\substack{\nu \mid \infty \\ \nu\neq \omega}} w_\nu^{-1}, |Q(s,t)|_\nu \leq w_\nu\Big\}.
\]
In the notation of Lemma~\ref{Le: QuadRepI}, each term in the sum is at most 
\begin{align*}
\sum_{\substack{\gamma\in \mathfrak{a}}}r_Q(\gamma, \bm{R})&\ll \sum_{\substack{\gamma\in \mathfrak{a}}}(\norm{Q}\mathrm{N}(\gamma)|\bm{R}|  )^\varepsilon\\
&\ll (\norm{Q}S|\bm{R}| )^\varepsilon\max\left\{1, \frac{S}{\mathrm{N}(\mathfrak{a})}\right\}
\end{align*}
where the sum over $\gamma$ runs over all $\gamma\in \mathfrak{a}$ such that $|\gamma|_\nu \leq {w_\nu}$ for $\nu\neq \omega$ and $|\gamma|_\omega \ll S\prod w_\nu^{-1}$ and we successively applied Lemmas~\ref{Le: QuadRepI} and \ref{Le: Number.OK.Points}. We clearly have $|Q(s,t)|_\nu \ll \norm{Q}R_\nu^2$ and hence also $|Q(s,t)|_\nu \gg \norm{Q}^{-s_K-1}|\bm{R}|^{-2(s_K+1)}$, so that there are $O((\norm{Q}|\bm{R}|)^\varepsilon)$ possibilities for $w_\nu$, which gives the result.
\end{proof}
We also need the following easy generalisation of a result due to Broberg \cite[Lemma 9]{broberg2001rational}.
\begin{lemma}\label{Le: Polysmall}
    Let $G\in K[x_1,\dots, x_n]$ be a polynomial of degree $d$ and $R, S\geq 1$. Then 
    \[
    M(G,R,S)\coloneqq \#\{\bm{x}\in \O_K^n\colon \norm{\bm{x}}\leq R^{1/s_K}, \mathrm{N}(G(\bm{x}))\leq S\}\ll_G R^{n-1+\varepsilon}S^{1/d}.
    \]
\end{lemma}
\begin{proof}
    We first give a proof for the case $n=1$, so that $G\in K[x]$. Upon dividing through by the leading coefficient, we may assume that $G$ is monic. Over an algebraic closure $\bar{K}$ of $K$ we then have the factorisation
    \[
    G(x)=\prod_{i=1}^d(x-a_i),
    \]
    for some $a_i\in \bar{K}$. As $L=K(a_1,\dots, a_n)$ is a finite degree extension of $K$, every place $\nu\in \Omega_{K,\infty}$ extends uniquely to $L$ and by abuse of notation we shall also denote it by $\nu$. We then have 
    \[
    \mathrm{N}(G(x))=\prod_{i=1}^d \prod_{\nu\mid \infty}|x-a_i|_\nu.
    \]
    In particular, if $\mathrm{N}(G(x))\leq S$, then we must have 
    \[
    \prod_{\nu\mid\infty}|x-a_i|_\nu\leq S^{1/d}
    \]
    for some $1\leq i \leq d$. Note that $\norm{G}+R^{1/s_K-d}\ll|x-a_i|_\nu\ll \norm{G}+R^{1/s_K} $, where the upper bound follows from $|x|_\nu\leq R^{1/s_K}$ and the lower bound is a consequence of the upper bound and $1\leq \mathrm{N}(G(x))$. Let us now fix a place $\omega\mid \infty$ of $K$ and put $|x-a_i|_\nu$ into dyadic intervals for $\nu\neq \omega$. We thus have that $M(G,R,S)$ is bounded above by
    \begin{align*}
     \sum_{i=1}^d \sum_{w_\nu} \#\Big\{x\in\O_K\colon |x|_\nu\leq R^{1/s_K}, |x-a_i|_\omega \ll S^{1/d}\prod_{\nu\neq\omega}w_\nu^{-1}, |x-a_i|_\nu \leq {w_\nu} \text{ for all }\nu\neq \omega \Big\},    
    \end{align*}
        where the sum over $w_\nu$ is over all tuples of  dyadic powers $(w_\nu)_{\nu\neq \omega}$ such that $|w_\nu|\ll (R+\norm{G})$ and $1\leq \prod_\nu w_\nu \leq S^{1/d}$. If $x,x'\in \O_K$ satisfy $|x-a_i|_\nu \leq {w_\nu}$ and $|x'-a_i|_\nu \leq {w_\nu}$, then $|x-x'|_\nu\ll {w_\nu}$. From Lemma~\ref{Le: Number.OK.Points} we thus obtain 
    \[
    M(G,R,S) \ll \sum_{w_\nu}S^{1/d}\ll (R\norm{G})^\varepsilon S^{1/d},
    \]
    which is satisfactory. 

    Next we assume that $n>1$. For  $\bm{a}=(a_2,\dots, a_n)\in \O_K^{n-1}$, define $\bm{x}'=(x_1, x_2+a_2x_1,\dots, x_n+a_nx_1)$. Let $G_d$ be the homogeneous degree $d$ part of $G$. We now claim that we can take $\bm{a}\in \O_{K}^{n-1}$ such that $\deg_{x_1'}(G(\bm{x}'))=d$. Indeed, the assertion is equivalent to $G_d(1,a_2,\dots, a_n)\neq 0$ and since $G_d$ is not the zero polynomial, we can always find such an $\bm{a}$ with $\norm{\bm{a}}\ll 1$. The new variables $\bm{x}'$ now satisfy $|\bm{x}'|_\nu \ll R^{1/s_K}$, where the implied constant depends on $\bm{a}$. Once $x_2',\dots, x_n'$ are fixed, then by construction $g(x_1')=G(x_1',\dots,x_n')$ is a degree $d$ polynomial and since $\norm{x'_i}\ll R^{1/s_K}$ for $i=2,\dots, n$, we have $\norm{g}\ll R^{d/s_K}$. In particular, from the case $n=1$ and Lemma~\ref{Le: Number.OK.Points}, it follows that 
    \[
    M(G,R,S)\ll \sum_{\substack{x'_2,\dots, x'_n\in \O_K\\ |x'_i|_\nu\ll R^{1/s_K}}}\#\{x'_1\in \O_K\colon \mathrm{N}(g(x'_1))\leq S, |x'_1|_\nu \ll R^{1/s_K}\}\ll R^{n-1+\varepsilon}S^{1/d}.
    \]
\end{proof}
For separable binary forms we can prove a very similar result, which we will require.
\begin{lemma} \label{lem.binary_form_separable_bound}
    Let $f \in K[s,t]$ be a separable binary form of degree $d$ and let $R, S \geq 1$. Then we have
    \[
    M(f,R,S) = \#\left\{ (s,t) \in Z_1 \colon \norm{(s,t)}_\infty \asymp R, \, \mathrm{N}(f(s,t)) \leq S \right\} \ll_f R^{1+\varepsilon}\left(1+\frac{S}{R^{d-1}}\right)
    \]
\end{lemma}
\begin{proof}
First note that there are $O(1)$ many primitive solutions to $f(s,t) = 0$.
 For the remainder we will therefore proceed to only count the contribution such that $f(s,t) \neq 0$.
    After rescaling $f$ by a constant we may assume that it is of the form
    \[
    f = \prod_{i=1}^d (s-a_i t),
    \]
    where $a_i \in \overline{K}$. Since $K(a_1, \hdots, a_d)$ is a finite algebraic extension of $K$ we may extend each place $\nu \mid \infty$. By an abuse of notation we will denote the extended place by $\nu$. The elements $a_i$ are pairwise different since $f$ was assumed to be separable. Therefore, if we take a constant $c >0$ sufficiently small in terms of the polynomial $f$, then given $\nu \mid \infty$ there is at most one index $i_\nu \in \{1, \hdots, d\}$ such that
    \[
    |s-a_{i_\nu}t|_\nu < c \max\{|s|_\nu,|t|_\nu\}.
    \]
    Therefore we have
    \[
    \mathrm{N}(f(s,t)) = \prod_{\nu \mid \infty} \prod_{i=1}^d |s-a_i t|_\nu \asymp R^{d-1} \prod_{\nu \mid \infty} |s-a_{i_\nu} t|_\nu.
    \]
    Fix a place $\omega \mid \infty$ and divide the absolute values $|s-a_{i_\nu} t|_\nu$ for $\nu \neq \omega$ into dyadic intervals $(\gamma_\nu,2\gamma_\nu]$, say. Then since $\mathrm{N}(f(s,t)) \leq S$ we have that
    \[
    |s-a_{i_\omega}t|_\omega \ll \frac{S}{R^{d-1}} \prod_{\nu \neq \omega} \gamma_\nu^{-1}.
    \]
    For fixed $t$ these restrictions imply that there are at most $O(1+ S/R^{d-1})$ many $s$ available. Since $\norm{t} \ll R$, from Lemma~\ref{Le: Number.OK.Points} we see that each such dyadic decomposition contributes at most $O(R(1+S/R^{d-1}))$. Clearly we must only consider dyadic intervals such that $\gamma_\nu \ll R^{1/s_K}$. Since $\mathrm{N}(f(s,t)) \geq 1$ whenever $(s,t) \in \O_K^2$ is not a zero of $f$ we have that
    \[
    \gamma_\nu \gg R^{-d-s_K}.
    \]
    It follows that there are $O(R^\varepsilon)$ relevant dyadic intervals to consider, which completes the proof of the lemma.
\end{proof}

\section{Geometry of numbers and rational points on conics}
In this section we will first recall basic properties of $\O_K$-lattices and establish analogues of classical results from the geometry of numbers. In the second part of this section, we use them to obtain uniform upper bounds for the number of rational points on conics.
\subsection{Lattices} \label{sec.lattices}
Let $R$ be a Dedekind domain with field of fractions $K$. We call a finitely generated torsion-free module $\Lambda$ an \emph{$R$-lattice} and denote the corresponding vector space $K\Lambda$ by $V$. The dimension of $V$ is called the \emph{rank} of $\Lambda$. Suppose that $\Lambda'\subset V$ is another $R$-lattice of the same rank $r$ as $\Lambda$. By the invariant factor theorem \cite[Theorem 4.14]{maximalorders}, there exist elements $e_1,\dots, e_r\in \Lambda$ and unique fractional ideals $\mathfrak{b}_1,\dots ,\mathfrak{b}_r, \mathfrak{c}_1,\dots, \mathfrak{c}_r \subset K$ such that $\mathfrak{c}_1\subset \cdots \subset\mathfrak{c}_r$ and
\begin{equation*}
    \Lambda = \bigoplus_{i=1}^r\mathfrak{b}_ie_i\quad\text{and}\quad \Lambda'=\bigoplus_{i=1}^r \mathfrak{c}_i\mathfrak{b}_ie_i.
\end{equation*}
Moreover, $\Lambda'\subset \Lambda$ if and only if $\mathfrak{c}_i\subset R$ for $i=1,\dots , r$. We then define the \emph{index ideal} of $\Lambda'$ in $\Lambda$ to be the fractional ideal $(\Lambda\colon\Lambda')=\mathfrak{c}_1\cdots \mathfrak{c}_r$. Note that if $\Lambda'\subset \Lambda$, then $\mathrm{N}((\Lambda\colon\Lambda'))=[\Lambda\colon \Lambda']$, where the right hand side denotes the ordinary index of abelian groups. In addition, if $a\in (\Lambda\colon \Lambda')$, then we have $\Lambda\subset a^{-1}\Lambda'$.

Let us now specialise to the case where $R=\O_K$ is the ring of integers of a global field. Suppose we are given an $\O_K$-lattice $\Lambda$. We can then form the $\O_\p$-lattice $\Lambda_\p\coloneqq \Lambda\otimes \O_\p$ for any finite prime $\p$. If $\Lambda'\subset V$ is another lattice of full rank, then we have $\Lambda_\p=\Lambda'_\p$ for almost all $\p$ and so it is clear that we have $(\Lambda\colon \Lambda')=\bigcap (\Lambda_\p\colon \Lambda'_\p)$, where the intersection takes place in $\O_K$. If we are given sublattices $L_\p\subset V_\p=K_\p V$ for all finite primes $\p$ such that $L_\p=\Lambda_\p$ for all but finitely many $\p$, then $\Lambda'\coloneqq \bigcap L_\p$ defines an $\O_K$-lattice in $V$ such that $\Lambda'_\p=L_\p$ for all $\p$. It follows from the local description of the index ideal that if $A\in \GL(V)$, we have $(\Lambda\colon A\Lambda')=(\det(A))(\Lambda\colon \Lambda')$. For details we refer the reader to Sections 4 and 5 of Reiner~\cite{maximalorders}.

If $V=K^n$ and $\Lambda\subset V$ is an $\O_K$-lattice of rank $n$, then the determinant $\det(\Lambda)$ is defined to be $\mathrm{N}((\O_K^n\colon \Lambda))$.
For $\nu \mid \infty$ let $S_\nu$ be an open, symmetric, convex and bounded subset of $K_\nu^n$ such that for all complex places $\nu$ we have $S_\nu = \alpha S_\nu$ whenever $\alpha \in k_\nu$ and $|\alpha|_\nu = 1$. We shall consider
\[
S = \prod_{\nu \mid \infty} S_\nu.
\]
We may realize $\Lambda$ via the diagonal embedding inside $\prod_{\nu \mid \infty}K_\nu^n$. Let $k_\infty$ be the completion of $k$ at the infinite place and $|\cdot|_\infty$ be the associated absolute value on $k_\infty$. That is, if $k=\QQ$ then $k_\infty=\RR$ and $|\cdot|_\infty$ is the usual absolute value, whereas if $k=\FF_q(t)$, then $k_\infty=\FF_q((t^{-1}))$ and $|\cdot|_\infty$ corresponds to the absolute value induced by the degree on the $\FF_q(t)$. We define the $i$-th successive minimum of $\Lambda$ with respect to $S$ to be
\[
\lambda_{i,S} \coloneqq \inf \{ |\lambda|_\infty\colon \lambda \in k_\infty\text{ and } \Lambda \cap \lambda S \text{ contains $i$ linearly independent vectors} \}.
\]
It is clear that we have
\[
\lambda_{1,S} \leq \lambda_{2,S} \leq \cdots \leq \lambda_{n,S}.
\]
Frequently we shall take 
\begin{equation} \label{eq.standard_S_nu}
S_\nu = \{ \bm{x} \in K_\nu^n \colon |x_i|_\nu < 1 \text{ for } i = 1, \hdots, n\},
\end{equation}
in which case we will denote the successive minima with respect to $S$ simply by $\lambda_i$. Recall that $K_\nu$ is a locally compact abelian group, so that in particular it can be endowed with a Haar measure $\dd x_\nu$ that we normalise in such a way that $\dd x_\nu$ is the usual Lebesgue measure when $\nu$ is real, 2 times the Lebesgue measure when $\nu$ is complex and $\int_{\O_\nu}\dd x_\nu=|\mathfrak{D}_\nu|_\nu^{1/2}$, where $\mathfrak{D}_\nu$ is the local different of the non-archimidean place $\nu$. We can extend the Haar measure to $K_\nu^n$ by $\dd\bm{x}_\nu=\dd x_{1,\nu}\cdots \dd x_{n,\nu}$. If $S_\nu\subset K_\nu^n$ is measurable, we shall write $\vol(S_\nu)=\int_{S_\nu}\dd \bm{x}_\nu$.
We have the following version of Minkwoski's second theorem over global fields.
\begin{lemma} \label{lem.minkowskis_second_theorem}
    Let $\Lambda$ and $S$ be as above and let $\lambda_{1,S} \leq \cdots \leq \lambda_{n,S}$ the successive minima of $\Lambda$ with respect to $S$. Then we have
    \[
    \det(\Lambda) \asymp_{K, n} \prod_{\nu \mid \infty} \vol(S_\nu) (\lambda_{1,S} \cdots \lambda_{n,S})^{d_K}.
    \]
\end{lemma}
This was proven independently by Bombieri--Vaaler~\cite[Theorems 3 and 6]{bombieri_vaaler} and McFeat~\cite{mcfeat} for number fields and by McFeat~\cite{mcfeat} in the case of function fields. The formulation in the respective works looks slightly different, but one can arrive at our presentation of the result by the same considerations as in the proof of the Corollary to Lemma 5 by Broberg~\cite{broberg2001rational}. We remark that if we choose $S_\nu$ as in~\eqref{eq.standard_S_nu} then Lemma~\ref{lem.minkowskis_second_theorem} states
\begin{equation}\label{Eq: SuccessiveMinima}
\det(\Lambda) \asymp (\lambda_1 \cdots \lambda_n)^{d_K}.
\end{equation}
Note that if $\lambda \in k_\infty$, then $|\lambda|_\infty = |\lambda|_\nu^{d_\nu}$, where $d_\nu$ is the degree of the extension $K_\nu/k_\infty$. Therefore, by the discreteness of $\Lambda$ and the definition of the successive minima, \eqref{Eq: SuccessiveMinima} implies that $\Lambda$ contains an element $\bm{x}$ with $|\bm{x}|_\nu \leq \lambda_1^{d_\nu}$ with equality for at least one $\nu\mid\infty$, so that in particular 
\[
\norm{\bm{x}}_\infty =\prod_{\nu\mid \infty}|\bm{x}|_\nu \leq\prod_{\nu\mid \infty}\lambda_1^{d_\nu}=\lambda_1^{d_K}\ll \det (\Lambda)^{1/n}. 
\]
A useful consequence of this is the following. Let $\mathfrak{a}\subset \O_K$ be an ideal. Then $\mathfrak{a}$ is clearly an $\O_K$-module with $\det (\mathfrak{a})=\mathrm{N}(\mathfrak{a})$ and by \eqref{Eq: SuccessiveMinima} it contains an element $x\in \O_K$ such that $\mathrm{N}(x)=\norm{x}_\infty \ll \mathrm{N}(\mathfrak{a})$. Moreover, because $(x)\subset \mathfrak{a}$, we also have $\mathrm{N}(\mathfrak{a})\leq \mathrm{N}(x)$, so that in fact $\mathrm{N}(\mathfrak{a})\asymp \mathrm{N}(x)$. We will make frequent use of this fact without further comment. 

For the following lemma, we need to introduce some notation. For any place $\nu\mid \infty$, we let $P_\nu$ denote the maximal compact subgroup of $\GL_n(K_\nu)$. When $K_\nu=\RR$, this is just $O_n(\RR)$, when $K_\nu=\CC$ this is $U_n(\CC)$, while for non-archimedean places it is $\GL_n(\O_\nu)$, where $\O_\nu$ is the ring of integers of $K_\nu$. By the Iwasawa decomposition, as presented by Weil \cite[Chapter II, \S 2, Theorem 1]{WeilBasicNT}, for any matrix $M\in \GL_n(K_\nu)$, there exists a matrix $A_\nu\in P_\nu$ such that $A_\nu M$ is upper triangular. Moreover, we have $|A_\nu\bm{x}|_\nu\asymp |\bm{x}|_\nu$ for any $\bm{x}\in K_\nu^n$ and $A_\nu\in P_\nu$.

\begin{lemma}\label{Le: GoodLattice}
    Let $\Lambda \subset K^n$ be a lattice with successive minima $\lambda_1\leq \cdots \leq \lambda_n$. Then there exists a free $\O_K$ lattice $\Lambda'\subset K^n$ with  basis $\bm{b}_1,\dots, \bm{b}_n$ such that 
    \begin{enumerate}[(i)]
        \item $\Lambda\subset \Lambda'$,
        \item $\det (\Lambda) \asymp_{K, n} \det (\Lambda')$,
        \item there exist matrices $A_\nu\in P_\nu$ such that $A_\nu \bm{b}_i= (b_{1,\nu}^{(1)},\dots, b_{i,\nu}^{(i)},0,\dots, 0)$ with $|b_{i,\nu}^{(i)}|_\nu \asymp_{K, n} \lambda_i^{d_\nu}$ for all $1\leq i \leq n$ and $\nu\mid \infty$,\item if $\bm{x}\in \Lambda$ is given by $\bm{x}=\sum_{i=1}^ny_i\bm{b}_i$ with $y_i\in \O_K$, then $|y_i|_\nu \ll_{K,n} |\bm{x}|_\nu \lambda_i^{-d_\nu}$ for all $\nu\mid\infty$. 
    \end{enumerate}
\end{lemma}
\begin{proof}
    By definition of the successive minima, we can find $K$-linearly independent vectors $\bm{c}_1,\dots, \bm{c}_n\in \Lambda$ such that $|\bm{c}_i|_\nu \leq \lambda_i^{d_\nu}$ for all $\nu\mid \infty$ with equality for at least one $\nu$. Let now $L$ be the free $\O_K$-lattice given by 
    \[
    L\coloneqq \bigoplus_{i=1}^n \O_K\bm{c}_i.
    \]
    Note that we clearly have $L\subset \Lambda$. Our next goal is to show that $\det (L)\asymp \det (\Lambda)$.
    
To do so, observe that since $L$ is free, we have $\det( L)=\mathrm{N}(\det(\bm{c}_1, \bm{c}_2,\dots, \bm{c}_n))$, where $(\bm{c}_1,\dots, \bm{c}_n)$ is the $n\times n$ matrix having $\bm{c}_i$ as its $i$th column. For every place $\nu\mid \infty$, choose a matrix $A_\nu$ in the maximal compact subgroup of $\GL_n(K_\nu)$ such that 
\begin{equation}\label{Eq: IwasawaDecomp}
A_\nu\bm{c}_i=(c_{i,\nu}^{(1)},\dots, c_{i,\nu}^{(i)},0,\dots, 0),
\end{equation}
which is always possible by the Iwasawa decomposition. Since $|\det(A_\nu)|_\nu=1$, we then have 
    \begin{equation}\label{Eq: DeterminantIndex}
    \mathrm{N}(\det (\bm{c}_1,\dots, \bm{c}_n))=\prod_{\nu\mid \infty}|c_{1,\nu}^{(1)}\cdots c_{n,\nu}^{(n)}|_\nu.   
    \end{equation}
    As $A_\nu$ is norm preserving, we also have $|c_{i,\nu}^{(i)}|_\nu\ll |\bm{c}_i|_\nu \leq \lambda_i^{d_\nu}$. Since $L\subset \Lambda$, we must have $\det( L) \geq\det (\Lambda)$, whence 
    \[
    \det (\Lambda) \leq \det (L)\ll (\lambda_1\cdots \lambda_n)^{d_K}
    \]
    by \eqref{Eq: DeterminantIndex}. However, by Lemma~\ref{lem.minkowskis_second_theorem} we have $\det (\Lambda) \asymp (\lambda_1\cdots \lambda_n)^{d_K}$ and thus \[
    \det (L)\asymp (\lambda_1\cdots \lambda_n)^{d_K}\asymp \det \Lambda\]
    as claimed.

    We now continue with the construction of the lattice $\Lambda'$ whose existence we want to show. Let us  choose $b\in (\Lambda\colon L)$, so that $\Lambda\subset b^{-1}L$. We have
    \[
    1\asymp \frac{\mathrm{N}(\O_K\colon L)}{\mathrm{N}(\O_K\colon \Lambda)}=\mathrm{N}(\Lambda \colon L),
    \]
so that upon multiplying $b$ with a unit if necessary, we may assume that $|b|_\nu\asymp 1$ for all $\nu\mid\infty$. If we define the lattice $\Lambda'=b^{-1}L$, then by construction we have $\Lambda\subset \Lambda'$. In addition, it holds that 
\[
\det (\Lambda') =\mathrm{N}(\O_K\colon b^{-1}\Lambda)=\frac{\det (\Lambda)}{\mathrm{N}(b)^n}\asymp \det (\Lambda),
\]
as claimed in (ii). It thus remains to verify the properties asserted in (iii) and (iv) of the lemma. If we define $\bm{b}_i=b^{-1}\bm{c}_i$ for $i=1,\dots, n$, then by construction $\Lambda'$ is the free $\O_K$-lattice with basis $\bm{b}_1,\dots, \bm{b}_n$. Let $A_\nu \in P_\nu$ be the matrices in \eqref{Eq: IwasawaDecomp} and define $b_{i,\nu}^{(j)}\in K_\nu$ to be the $j$th entry of $A_\nu\bm{b}_i$. Note that this value is explicitly given by $b^{-1}c_{i,\nu}^{(j)}$. Returning to \eqref{Eq: DeterminantIndex} and recalling that $\det (\Lambda)\asymp \det (L)=\mathrm{N}(\det(\bm{c}_1,\dots,\bm{c}_n))$, we deduce from Lemma~\ref{lem.minkowskis_second_theorem} that
\begin{align*}
    (\lambda_1\cdots \lambda_n)^{d_K}&\asymp \prod_{\nu\mid \infty}|c_{1,\nu}^{(1)}\cdots c_{n,\nu}^{(\nu)}|_\nu \\
    &\ll (\lambda_1\cdots \lambda_{i-1})^{d_K}\lambda_i^{d_K-d_\nu}|c_{i,\nu}^{(i)}|_\nu (\lambda_{i+1}\cdots \lambda_n)^{d_K},
\end{align*}
for any $1\leq i \leq n$ and $\nu\mid \infty$, where we used that $|c_{i,\nu}^{(i)}|_\nu \ll |\bm{c}_i|_\nu \leq \lambda_i^{d_\nu}$ and that $\sum d_\nu=d_K$. It follows that $|c_{i,\nu}^{(i)}|_\nu \asymp \lambda_i^{d_\nu}$. As $|b|_\nu\asymp 1$ for all $\nu\mid\infty$ and $|b_{i,\nu}^{(i)}|_\nu =|b|^{-1}_\nu |c_{i,\nu}^{(i)}|_\nu$, this completes the verification of (iii). 

Turning to (iv), suppose that $\bm{x}\in \Lambda$ is given by $\bm{x}=\sum_{i=1}^ny_i\bm{b}_i$ with $y_i\in \O_K$. Then we have 
\[
    |\bm{x}|_\nu \asymp |A_\nu\bm{x}|_\nu = \left|\begin{pmatrix} b_{1,\nu}^{(1)} & b_{2,\nu}^{(1)}&\cdots &b_{n,\nu}^{(1)}\\
    0 & b_{2,\nu}^{(2)} & \cdots & b_{n,\nu}^{(2)}\\
    \vdots &\ddots &\cdots &\vdots \\
    0&\cdots &0 & b_{n,\nu}^{(n)}\end{pmatrix}\begin{pmatrix} y_1 \\ \vdots \\ \vdots \\ y_n\end{pmatrix}\right|_\nu,
    \]
    which implies $|y_n|_\nu=|(A_\nu \bm{x})_n|_\nu |b_{n,\nu}^{(n)}|_\nu^{-1}$, where $(A_\nu\bm{x})_n$ denotes the $n$th entry of $A_\nu\bm{x}$ and hence 
    \[
    |y_n|_\nu \ll |\bm{x}|_\nu/\lambda_n^{d_\nu}
    \]
    by (iii). Similarly, we get 
    \[
    b_{n-1,\nu}^{(n-1)}y_{n-1}=(A_\nu\bm{x})_{n-1}-y_nb_{n,\nu}^{(n-1)}
    \]
    and thus 
    \[
    |y_{n-1}|_\nu \ll (|\bm{x}|_\nu+|y_nb_{n,\nu}^{(n-1)}|_\nu)/|b_{n-1,\nu}^{(n-1)}|_\nu^{-1}\ll |\bm{x}|_\nu \lambda_{n-1}^{-d_\nu}
    \]
    again by (iii) and using that $|b_{n,\nu}^{(n-1)}|_\nu \ll \lambda_n^{d_\nu}$. Continuing in this fashion completes the verification of (iv).
\end{proof}
Using this, by the same ideas and techniques as in chapter 12 of~\cite{davenport_book} one obtains very good bounds on the number of lattice points within a bounded box. 

\begin{lemma} \label{lem.lattice_in_box}
    Let $\Lambda$ be an $\O_K$-lattice of rank $n$ and for $\bm{R}=(R_\nu)_{\nu\mid\infty}\in \RR_{>0}^{s_K}$, define
    \[
    N(\Lambda,\bm{R})\coloneqq \# \left\{ \bm{x} \in \Lambda  \colon |\bm{x}|_\nu < R_\nu  \text{ for all }\nu\mid\infty \right\}. 
    \]
    Then if $\lambda_1,\dots, \lambda_n$ are the successive minima of $\Lambda$, we have 
    \[
    N(\Lambda,\bm{R})\asymp_{K,n} \prod_{i=1}^n\max\left\{1,\frac{|\bm{R}|}{\lambda_i^{d_K}}\right\}.
    \]
\end{lemma}
\begin{proof}
For the lower bound, let $\bm{a}_1,\dots, \bm{a}_n\in \Lambda$ be such that $|\bm{a}_i|_\nu \leq \lambda_i^{d_\nu}$ for all $1\leq i \leq n$ and $\nu \mid \infty$. Then any $\bm{x}= \sum_{i=1}^n\mu_i \bm{a}_i$ with $\mu_i\in \O_K$ and $|\mu_i|_\nu \ll R_\nu / \lambda_i^{d_\nu}$ will be counted by $N(\Lambda, \bm{R})$ provided the implied constant is sufficiently small with respect to $n$. By Lemma~\ref{Le: Number.OK.Points} we have 
\[
\#\{(\mu_1,\dots, \mu_n)\in \O_K^n\colon |\mu_i|_\nu \ll R_\nu/\lambda_i^{d_\nu}\text{ for all }1\leq i \leq n \text{ and }\nu\mid\infty\}\asymp \prod_{i=1}^n \max\left\{1, \frac{|\bm{R}|}{\lambda_i^{d_K}}\right\},
\]
from which the lower bound follows.

Turning to the upper bound, let $\Lambda'$ be the lattice from Lemma~\ref{Le: GoodLattice}. Then as $\Lambda\subset \Lambda'$ it clearly suffices to prove the claimed upper bound for $N(\Lambda', \bm{R})$ instead of $N(\Lambda,\bm{R})$. Let $\bm{b}_1,\dots, \bm{b}_n$ be the basis of $\Lambda'$ and write $\bm{x}=\sum_{i=1}^ny_i\bm{b}_i$ with $y_i\in \O_K$. Then if $|\bm{x}|_\nu \leq R_\nu$, it follows from (iv) of Lemma~\ref{Le: GoodLattice} that 
\[
|y_i|_\nu \ll |\bm{x}|_\nu\lambda_i^{-d_\nu}.
\]
In particular, if $|\bm{x}|_\nu \leq R_\nu$, then $|y_i|_\nu \ll R_\nu \lambda_i^{-d_\nu}$ and by Lemma~\ref{Le: Number.OK.Points} the number of such $y_i\in \O_K$ is $O(\max\{1, |\bm{R}|\lambda_i^{-d_K}\})$ as desired. 
\end{proof}

\subsection{Rational points on conics}
Suppose we are given $\bm{R}_1,\dots, \bm{R}_n\in \RR_{>0}^{s_K}$ with $\bm{R}_i=(R_{i,\nu})_{\nu\mid\infty}$ and set 
\[
L(\bm{R}_1,\dots, \bm{R}_n)=\{\bm{x}\in \O_K^n\colon |x_i|_\nu \leq R_{i,\nu}\text{ for all }\nu\mid\infty \}.
\]
Let $F\in\O_K[x_1,x_2,x_3]$ be a quadratic form. In this section we are concerned with upper bounds for the number of elements in $\bm{x}\in L(\bm{R}_1,\bm{R}_2,\bm{R}_3)$ such that $F(\bm{x})=0$ which are uniform with respect to $F$. Given such $F$, we let $M\in \text{Mat}_{3\times 3}(K)$ denote the underlying matrix and define $\Delta(F)\subset \O_K$  and $\Delta_0(F)\subset \O_K$ to be the ideal generated by the determinant and $2\times 2$ minors of $M$ respectively. The following results were proved by Browning and Swarbrick-Jones for number fields \cite{browning_S-J} and go back to work of Heath-Brown \cite{HeathBrownCubicConicBundle} and Broberg~\cite{broberg2001rational}, who proved a slightly weaker version in the context of the rational numbers and number fields respectively. 
\begin{prop}\label{Prop: UniformConics}
    Let $F\in \O_K[x_0,x_1,x_2]$ be a non-singular quadratic form and $\bm{R}_1,\bm{R}_2,\bm{R}_3 \in \RR^{s_K}_{>1}$. Then 
    \[
    N_F(\bm{R}_1,\bm{R}_2,\bm{R}_3)\coloneqq \#\{\bm{x}\in \PP^2(K)\cap L(\bm{R}_1,\bm{R}_2,\bm{R}_2)\colon  F(\bm{x})=0 \}\ll_K 1+(|\bm{R}_1||\bm{R}_2||\bm{R}_3|)^{1/3}. 
    \]
\end{prop}
\begin{proof}
    If $K$ is  a number field, this is Theorem~4.7 of \cite{browning_S-J}, so we may assume that $K$ is a function field. Moreover, the proof is almost identical, so we shall be brief. Define $R=|\bm{R}_1||\bm{R}_2||\bm{R}_3|$ and choose prime ideals $\mathfrak{p}_1,\dots, \mathfrak{p}_r \subset \O_K$ such that 
    \[
    cR^{1/3}\leq \mathrm{N}(\p_1)\leq \cdots \leq \mathrm{N}(\p_r) \ll R^{1/3},
    \]
    where $c>0$ and $r$ are constants to be determined in due course. First suppose that $\mathrm{N}(\p_i)\mid \Delta(F)$ for $i=1,\dots ,r$. Then we have $\mathrm{N}(\Delta(F))\gg R^{r/3}$ and hence $\norm{F} \gg \mathrm{N}(\Delta(F))^{1/3s_K}\gg R^{r/9s_K}$. Letting $B=\prod_{\nu\mid\infty}\max\{R_{1,\nu}, R_{2,\nu}, R_{3,\nu}\}$, then it is clear that $B\leq R$ and that any $\bm{x}\in Z_2$ with $|x_i|_\nu\leq R_{i,\nu}$ satisfies $\norm{\bm{x}}_\infty\leq B$. In particular, if we choose $r>108$, it follows from Lemma~\ref{Le:SizeCoeffcs} that $N_F(\bm{R}_1,\bm{R}_2,\bm{R}_3)\ll 1$, which is sufficient. We may therefore assume from now on that there is a prime ideal $\mathfrak{p}$ with $\mathrm{N}(\p)\asymp R^{1/3}$ such that $\p\nmid \Delta(F)$.   
    Let $\FF_\p=\O_\p/\p\O_\p$. As the reduction of $F$ modulo $\p$ is non-singular, the projective conic defined by $F=0$ over $\FF_\p$ has precisely $\mathrm{N}(\p)+1$ points and our goal is to show that any such solution gives rise to at most $2$ projective points. Given that $\mathrm{N}(\p)\ll R^{1/3}$, this will be sufficient to complete the proof.

    Suppose that $\bm{x}_0\in \FF_\p^3\setminus\{\bm{0}\}$ satisfies $F(\bm{x}_0)\equiv 0 \:(\p)$. Then by a Hensel lifting argument we can always find $\bm{x}_1\in \O_\p$ such that $\bm{x}_1\equiv \bm{x}_0 \:(\p)$ and $F(\bm{x}_1)\equiv 0 \:(\p^2)$. Suppose now that $\bm{x}\in Z_2$ is such that $\bm{x}\equiv \lambda \bm{x}_1 \:(\p)$ for some $\lambda \in \O_K$ and $F(\bm{x})=0$.  It follows that there exists $\bm{z}\in \O_K^3$ such that $\bm{x}=\lambda \bm{x}_1+\pi\bm{z}$, where $\pi$ is a uniformizer of $\O_\p$. We then have 
    \[
     \pi \lambda \bm{z}\cdot \nabla F(\bm{x}_1) \equiv 0 \:(\p^2).
    \]
    Moreover, we can assume that $R$ is sufficiently large so that $\p$ is not equal to one of the fixed representatives of the class group of $K$. As $\bm{x}\in Z_2$, we must then have $\lambda \not\in \p$ and hence $\bm{z}\cdot \nabla F(\bm{x}_1)\equiv 0 \:(\p)$. Therefore, $\bm{x}\cdot \nabla F(\bm{x}_1)\equiv 0 \:(\p^2)$ and $\bm{x}$ must belong to the set 
    \[
    \{\bm{x}\in \O^3_\p\colon \bm{x}\equiv \lambda \bm{x}_0 \:(\p)\text{ and }\bm{x}\cdot \nabla F(\bm{x}_1)\equiv 0 \:(\p^2)\},
    \]
    which defines an $\O_\p$-lattice $L_\p$ of determinant $\mathrm{N}(\p)^3$ and rank 3. Let now $\Lambda$ be the unique $\O_K$-lattice such that $\Lambda_\p=L_\p$ and $\Lambda_\mathfrak{q}= \O_\mathfrak{q}$ for all $\mathfrak{q}\neq \p$. Moreover, choose elements $\gamma_1,\gamma_2,\gamma_3\in K$ such that $|\gamma_i|_\nu \asymp \prod_{j=1,2,3, j\neq i}R_{i,\nu}$ and define 
    \[
    L'=\{(\gamma_1x_1,\gamma_2x_2,\gamma_3x_3)\in K^3\colon (x_1,x_2,x_3)\in \Lambda\},
    \]
    which is an $\O_K$-lattice of determinant $\det(L')=|R|^2\mathrm{N}(\p)^3$. Let $\Lambda'$ be the lattice from Lemma~\ref{Le: GoodLattice} containing $L'$ with basis $\bm{b}_1,\bm{b}_2,\bm{b}_3\in K^3$ and successive minima $\lambda_1\leq \lambda_2\leq \lambda_3$. If we write $\bm{x}=y_1\bm{b}_1+y_2\bm{b}_2+y_3\bm{b}_3$, then it follows from (iv) of Lemma~\ref{Le: GoodLattice} that $|y_3|\ll |\bm{x}|_\nu \lambda_3^{-d_\nu}\ll R_{1,\nu}R_{2,\nu}R_{3,\nu}\lambda_3^{-d_\nu}$. Moreover, it follows from \eqref{Eq: SuccessiveMinima} that $\lambda_3^{3d_K}\gg \det (\Lambda') \asymp R^2 \mathrm{N}(\p)^3$ and hence after taking the product over all places, we deduce 
    \[
    \mathrm{N}(y_3)\ll \frac{R^3}{R^{2/3}\mathrm{N}(\p)^{1/3}}\ll 1.
    \]
    In particular, if we choose $c$ from the beginning of the lemma sufficiently large, then we must have $y_3=0$. It follows that $\bm{x}$ lies on the intersection of the quadric with a line, and hence contains at most 2 projective points by B\'ezout's theorem. 
\end{proof}

\begin{lemma}\label{Le: latticeQuadFrom}
    Let $\bm{x}\in \O_K$ be such that $F(\bm{x})=0$. Then if $\cha(K)\neq 2$, $\bm{x}$ must lie in at least one of $O(\tau(\Delta(F)))$ many  lattices $\Gamma$  such that $\det(\Gamma)\gg_K \frac{\mathrm{N}(\Delta(F))}{\mathrm{N}(\Delta_0(F))^{3/2}}$, where $\tau$ is the divisor function on ideals.
\end{lemma}
\begin{proof}
    If $K$ is a number field, this is Corollary~4.6 of \cite{browning_S-J}, so we may assume that $\cha(K)>2$ from now on. Let $\p\subset \O_K$ be a prime such that $\p\mid \Delta(F)$ and $\pi$ a uniformizer of $\O_\p$. As in the proof of \cite[Lemma 4(b)]{broberg2001rational}, we may diagonalize $F$ and assume that it is given by 
    \[
F(\bm{x})=\varepsilon_1\pi^{\alpha_1}x_1^2+\varepsilon_2\pi^{\alpha_2}x_2^2+\varepsilon_3\pi^{\alpha_3}x_3^2,
    \]
    where $\alpha_1\geq \alpha_2\geq\alpha_3\geq 0$ and $\varepsilon_i$ are units.
    In particular, if we set $a_\p=\nu_\p(\Delta(F))$ and $b_\p=\nu_\p(\Delta_0(F))$, then $a_\p=\alpha_1+\alpha_2+\alpha_3$ and $b_\p=\alpha_1+\alpha_2$. We now claim that if $\bm{x}\in \O_\p^3$ satisfies $F(\bm{x})=0$, then there are $\O_\p$-lattices $L_1,\dots, L_M$ with $M\leq \alpha_3+1$ and $\det(L_i)\geq \mathrm{N}(\p)^{(2\alpha_3-\alpha_2-\alpha_1)/2}=\mathrm{N}(\p)^{(2a_\p-3b_\p)/2}$. The statement will then follow upon letting $\Gamma$ to be one of the lattices such that $\Gamma_\p= L_i$ for some $1\leq i \leq M$ and $\p\mid \Delta(F)$. 

    Suppose that $x_i=\pi^{\xi_i}u_i$, where $u_i\in \O_\p^\times$ for $i=1,2$. Then if $F(\bm{x})=0$, we must have 
    \begin{equation}\label{Eq: ProofLatticeQuadForm}
        \varepsilon_1u_1^2\pi^{\alpha_1+2\xi_1}+\varepsilon_2u_2^2\pi^{\alpha_2+2\xi_2}\equiv 0 \:(\pi^{\alpha_3}).
    \end{equation}
    We now consider two cases. First, let us assume that $\alpha_3\leq \min_{i=1,2}\{\alpha_i+2\xi_i\}$. Then if we set 
    \[
    L_1=\{\bm{x}\in\O_\p^3\colon x_i\in (\p\O_\p)^{\max\{0, \lceil \frac{\alpha_3-\alpha_i}{2}\rceil\}}, i=1,2\},
    \]
    it is clear that $L_\p$ is an $\O_\p$-lattice of rank 3. We must have $\bm{x}\in L_1$ and in addition \begin{align*}
        \det(L_1)&\geq \mathrm{N}(\p)^{\max\{0, \lceil \frac{\alpha_3-\alpha_1}{2}\rceil\}+\max\{0, \lceil \frac{\alpha_3-\alpha_2}{2}\rceil\}}\\
        &\geq \mathrm{N}(\p)^{(2\alpha_3-\alpha_1-\alpha_2)/2},
    \end{align*}
    which is satisfactory. 

    Let us now assume that $\alpha_3>\min_{i=1,2}\{\alpha_i+2\xi_i\}$. In this case \eqref{Eq: ProofLatticeQuadForm} implies that $\alpha_1+2\xi_1=\alpha_2+2\xi_2=\eta$, say. Moreover, \eqref{Eq: ProofLatticeQuadForm} also gives $(u_1/u_2)^2\equiv - \varepsilon_2/\varepsilon_1 \:(\pi^{\alpha_3-\eta})$. As $\cha(K)\neq 2$, Hensel's lemma implies that if this equation is solvable, that there are $r_1,r_2 \in (\O_\p/ \p^{\alpha_3-\eta}\O_\p)^{\times}$ such that $u_1\equiv r_iu_2 \:(\p^{\alpha_3-\eta})$. In particular, $\bm{x}$ must satisfy 
    \[
    x_i\equiv 0\: (\p^{\xi_i})\quad\text{and}\quad x_1\pi^{-\xi_1}\equiv r_ix_2\pi^{-\xi_2}\:(\p^{\alpha_3-\eta}).
    \]
    These conditions define an $\O_\p$-lattice of determinant $\mathrm{N}(\p)^{\alpha_3+\xi_1+\xi_2}\geq \mathrm{N}(\p)^{(2\alpha_3-\alpha_1-\alpha_2)/2}$ which is sufficient. Moreover, taking into account the possible values that $\xi_1$ and $\xi_2$ can take, we see that are at most $\alpha_3+1$ possible lattices in total.
\end{proof}

\begin{corollary}\label{Cor: ConicsDisc}
Suppose that $\cha(K)\neq 2$. Let $\Delta(F)$ and $\Delta_0(F)$ be the fractional ideals in $K$ generated by the discriminant and the $2\times 2$ minors of $F$. Then \[
N_F(\bm{R}_1,\bm{R}_2,\bm{R}_3) \ll_\varepsilon \left(1+\frac{|\bm{R}_1||\bm{R}_2||\bm{R}_3| \mathrm{N}(\Delta_0(F))^{3/2}}{\mathrm{N}(\Delta(F))}\right)^{1/3}\tau(\Delta(F)).
\]
\end{corollary}
\begin{proof}
    Let $\Gamma$ be one of the lattices from Lemma~\ref{Le: latticeQuadFrom} with $\det(\Gamma)\gg \mathrm{N}(\Delta(F))/\mathrm{N}(\Delta_0(F))^{3/2}$ and choose $\gamma_1, \gamma_2,\gamma_3\in K$ such that $|\gamma_i|_\nu \asymp \prod_{j=1,2,3, j\neq i}R_{j,\nu}$. Put 
    $L=\{(\gamma_1x_1,\gamma_2x_3,\gamma_3x_3)\in K^3\colon (x_1,x_2,x_3)\in \Gamma\}$, so that $L$ is an $\O_K$-lattice of determinant $R^2\det(\Gamma)$. Moreover, note that if $\bm{x}\in \Gamma$ satisfies $|x_i|_\nu \leq R_{i,\nu}$, then $|\gamma_ix_i|_\nu \ll R_{1,\nu}R_{2,\nu}R_{3,\nu}$. 

    Let $\Gamma'$ be the lattice from Lemma~\ref{Le: GoodLattice} such that $L\subset \Gamma'$ with basis $\bm{b}_1,\bm{b}_2,\bm{b}_3$. Suppose that $\bm{x}\in \Gamma$ satisfies $F(\bm{x})=0$ and write $\bm{x}=y_1\bm{b}_1+y_2\bm{b}_2+y_3\bm{b}_3$. Then according to (iv) of Lemma~\ref{Le: GoodLattice} we must have $|y_i|_\nu \ll |\bm{x}|_\nu\lambda_i^{-d_\nu}$. In particular, it follows from Lemma~\ref{lem.minkowskis_second_theorem} and Proposition~\ref{Prop: UniformConics} that 
    \begin{align*}
    \#\left\{\bm{y}\in \PP^2(K)\colon \begin{array}{l}F(y_1\bm{b}_1+y_2\bm{b}_2+y_3\bm{b}_3)=0,\\  |y_i|_\nu \ll R_{1,\nu}R_{2,\nu}R_{3,\nu}\lambda_i^{d_\nu}\end{array}\right\} 
    &\ll 1+\left(\frac{|\bm{R}_1||\bm{R}_2||\bm{R}_3|^3}{(\lambda_1\lambda_2\lambda_3)^{d_K}}\right)^{1/3}\\
    &\ll 1+ \left(\frac{|\bm{R}_1||\bm{R}_2||\bm{R}_3|}{\det(\Gamma)}\right)^{1/3}\\
    &\ll 1+\left(\frac{|\bm{R}_1||\bm{R}_2||\bm{R}_3|\mathrm{N}(\Delta_0(F))^{3/2}}{\mathrm{N}(\Delta(F))}\right)^{1/3}.    
    \end{align*}   
    As every $\bm{x}$ with $F(\bm{x}) = 0$ is contained in at most $O(\tau(\Delta(F)))$ different lattices $\Gamma$ as above, the result follows.
\end{proof}
In addition we will at some point require the following result, which was proved for number fields by Broberg~\cite[Proposition 7]{broberg2001rational}.
\begin{lemma} \label{lem.quadratic_lemma_for_dp4_broberg} Let $q\in \O_K[x_1,x_2,x_3]$ be a non-singular quadratic form such that $q(0, x_2,x_3)$ is also non-singular. Let $\bm{R}\in \RR_{\geq 1}^{s_K}$ and $R\geq 1$ be given. Then 
\[
\#\{\bm{x}\in (L(\bm{R})\times \O_K^2)\cap Z_2\colon \norm{x_i}\leq R, i=2,3, q(\bm{x})=0\}\ll_K 1+|\bm{R}|(|\bm{R}|\norm{q}R)^\varepsilon.
\]
    
\end{lemma}
\begin{proof}
    As $q(0,x_2,x_3)$ is non-singular, we can find a matrix $P\in \GL_2(K)$ with entries $p_{ij}\in \O_K$ satisfying $\norm{p_{ij}}\ll \norm{q}^B$ for some absolute constant $B>0$ and such that $q(x_1, P(x_2,x_2))$ takes the shape 
    \[
    \alpha L_1(x_1,x_2,x_3)^2+\beta L_2(x_1,x_2,x_3)^2-\gamma x_1^2,
    \]
with non-zero coefficients $\alpha, \beta, \gamma$ and $L_i$ are linear forms such that $L_1(0, x_2,x_3)$ and \break $L_2(0,x_2,x_3)$ are not proportional. Upon multiplying by a suitable constant, we may assume that $\alpha,\beta,\gamma$ lie in $\O_K$ and satisfy $\norm{\alpha},\norm{\beta}, \norm{\gamma} \ll \norm{q}^B$, with possibly a new value of $B$. If we set $L=K(\sqrt{\beta})$, then the equation $q(x_1,P(x_2,x_3))=0$ becomes
\[
L_3(x_1,x_2,x_3)L_4(x_1,x_2,x_3)=\gamma x_1^2,
\]
where $L_3,L_4$ are linear forms with coefficients in $\O_L$. Note that if $x_1=0$, then there are $O(1)$ possibilities for $(0,x_2,x_3)\in Z_3$ such that $q(0,x_2,x_3)=0$, so we may suppose $x_1\neq 0$ from now on. Set $z_2=L_3(x_1,x_2,x_3)$ and $z_3=L_4(x_1,x_2,x_3)$ and observe that the assignment is injective, as $q(0,x_2,x_3)$ is non-singular. We then have $|z_i|_\nu \ll \norm{q}^B \max\{R_\nu, R\}\ll \norm{q}^B|\bm{R}|R$, so that Lemma~\ref{Le: Divisor.bound.elements} implies that for $x_1$ fixed, there are $O((\norm{q}^B|\bm{R}|R\mathrm{N}(\gamma x^2_1))^\varepsilon)=O((\norm{q}|\bm{R}|R)^\varepsilon)$ possible $(z_2, z_3)$ such that $z_2z_3=\gamma x_1^2$. As there are $O(|\bm{R}|)$ available $x_1\in L(\bm{R})$ by Lemma~\ref{Le: Number.OK.Points}, which is sufficient. 
\end{proof}
\section{Rational points on smooth genus 1 curves}
Let $E\subset \PP^n$ be a smooth genus 1 curve of degree $d$ and for $B\geq 1$ define 
\[
N_E(B)\coloneqq \#\{x\in E(K)\colon H(x)\leq B\},
\]
where $H\colon \PP^n(K)\to \RR_{>0}$ denotes the usual height function. Our goal of this section is to produce a \emph{uniform} upper bound for $N_E(B)$ and prove Proposition \ref{Prop.UpperBoundEllCurve}. Let $\mathscr{D}_E$ and $\mathscr{C}_E$ denote the minimal discriminant and the conductor of $E$, which are either ideals or divisors depending on whether $K$ is a number field or a function field. The main input is the \emph{rank growth hypothesis (RGH)} stated in Conjecture \ref{Conj: RGH} for elliptic curves $E$:
\begin{equation}\label{Eq.RGH}
r_E = o(\log \mathrm{N} (\mathscr{C}_E)) \quad\text{as }\mathrm{N} (\mathscr{C}_E)\to \infty,
\end{equation}
where $r_E$ denotes the rank of $E$. In his remarkable work, Brumer~\cite[Proposition 6.9]{Brumerelliptic} proved  Conjecture\eqref{Conj: RGH} when $\cha(K)>3$.
\begin{prop}
    Suppose $\cha(K)>3$. Then \eqref{Eq.RGH} holds with the estimate
    \[
    r_E \ll_K \frac{\log \mathrm{N}(\mathscr{C}_E)}{\log \log \mathrm{N}(\mathscr{C}_E)}.
    \]
\end{prop}

We will prove Proposition \ref{Prop.UpperBoundEllCurve} in several steps. Let us first show how \eqref{Eq.RGH} implies a uniform upper bound for $N_E(B)$ when $E$ is an elliptic curve given in Weierstrass form 
\begin{equation}\label{Eq:Weierstrass}
 E\colon zy^2=x^3+axz^2+bz^3  
\end{equation}
with $a,b\in \O_K$. By abuse of notation we shall also write
\[
E(x,y,z)=zy^2-(x^3+axz^2+bz^3). 
\]
Let $h\colon E(K)\to \RR_{\geq 0}$ be the canonical height of $E$ and define the height of $E$ to be 
\[
H_E\coloneqq \prod_{\nu\in \Omega_K}\max\{1, |a|_\nu^{1/2},|b|_\nu^{1/3}\}.
\]
Provided $\cha(K)\neq 2,3$, it follows from work of Zimmer \cite[p. 40]{ZimmerHeight} that 
\begin{equation}\label{Eq:CanWeilheight}
h(P)= \log H(P)+O(\log H_E)
\end{equation}
for $P\in E(K)$. The Mordell--Weil group $E(K)$ is a finitely generated abelian group of rank $r_E$ and so any point $P\in E(K)$ may be uniquely written as 
\[
P=T+\sum_{i=1}^{r_E}n_iT_i,
\]
where $T\in E(K)_{\text{tors}}$ is a torsion point and $T_1,\dots, T_{r_E}$ are generators of $E(K)$. In this case we have 
\[
h(P)= h\left(\sum n_iT_i\right)=Q_E(n_1,\dots, n_{r_E})
\]
where $Q_E\in \ZZ[x_1,\dots, x_{r_E}]$ is a positive-definite quadratic form. By work of Mazur for $K=\QQ$ \cite{MazurBound}, Merel for number fields \cite{MerelBound} and Levin for non-isotrivial elliptic curves over function fields \cite{LevinBound}, we know $E(K)_{\text{tors}}\ll 1$, where the implied constant only depends on the degree of the number field in characteristic 0 or the genus of the function field in positive characteristic. For isotrivial curves over function fields, it is explained in~\cite{ulmer_lecture_notes} that we always obtain the bound $E(K)_{\text{tors}}\ll q^2$, where $q$ is the cardinality of the field of constants $K$. It thus follows from \eqref{Eq:CanWeilheight} that for a suitable absolute constant $C>0$ we have 
\[
N_E(B) \ll_K \#\left\{(n_1,\dots, n_{r_E})\in \ZZ^{r_E}\colon Q_E(n_1,\dots, n_{r_E})\leq C\log(H_E B)\right\}. 
\]
Lemma~4 of Heath-Brown \cite{heath1998counting} implies that 
\begin{equation} \label{eq.upper_bound_smalles_value_quadric}
N_E(B)\ll 1 + (9C\log(H_EB)/B_{0,E})^{r_E/2},
\end{equation}
where 
\[
B_{0,E}=\min\{h(P) \colon h(P) \neq 0\}.
\]
Hindry--Silverman \cite[Corollary 4.2]{HindrySilverman} show that 
\[
B_{0,E} \geq \log (\mathrm{N}(\mathscr{D}_E))\exp(-A \log( \mathrm{N}(\mathscr{D}_E))/ \log(\mathrm{N}(\mathscr{C}_E))),
\]
where $\mathscr{D}_E$ is the minimal discriminant of $E$. Note that strictly speaking Corollary~4.2 of Hindry--Silverman is only stated for number fields. However, it is an immediate consequence of Theorem~4.1 in the same work, which is valid for any global field. For now assume that $\mathrm{N}(\mathscr{C}_E) \geq C(\varepsilon)$ where $C(\varepsilon)$ is a constant chosen sufficiently large so that $r_E \leq \varepsilon \log \mathrm{N}(\mathscr{C}_E)/A$ holds.
Following the analysis of Heath-Brown \cite{heath1998counting} on page 23 word by word, we arrive at the estimate 
\begin{equation}\label{Eq:UpperBoundWeierstrassI}
N_E(B)\ll_\varepsilon (BH_E)^\varepsilon.
\end{equation}
For the case when $\mathrm{N}(\mathscr{C}_E) <C(\varepsilon)$ our treatment differs according to whether $K$ is a function field or a number field.  If $K$ is a number field then there are only finitely many elliptic curves of a given conductor (cf.~\cite[IX.6]{silverman_arithmetic_ell_curves}). Thus we find
\[
\max_{E \colon \mathrm{N}(\mathscr{C}_E) \leq C(\varepsilon)} r_E \ll_{\varepsilon,K} 1,
\]
and
\[
\min_{E \colon \mathrm{N}(\mathscr{C}_E) \leq C(\varepsilon)} B_{0,E}\gg_{\varepsilon,K} 1,
\]
where the lower bound follows from the fact that the quadratic forms that determine the height of a point are positive definite.
Hence~\eqref{eq.upper_bound_smalles_value_quadric} delivers~\eqref{Eq:UpperBoundWeierstrassI} for elliptic curves $E$ with $\mathrm{N}(\mathscr{C}_E) \leq C(\varepsilon)$.

If $K$ is now a function field, then Theorem~6.1 in~\cite{HindryRegulators} shows that unless $H_E = 1$ we have
\[
B_{0,E} \gg_K p^{-2e}\log \mathrm{N}(\mathscr{D}_E),
\]
where $p^e$ denotes the inseparability degree of the field extension $K/\FF_q(j(E))$, which by convention is 1 if $j(E)\in \FF_q$. If $j(E)$ is non-constant, then $e$ is the maximal non-negative integer such that $j(E)\in K^{p^e}$. In particular, we have $p^e\ll \log \mathrm{N}(j(E)) \ll \log \mathrm{N}(\mathscr{D}_E)$. Moreover, since $\log \mathrm{N}(\mathscr{C}_E) \ll_\varepsilon 1$ again implies $r_E \ll_\varepsilon 1$, and noting $\log \mathrm{N}(\mathscr{D}_E) \ll \log H_E$ we obtain~\eqref{Eq:UpperBoundWeierstrassI} by substituting the aforementioned estimates into~\eqref{eq.upper_bound_smalles_value_quadric}. Finally, if $\mathrm{N}(\mathscr{D}_E) = 1$ then $E$ has everywhere good reduction. This implies that $E$ is isotrivial, meaning that its $j$-invariant $j(E)$ is an element of the field of constants  $\FF_q$. 
In particular it follows from the theory of elliptic curves over function fields that $E$ is isomorphic to a twist of a constant elliptic curve over $K$. Since twisting by any non-constant element of $K$ increases
the norm $\mathrm{N}(\mathscr{D}_E)$, which is for example proved in Proposition 5.7.1 in~\cite{connell1999elliptic} for quadratic twists (the other twists follow analogously). We deduce that the number of elliptic curves $E$ with everywhere good reductions is bounded by $O_{\FF_q}(1)$. In particular we have $B_{0,E} \gg_{\FF_q} 1$ for all such curves and thus obtain~\eqref{Eq:UpperBoundWeierstrassI}.

Therefore it only remains to remove the dependence on $H_E$ in~\eqref{Eq:UpperBoundWeierstrassI} to complete the proof of Proposition \ref{Prop.UpperBoundEllCurve} for elliptic curves in Weierstrass form. To do so, we reprove Theorem~4 of Heath-Brown~\cite{HeathBrownCurvesSurfaces} over arbitrary global fields. 
\begin{lemma}\label{Le:SizeCoeffcs}
    Let $F\in \O_K[x_1,x_2, x_3]$ be an irreducible form of degree $d$ whose coefficient vector lies in $Z_{N-1}$, where $N=(d+1)(d+2)/2$. Then  
    \[
    \#\{\bm{x}\in \PP^2(K)\colon F(\bm{x})=0,\, H(\bm{x})\leq B\}\leq d^2 \quad\text{or}\quad \norm{F}\ll B^{dN/s_K}.
    \]
\end{lemma}
\begin{proof}
    The proof is identical to that of Heath-Brown, and so we shall be brief. Write $N=(d+1)(d+2)/2$ and suppose $N_F(B)\geq d^2+1$. Let $\bm{x}_1,\dots, \bm{x}_N\in Z_2$ be such that $H(\bm{x}_i)\leq B$ and $F(\bm{x}_i)=0$. Consider the $(d^2+1)\times N$ matrix $M$ whose $i$th row consists of all possible monomials of degree $d$ formed by the coordinates of $\bm{x}_i$. If $\bm{f}\in Z_{N-1}$ is the coefficent-vector associated to $F$, then we have $M\bm{f}=0$. Moreover, as $\bm{f}\neq 0$, $M$ must have rank at most $N-1$ and so there exists a non-zero solution $\bm{g}\in Z_{N-1}$  constructed out of the minors of $M$, so that $\norm{\bm{g}}\ll B^{dN/s_K}$. Let $G$ be the form of degree $d$ associated to $\bm{g}$. Since the hypersurfaces defined by $F$ and $G$ intersect in $d^2+1$ points, Bezout's theorem implies that $G$ is a multiple of $F$, so that $\norm{F}\ll \norm{G}\ll B^{dN/s_K}$ since both $\bm{f}$ and $\bm{g}$ belong to $Z_{N-1}$. 
\end{proof}
It is clear that $E$ given as in~\eqref{Eq:Weierstrass} is defined by a form whose coefficients define the unit ideal and upon multiplying it by a unit $u\in \O_K$, we may assume that its coefficient vector lies in $Z_{9}$. We then have 
\[
H_E\leq \prod_{\nu\mid\infty}\max\{1,|a|_\nu, |b|_\nu|\} = \prod_{\nu\mid\infty}\max\{|u|_\nu, |ua|_\nu, |ub|_\nu\} \ll \norm{E}^{s_K}.
\]
In particular, Lemma \ref{Le:SizeCoeffcs} implies that either $N_E(B)\leq 9$ or $H_E\ll B^{10}$. The former case is clearly sufficient and if the latter holds, then \eqref{Eq:UpperBoundWeierstrassI} hands us the estimate 
\begin{equation}\label{Eq:UpBoundWeierstrassII}
    N_E(B)\ll B^\varepsilon,
\end{equation}
for any elliptic curve in Weierstrass form where the implied constant only depends on the ground field $K$. 

We will now use \eqref{Eq:UpBoundWeierstrassII} to deduce Proposition \ref{Prop.UpperBoundEllCurve}. First let us suppose that $E\subset \PP^2$ is a smooth genus 1 curve given as the vanishing locus of a non-singular cubic form $G\in \O_K[x, y, z]$. If $N_E(B)=0$, then we clearly have the desired upper bound. If not, we can find $t_0\in E(K)$ with $H(t_0)\leq B$. and use $t_0$ to construct a birational map $\theta\colon \PP^2\to \PP^2$ defined over $K$ transforming $E$ into an elliptic curve $E'$ defined in Weierstrass form. It is clear that $\theta$ has coefficents that are rational functions in the coefficients of $G$ and the coordinates of $t_0$. It thus follows for any $x\in E(K)$ with $H(x)\leq B$ that $H(\theta(x))\leq C H(t_0)^A\norm{G}^A$ for some absolute  constant $C>0$ and hence 
\[
N_E(B)\ll N_{E'}(CB^A\norm{G}^A).  
\]
By Lemma \ref{Le:SizeCoeffcs} we have $N_E(B)\leq 9$ or $\norm{G}\ll B^A$, so that in the latter case \eqref{Eq:UpBoundWeierstrassII} yields
\[
N_E(B)\ll (C\norm{G}^AB^A)^\varepsilon \ll B^\varepsilon,
\]
which is sufficient for the proof of Proposition \ref{Prop.UpperBoundEllCurve} for plane elliptic curves. \\

Next we suppose $E\subset \PP^3$. If $E$ is contained in any plane, say $H\subset \PP^3$, we can find a point $p\in \PP^3(K)\setminus H(K)$ with $H(p)\ll 1$. Projecting away from $p$ gives a map $\varphi\colon\PP^3\setminus\{p\}\to \PP^2(K)$ that restricts to an isomorphism $E\to \varphi(E)\subset \PP^2(K)$. Since $H(p)\ll 1$, it follows that there exists some absolute constant $C>0$ such that $H(\varphi(x))\leq C H(x)$ for any $x\in \PP^3(K)\setminus\{p\}$, so that 
\[
N_E(B)\leq N_{\varphi(E)}(CB)\ll B^\varepsilon,
\]
by what we have shown for planar curves. We may therefore assume that $E$ is not contained in any hyperplane. Suppose there exists $t_0\in E(K)$ with $H(t_0)\leq B$. Then there exists a change of variables $L\colon \PP^3\to \PP^3$ sending $t_0$ to the point $t_1=[1, 0, 0, 0]$, so that ${H(L(x))\ll H(t_0)^A H(x)\ll B^AH(x)}$. Let $E'$ be the image of $E$ under this linear transformation. Any elliptic curve in $\PP^3$ that is not contained in a hyperplane may be defined as the complete intersection of two quadrics. In particular, we may assume that $E'$ is given by 
\begin{equation}\label{Eq:EllInt2Quadrics}
x_0L_1(x_1,x_2,x_3)=q_1(x_1,x_2,x_3)\quad\text{and}\quad x_0L_2(x_1,x_2,x_3)=q_2(x_1,x_2,x_3)
\end{equation}
for some linear forms $L_1,L_2\in\O_K[x_2,x_3,x_4]$ and quadratic forms $q_1,q_2\in\O_K[x_2,x_3,x_4]$. Since $E'$ is non-singular, the Jacobian criterion implies that $L_1$ and $L_2$ are not proportional. In particular, we can eliminate $x_0$ from \eqref{Eq:EllInt2Quadrics} to produce an equation 
\[
L_2(x_1,x_2,x_3)q_1(x_1,x_2,x_3)=L_1(x_1,x_2,x_3)q_2(x_1,x_2,x_3)
\]
which gives a plane elliptic curve $E''$ containing the point corresponding to $L_1(x_1,x_2,x_3)=L_2(x_1,x_2,x_3)=0$. In particular, the map $E'\setminus\{t_1\}\to E''$ defines a birational map $E'\dashrightarrow E''$ which is one-to-one except when $L_1(x_1,x_2,x_3)=L_2(x_1,x_2,x_3)=0$. There are are most $O(1)$ such points and thus 
\[
N_E(B)\ll N_{E''}(CB^A)\ll B^\varepsilon,
\]
again by our estimate for planar curves. \\

The only remaining cases are when $n>3$, which we shall reduce to the case $n=3$ shortly. So suppose $E\subset \PP^n$ is a non-singular genus 1 curve of degree $d$. If $E$ is contained in a hyperplane $H\subset \PP^n$, we can find a point $p\in \PP^n(K)\setminus H(K)$ such that $H(p)\ll 1$. Projection away from $p$ defines an isomorphism $\varphi\colon E\to E'\subset \PP^{n-1}$. Moreover, it is easy to see that $H(p)\ll 1$ implies $H(\varphi(x))\ll H(x)$. We may therefore assume that $E$ is not contained in any hyperplane from now on. Let $S(E)\subset \PP^n$ denote the secant variety of $E$, which by definition is the closure of all lines that meet $E$ in two points. It is well known that $\dim (S(E))\leq 3$, $S(E)$ is irreducible and that projection away from a point $p\in \PP^n$ restricts to a closed immersion of $E$ if and only if $p\not\in S(E)$. Theorem~4.3 of \cite{terracini} gives a formula for the degree of $S(E)$ that only depends on $d$ and $n$, and since $n>\dim(S(E))$ we may use \eqref{Eq:UnifVariety} to deduce the existence of a point $p\in \PP^n(K)\setminus S(E)(K)$ such that $H(p)\ll_{d,n,K} 1$ and $\varphi \colon \PP^n\setminus \{p\}\to \PP^{n-1}$ restricts to an isomorphism $E\to E'\subset \PP^{n-1}$, where $E'=\varphi(E)$ is a non-singular genus one curve of degree at most $d$. Since $H(p)\ll 1$, we again have $H(\varphi(x))\ll_{d,n,K} H(x)$. Regardless of whether $E$ is contained in a hyperplane or not, we have shown the existence of an elliptic curve $E'\subset \PP^{n-1}$ of degree at most $d$ such that 
\[
N_E(B)\leq N_{E'}(CB)
\]
for some constant $C$ only depending on $d,n$ and $K$. As $n>3$, it is clear that we continue this process until $E'\subset \PP^3$, a case that we already dealt with. Therefore, we have completed our proof of Proposition \ref{Prop.UpperBoundEllCurve}.
 \begin{remark}
     It is well known that any elliptic curve over a global field $K$ with $\cha(K)\neq 2,3$ can be put into short Weierstrass form \cite[Chapter III, Proposition 3.1]{silverman_arithmetic_ell_curves}. This is usually proved via the Riemann-Roch theorem. However, for our applications it is important that we keep track of how the height of points is affected, which we could achieve by working with projections instead. 
 \end{remark}
\section{Conic bundles}\label{Sec.conicbundles}
The aim of this section is to prove Theorem~\ref{Th: TheConicBundleTheorem}.
That is, we shall prove that  if $\cha(K) \neq 2$ and $X$ is a del Pezzo surface of degree $d=4$ or $d=5$ admitting a conic bundle structure, then we have
 \[
    N_U(B) \ll B^{1+\varepsilon}.
\]
From now on until the end of this work, all implied constants are allowed to depend on the del Pezzo surface under consideration.
\subsection{del Pezzo surfaces of degree $5$ with a conic bundle structure}
According to Section 5 of~\cite{frei_loughran_sofos} a del Pezzo surface of degree $5$ may be realised as a nonsingular hypersurface defined by
\begin{equation} \label{eq.dp5_conic_defn}
sQ_1(x_0,x_1,x_2) + tQ_2(x_0,x_1,x_2) = 0
\end{equation}
inside $\mathbb{P}^1 \times \PP^2$. At the beginning of Section $5$ in~\cite{frei_loughran_sofos} the authors make the assumption that the field $K$ is perfect. This is not necessary in order to show that the $X$ can be realised as described above, in particular their proof for~\eqref{eq.dp5_conic_defn} goes through as long as $\cha(K) \neq 2$.
Since $X$ is nonsingular the determinant 
\[
\Delta(s,t) = \det(sM_1+tM_2)
\]
defines a separable cubic form in $(s,t)$, where $M_i$ are the $3 \times 3$ matrices defining $Q_i$. 
The associated height function is given by
\[
H([s,t],[x_0,x_1,x_2]) = H([s,t]) H([x_0,x_1,x_2]),
\]
where the factors on the right hand side denote the usual height in $\PP^1$ and $\PP^2$, respectively. Furthermore, the exceptional curves on $X$ are given by the points such that $Q_1(\bm{x}) = Q_2(\bm{x}) = 0$ or $\Delta(s,t) = 0$. We proceed in two cases.

\textbf{Case 1:} $H([s,t]) \leq B^{1/2}$. Since we wish to count points such that $H(\bm{x})H([s,t]) \leq B$ we can cover the contribution to the counting function $N_U(B)$ by
\[
\sum_{\substack{\norm{(s,t)}_\infty \leq B^{1/2} \\ \Delta(s,t) \neq 0}} N_{{(s,t)}}\left( \frac{B}{\norm{(s,t)}_\infty}\right),
\]
where the sum runs over $(s,t) \in Z_1$ and where
\[
N_{(s,t)}(R) \coloneqq \# \left\{ \bm{x} \in Z_2 \colon \norm{\bm{x}} \ll R^{1/s_K}, \; sQ_1(\bm{x}) + tQ_2(\bm{x}) = 0 \right\}.
\]
Note that we are permitted to exclude the zeroes of the discriminant from this summation, since they lie on the exceptional locus. We can apply Proposition~\ref{Prop: UniformConics} to see that the contribution is bounded above by
\[
\sum_{\norm{(s,t)}_\infty \leq B^{1/2}} B^\varepsilon \left(1 + \frac{B \mathrm{N}(\Delta_0(s,t))^{1/2}}{\norm{(s,t)}_\infty \mathrm{N}(\Delta(s,t))^{1/3}}\right).
\]
The proof of Lemma 7 in~\cite{broberg2001rational}, which is a basic application of elimination theory, carries over identically to our setting and shows that we have $\mathrm{N}(\Delta_0(s,t))^{1/2} \ll 1$. For the remaining range we will divide the values of $\norm{(s,t)}$ and $\mathrm{N}(\Delta(s,t))$ into dyadic intervals. If $\norm{(s,t)} \sim R^{1/s_K}$ and $\mathrm{N}(\Delta(s,t)) \sim S$ then Lemma~\ref{Le: Polysmall} shows that there are at most $R^{1+\varepsilon}S^{1/3}$ many points $(s,t)$ that lie within such a dyadic interval. Hence the contribution from this dyadic interval is bounded by
\[
R^{1+\varepsilon} S^{1/3}\left(1+\frac{B}{RS^{1/3}} \right) \ll B^{1+\varepsilon},
\]
since $R \leq B^{1/2}$ and $S \ll R^{3} \ll B^{3/2}$. The number of dyadic intervals that we have to consider is bounded by $\log^2(B)$ and since each individual contribution is bounded by $B^{1+\varepsilon}$ this suffices to show that the contribution to $N_U(B)$ is indeed bounded by $B^{1+\varepsilon}$.

\textbf{Case 2:} $H([s,t]) \geq B^{1/2}$. In particular we must have $H([x_0,x_1,x_2]) \leq B^{1/2}$. Let $([s,t],[x_0,x_1,x_2])$ be a point of the del Pezzo surface in question. 
Hence by~\eqref{eq.dp5_conic_defn} a representative of $[s,t]$ takes the shape 
\[
(s,t) = (Q_1(\bm{x}),-Q_2(\bm{x})).
\]
Writing $\mathfrak{d}$ for the ideal generated by $Q_1(\bm{x})$ and $Q_2(\bm{x})$, we then have 
\[
H([s,t]) = \frac{\norm{Q_1(\bm{x}),Q_2(\bm{x})}_\infty}{\mathrm{N}(\mathfrak{d})}.
\]
In what follows, given an ideal $\mathfrak{d} \subset \O_K$ such that $\mathrm{N}(\mathfrak{d}) \ll B$ we will count the number of $\bm{x} \in Z_2$ with $Q_i(\bm{x}) \equiv 0 \; \mathrm{mod} \, \mathfrak{d}$. In particular, given $\alpha \ll B^{1/2}$ we may divide the height of $\bm{x}$ into dyadic intervals. We are therefore led to consider sums of the shape
\begin{equation} \label{eq.first_sum_dp5}
\sum_{\substack{\bm{x} \in Z_2 \colon \norm{\bm{x}}_\infty \sim \alpha \\ (Q_1(\bm{x}),Q_2(\bm{x})) = \mathfrak{d} \\ \norm{Q_i(\bm{x})}_\infty  \leq \frac{B\mathrm{N}(\mathfrak{d})}{\alpha}}} 1.
\end{equation}
In view of this we define the set
\[
\widetilde{V}_\mathfrak{d} = \left\{ \bm{x} \in \left(\O_K/\mathfrak{d} \right)^3 \colon (\bm{x}, \mathfrak{d}) \in Z_3, \; Q_1(\bm{x}) \equiv Q_2(\bm{x}) \equiv 0 \; \mathrm{mod} \; \mathfrak{d} \right\},
\]
where the condition $(\bm{x}, \mathfrak{d}) \in Z_3$ is meant to indicate that the ideal generated by $x_0,x_1,x_2$ and $\mathfrak{d}$ is equal to one of the fixed representatives $\mathfrak{a}_i$ of the ideal class group of $K$. Further we define
\[
V_\mathfrak{d} = \widetilde{V}_\mathfrak{d} /(\O_K/\mathfrak{d})^\times.
\]
Given $\bm{y} \in V_\mathfrak{d}$ we further define 
\[
\Lambda_\mathfrak{d}(\bm{y}) \coloneqq \{\bm{z} \in \O_K^3 \colon \bm{z} \equiv \lambda \bm{y} \; \mathrm{mod} \; \mathfrak{d} \text{ for some $\lambda \in \O_K$} \}.
\]
Changing the order of summation in~\eqref{eq.first_sum_dp5}, using the notation introduced above it is easy to see that it suffices to estimate
\begin{equation} \label{es.S_alpha_defn}
S_\alpha \coloneqq \sum_{\mathrm{N}(\mathfrak{d}) \ll \alpha^2} \sum_{\bm{y} \in V_\mathfrak{d}} \sum_{\substack{\bm{x} \in Z_2 \colon \norm{\bm{x}}_\infty \sim \alpha \\ \norm{Q_i(\bm{x})}_\infty \leq \frac{B\mathrm{N}(\mathfrak{d})}{\alpha} \\ \bm{x} \in \Lambda_\mathfrak{d}(\bm{y})}} 1.
\end{equation}
To this end, the following two lemmas help us deal with $V_\mathfrak{d}$ and the lattice $\Lambda_\mathfrak{d}(\bm{y})$.
Since the del Pezzo surface in question is non-singular it follows from the definition~\eqref{eq.dp5_conic_defn} that 
$Q_1(\bm{x}) = Q_2(\bm{x}) = 0$ defines a non-singular intersection in $\PP^2$.
\begin{lemma} 
    Given an ideal $I \subset \O_K$ we have
    \[
    \# {V}_I \ll \mathrm{N}(I)^{\varepsilon}.
    \]
\end{lemma}
\begin{proof}
    The proof is the same as Lemma 2 in~\cite{brownMunsh} adjusted to our more general setting and therefore we will be brief. Note first that via multiplicativity and homogeneity it suffices to show
    \[
    \# \widetilde{V}_{\mathfrak{p}^s} \ll \mathrm{N}(\mathfrak{p}^s),
    \]
    for any prime ideal $\mathfrak{p} \subset \O_K$ and any positive integer $s$. Note also that unless $\mathfrak{p}$ is one of the $O(1)$ many fixed representatives for the class group of $K$ the condition $(\bm{x},\mathfrak{p}^s) \in Z_3$ is equivalent to saying that $(x_1,x_2,x_3,\mathfrak{p}^s)$ generates $\O_K$. Also consider the isomorphism
    \[
    \O_K/\mathfrak{p}^s \cong \O_{\nu}/(\pi)^s,
    \]
    where $\O_\nu \subset K_\nu$, for $\nu$  the place corresponding to $\mathfrak{p}$, and where $\pi$ is a uniformizer for $K_\nu$. Then we find that it suffices to consider
    \[
    \rho(\pi^s) = \#\left\{\bm{x} \in (\O_\nu/(\pi^s))^3 \colon \langle x_1,x_2,x_3,\pi\rangle = \O_\nu, \, Q_i(\bm{x}) \equiv 0 \; \mathrm{mod} \, (\pi^s), \text{ for } i=1,2 \right\}.
    \]
    For ease of notation write $\psi$ for the additive character on $K_\nu$. Using the orthogonality relation from Lemma~\ref{lem.character_orthogonality} we find
    \begin{equation} \label{eq.rho}
    \rho(\pi^s)= \frac{1}{\mathrm{N}(\pi)^{2s}} \sum_{\bm{b} \; \mathrm{mod}\, \pi^s} \; \psum_{\bm{x} \; \mathrm{mod}\, \pi^s} \psi\left( \frac{b_1Q_1(\bm{x}) + b_2Q_2(\bm{x})}{\pi^s} \right),
    \end{equation}
    where $\sum^*$ indicates that we only sum over tuples $\bm{x}$ such that $\langle \bm{x},\pi^s\rangle = \O_\nu$, or equivalently $\pi \nmid x_i$. Extracting common factors between $\pi^s$ and $\bm{b}$ in the display above we obtain
    \[
    \rho(\pi^s) = \frac{1}{\mathrm{N}(\pi)^{2s}} \sum_{0 \leq i < s} \mathrm{N}(\pi)^{3i} S(s-i) + \mathrm{N}(\pi)^s \left(1-\frac{1}{\mathrm{N}(\pi)}\right),
    \]
    where
    \[
    S(k) = \psum_{\bm{b} \; \mathrm{mod}\, \pi^k} \;  \psum_{\bm{x}\;  \mathrm{mod}\, \pi^k} \psi\left( \frac{F(\bm{b},\bm{x})}{\pi^k} \right),
    \]
    and where $F(\bm{b},\bm{x}) = b_1Q_1(\bm{x})+b_2Q_2(\bm{x})$. It suffices to show $S(k) = O(1)$ for $k \geq 2$ and $S(1) = O(\mathrm{N}(\pi)^3)$. Regarding $S(1)$ note that by Bézout's theorem we find that $\rho(\pi) \ll \mathrm{N}(\pi)$ since $\O_\nu/(\pi) \cong \O_K/\mathfrak{p} \cong \FF_{\mathrm{N}(\mathfrak{p})}$. Note that there is a finite number of primes $\mathfrak{p}$ the two quadrics might share a common component and so Bézout's theorem does not apply for these cases. However, by choosing the implied constant large enough since there are only finitely many primes involved we still obtain $\rho(\pi) \ll \mathrm{N}(\pi)$ for all primes. Substituting this into~\eqref{eq.rho} for $s=1$ we indeed find $S(1) = O(\mathrm{N}(\pi)^3)$. If $k \geq 2$ then after  introducing a dummy sum over $a \in (\O_\nu/\pi^k)^\times$, and making a change of variables $\bm{b} \mapsto a \bm{b}$ we may evaluate the arising Ramanujan sums to obtain
    \[
    S(k) = \left(1-\frac{1}{\mathrm{N}(\pi)} \right)^{-1} \left( \mathcal{N}(k) - \mathrm{N}(\pi)^{4} \mathcal{N}({k-1}) \right),
    \]
    where $\mathcal{N}(k)$ denotes the number of solutions to $F(\bm{b},\bm{x}) \equiv 0 \; \mathrm{mod} \; \pi^k$ such that $\pi \nmid \bm{b}$ and $\pi \nmid \bm{x}$. Writing the discriminant $D \in K$ of the pair of quadrics $Q_i$ as the resultant of the $5$ quadratic forms appearing in $\nabla F(\bm{b},\bm{x})$ (see~\cite[Chapter 13]{gelfand_discriminants}) via elimination theory we obtain polynomials $G_{ij}(\bm{y})$ with coefficients in $\O_K$ where $\bm{y} = (\bm{b},\bm{x})$ and a positive integer $R$ such that
    \[
    Dy_i^R = \sum_{1 \leq j \leq 5} G_{ij}(\bm{y}) \frac{\partial F}{\partial y_i}, \quad \text{for } 1 \leq i \leq 5.
    \]
    Since $Q_i$ define a non-singular intersection over $K$ we have that $D \neq 0$. Writing $\delta_{\pi} = \nu_{\mathfrak{p}}(D)$ we thus obtain that if $\pi^m \mid \nabla F(\bm{b},\bm{x})$ and $\pi \nmid \bm{x}$ and $\pi \nmid \bm{b}$ then $m \leq \delta_\pi$. Since $D \in \O_K$ is nonzero we first note that $\delta_\pi =0$ for all but finitely many $\mathfrak{p}$.  In particular, if $\delta_\pi > 0$  and $2 \leq k \leq 2 \delta_\pi +1$ then we may choose our implied constant sufficiently big, only depending on $Q_i$ to get $S(k) = O(1)$. If $k \geq 2 \delta +2$ then we may apply a standard Hensel lifting argument to show that
    \[
    C_m(k+1) = \mathrm{N}(\pi)^{4} C_m(k),
    \]
    where $C_m(k)$ is the number of $\bm{y} \; \mathrm{mod} \, \pi^k$ with $\pi \nmid \bm{b}$ and $\pi \nmid \bm{x}$ such that $\pi^k \mid F(\bm{y})$ and $\pi^m \parallel \nabla F(\bm{y})$. We note that since $\O_\nu$ is a principal ideal domain the lifting argument goes through completely analogously. Finally, noting that
    \[
    \mathcal{N}(k) = \sum_{0 \leq m \leq \delta} C_m(k)
    \]
    yields $S(k) = 0$ if $k \geq 2 \delta +2$, and thus $S(k) = O(1)$ for all $k \geq 2$, thereby completing the proof.
\end{proof}

\begin{lemma} \label{lem.determinant_for_lattice}
    Given $\bm{y} \in \O_K^3$ such that $(y_1,y_2,y_3,\mathfrak{d})$ generates one of the fixed representatives $\mathfrak{a}_i$ of the ideal class group, then the lattice $\Lambda_\mathfrak{d}(\bm{y})$ has rank $3$ and
    \[
    \det \Lambda_\mathfrak{d}(\bm{y}) \asymp_K \mathrm{N}(\mathfrak{d})^2.
    \]
\end{lemma}
\begin{proof}
    First note that $\Lambda_\mathfrak{d}(\bm{y}) \supset (\mathfrak{d} \O_K)^3$ and therefore
    \[
    \det \Lambda_\mathfrak{d}(\bm{y}) = [\O_K^3 \colon \Lambda_\mathfrak{d}(\bm{y})] = \frac{[\O_K^3 \colon (\mathfrak{d} \O_K)^3]}{[\Lambda_\mathfrak{d}(\bm{y}) \colon (\mathfrak{d} \O_K)^3]} = \frac{\mathrm{N}(\mathfrak{d})^3}{[\Lambda_\mathfrak{d}(\bm{y}) \colon (\mathfrak{d} \O_K)^3]}.
    \]
    Hence it suffices to establish $[\Lambda_\mathfrak{d}(\bm{y}) \colon (\mathfrak{d} \O_K)^3] \asymp \mathrm{N}(\mathfrak{d})$. First note that it is clear that $\lambda \bm{y} + (\mathfrak{d} \O_K)^3$ where $\lambda$ runs through all possible elements of $\O_K/\mathfrak{d}$ exhausts all possible cosets in  $\Lambda_\mathfrak{d}(\bm{y}) / (\mathfrak{d} \O_K)^3$ and hence $[\Lambda_\mathfrak{d}(\bm{y}) \colon (\mathfrak{d} \O_K)^3] \leq \mathrm{N}(\mathfrak{d})$.

    For a lower bound regarding the index, given $c >0$ we consider elements $r \in \O_K$ such that $\norm{r} < c \mathrm{N}(\mathfrak{d})^{1/s_K}$. We claim that if we choose $c$ small enough, only depending on $K$, then the elements $r \bm{y} + (\mathfrak{d} \O_K)^3$ where $r$ runs through elements as above are all distinct cosets. Since the number of elements $r$ with the above property is $\gg \mathrm{N}(\mathfrak{d})$ this suffices in order to prove the lemma.
    If not of all the above cosets are distinct then upon taking their difference we see that there exists some $r_0 \in \O_K \setminus \{0\}$ with $
\norm{r_0} \ll c \mathrm{N}(\mathfrak{d})^{1/s_K}$ and $r_0y_i \in \mathfrak{d}$ for $i = 1, 2,3$. In particular we have
    \[
    \langle r_0y_1, r_0y_2,r_0y_3,r_0 \mathfrak{d} \rangle \subset \mathfrak{d} 
    \]
    Since the ideal generated by $y_i$ and $\mathfrak{d}$ is one of the fixed representatives of the class group of $K$ there exists some nonzero $z \in \langle y_1,y_2,y_3,\mathfrak{d} \rangle$ such that $\norm{z} = O_K(1)$. Hence $r_0z \in \mathfrak{d}$ and so after multiplying $r_0$ with a unit if necessary we find
    \[
     \mathrm{N}(\mathfrak{d})^{1/s_K} \leq \norm{r_0z}_\infty^{1/s_K} \ll \norm{r_0z}.
    \]
    Therefore we have
    \[
    \mathrm{N}(\mathfrak{d})^{1/s_K} \ll \norm{r_0z} \ll \norm{r_0}\norm{z} \ll \norm{r_0} \ll c \mathrm{N}(\mathfrak{d})^{1/s_K},
    \]
    and so upon choosing $c>0$ small enough we obtain a contradiction. Thus the claim and consequently the Lemma transpire.
\end{proof}

We now return to estimating $S_\alpha$ as it was defined in~\eqref{es.S_alpha_defn}. Since there are only $B^{\varepsilon}$ many different values of $\alpha$ we need to consider, it suffices to show $S_\alpha \ll B^{1+\varepsilon}$.

We denote the successive minima of $\Lambda_\mathfrak{d}(\bm{y})$ by $\lambda_1 \leq \lambda_2 \leq \lambda_3$, as they were defined in Section~\ref{sec.lattices}.
 Note first that for every $\bm{y} \in V_\mathfrak{d}$ using Lemma~\ref{lem.minkowskis_second_theorem} and Lemma~\ref{lem.good_units} we may take a representative $\bm{y} \in \O_K^3$ such that $|\bm{y}|_\nu \ll \mathrm{N}(\mathfrak{d})^{d_\nu/d_K}$. Let $\bm{x}_1,\bm{x}_2 \in \O_K^3$ with $|\bm{x}_i|_\nu \ll 1$ for all $\nu \mid \infty$ and $i=1,2$ such that $\bm{y},\bm{x}_1,\bm{x}_2$ are linearly independent. Moreover, choose $\delta \in \mathfrak{d}$ such that $|\delta|_\nu \ll \mathrm{N}(\mathfrak{d})^{d_\nu/d_K}$, which is possible by Lemma~\ref{lem.minkowskis_second_theorem}. Then the set $\{\bm{y}, \bm{y} + \delta \bm{x}_1, \bm{y}+\delta \bm{x}_2\}$ constitutes a linearly independent set inside $\Lambda_\mathfrak{d}(\bm{y})$. The $\nu$-adic absolute value of each of these vectors is bounded by $O(\mathrm{N}(\mathfrak{d})^{d_\nu/d_K})$ for all $\nu \mid \infty$. We deduce
\[
\lambda_3^{d_K} \ll \prod_{\nu \mid \infty} \mathrm{N}(\mathfrak{d})^{d_\nu/d_K}= \mathrm{N}(\mathfrak{d}),
\]
and so Lemma~\ref{lem.minkowskis_second_theorem} yields
\[
\frac{1}{(\lambda_1 \lambda_2)^{d_K}} \ll \frac{1}{\mathrm{N}(\mathfrak{d})},
\]
since $\det \Lambda_\mathfrak{d}(\bm{y}) \asymp \mathrm{N}(\mathfrak{d})^2$ by Lemma~\ref{lem.determinant_for_lattice}.
We will use this fact as well as Lemma~\ref{lem.lattice_in_box} in order to estimate $S_\alpha$.
If $\alpha < \lambda_2^{d_K}$ then the number of lattice points with $\norm{\bm{x}}_\infty \sim \alpha$  is bounded by $O(1)$ since we only consider the contribution from primitive points. Thus $S_\alpha$ is bounded by
\[
\sum_{\mathrm{N}(\mathfrak{d}) \ll \alpha^2} \sum_{\bm{y} \in V_\mathfrak{d}} 1 \ll \alpha^{2+\varepsilon} \ll B^{1+\varepsilon}.
\]
If $\lambda_2^{d_K} \leq \alpha < \lambda_3^{d_K}$, then via Lemma~\ref{lem.lattice_in_box} the number of lattice points that we count is bounded by $O(\alpha^2/(\lambda_1\lambda_2)^{d_K})$. Thus we obtain
\[
S_\alpha \ll \sum_{\mathrm{N}(\mathfrak{d}) \ll \alpha^2} \sum_{\bm{y} \in V_\mathfrak{d}} \frac{\alpha^2}{\mathrm{N}(\mathfrak{d})} \ll \alpha^{2+\varepsilon} \ll B^{1+\varepsilon}.
\]
Finally, if $\lambda_3^{d_K}\leq \alpha$ then we divide the contribution into the range $\alpha^3/B \ll \mathrm{N}(\mathfrak{d}) \ll \alpha^2$ and the range $\mathrm{N}(\mathfrak{d}) \ll \alpha^3/B$. For the first range we may employ Lemma~\ref{lem.lattice_in_box} to find that the contribution is bounded by
\begin{equation*} 
\sum_{\alpha^3/B \ll \mathrm{N}(\mathfrak{d}) \ll \alpha^2} \frac{\alpha^3}{\mathrm{N}(\mathfrak{d})^2} \ll \alpha^{3}(\alpha^{-2} + B/\alpha^3) \alpha^\varepsilon \ll B^{1+\varepsilon}.
\end{equation*}
It remains to handle
\begin{equation} \label{eq.dp5_count_lattice_quadratic}
\sum_{\mathrm{N}(\mathfrak{d}) \ll \alpha^3/B} \sum_{\bm{y} \in V_\mathfrak{d}} \sum_{\substack{\bm{x} \in Z_2 \colon \norm{\bm{x}}_\infty \sim \alpha \\  \norm{Q_i(\bm{x})}_\infty \leq \frac{B\mathrm{N}(\mathfrak{d})}{\alpha} \\ \bm{x} \in \Lambda_\mathfrak{d}(\bm{y})}} 1.
\end{equation}
At this point it does not suffice to just count the points contained in a ball inside the lattice, and we need to take advantage of the additional restriction given by the quadratic forms. We deal with this in the following Lemma.
\begin{lemma} \label{lem.lattice_quadric_bounded_region}
    Let $\Lambda \subset K^3$ be a lattice with successive minima $\lambda_1 \leq \lambda_2 \leq \lambda_3$ and let $Q \in \O_K[x_1,x_2,x_3]$ be a quadratic form of rank at least $2$. Let $\alpha, R \geq 1$ be real numbers such that $R^{1/2} \ll \alpha \ll R$. Consider
    \[
    N(\alpha,R) \coloneqq \#\left\{ \bm{x} \in {Z}_2 \cap \Lambda \colon \norm{\bm{x}}_\infty < \alpha, \; \norm{Q(\bm{x})}_\infty < R \right\}.
    \]
We have the bound
\[
N(\alpha,R)\ll  \frac{\alpha^{4+\varepsilon}}{R^2} + \frac{\alpha^{3+\varepsilon}}{R \lambda_1^{d_K}}  + \frac{\alpha^{2+\varepsilon}}{(\lambda_1\lambda_2)^{d_K}} + \frac{\alpha^{1+\varepsilon} R}{\det \Lambda}+ \frac{\alpha}{(\det \Lambda)^{1/3}}. 
\]
\end{lemma}
Deferring the proof for now, applying Lemma~\ref{lem.lattice_quadric_bounded_region} to the inner sum in~\eqref{eq.dp5_count_lattice_quadratic} with $R = B \mathrm{N}(\mathfrak{d})/\alpha$ we find that 
\begin{equation*} 
S_\alpha \ll B^\varepsilon \sum_{\mathrm{N}(\mathfrak{d}) \ll \alpha^3/B} 
\left( \frac{\alpha^6}{B^2 \mathrm{N}(\mathfrak{d})^2} + \frac{\alpha^4}{B\mathrm{N}(\mathfrak{d}) \lambda_1^{d_K}} +  \frac{\alpha^2}{(\lambda_1 \lambda_2)^{d_K}} + \frac{B}{\mathrm{N}(\mathfrak{d})} + \frac{\alpha}{\mathrm{N}(\mathfrak{d})^{2/3}} \right).
\end{equation*}
Using the fact that $\lambda_1 \gg 1$ as well as $1/(\lambda_1\lambda_2)^{d_K} \ll \mathrm{N}(\mathfrak{d})^{-1}$ one may easily check that the above expression is bounded by $B^{1+\varepsilon}$. Once we prove Lemma~\ref{lem.lattice_quadric_bounded_region} this concludes the proof of Theorem~\ref{Th: TheConicBundleTheorem} for $d=5$.

\begin{proof}[Proof of Lemma~\ref{lem.lattice_quadric_bounded_region}]
First note that the contribution to $N(\alpha,R)$ from the vectors $\bm{x} \in Z_2 \cap \Lambda$ such that $Q(\bm{x}) = 0$ is bounded above by $O(1+\alpha/\det(\Lambda)^{1/3})$ by using the same argument as in the proof of Corollary~\ref{Cor: ConicsDisc}.

For the remaining contribution we begin by decomposing the possible local absolute values that the quadratic forms may take into dyadic intervals. Note first that if $Q(\bm{x}) \neq 0$ then $\prod_\nu |Q(\bm{x})|_\nu \geq 1$. Further, if $\bm{x}$ is counted by $N(\alpha,R)$ then $|Q(\bm{x})|_\nu \ll \alpha^{2/s_K}$ holds for all $\nu$. We deduce that the vectors counted by $N(\alpha,R)$ satisfy
\[
\alpha^{-2} \leq  |Q(\bm{x})|_\nu \ll \alpha^{2 /s_K},
\]
for all $\nu \mid \infty$. Fix now a place $\omega \mid \infty$ and let $\bm{r} = (r_\nu)_{\nu \neq \omega}$. We define
\[
N(\alpha,\bm{r}) = \#\left\{ \bm{x} \in  \Lambda \colon |\bm{x}|_\nu \ll \alpha^{1/s_K}, \;  |Q(\bm{x})|_\nu \sim r_\nu \text{ and } \norm{Q(\bm{x})}_\infty \leq R \right\}.
\]
It suffices to obtain establish the desired upper bound for $N(\alpha,\bm{r})$ whenever $\alpha^{-2} \leq r_\nu \leq R \alpha^{2 s_K}$ is satisfied, since the number of dyadic decompositions we require in order to cover the set of points counted by $N(\alpha,R)$ is bounded by $\alpha^\varepsilon$.
Writing $r_\omega = 2^{1-s_K}R/\prod_{\nu \neq \omega} r_\nu$ we clearly see that
\[
N(\alpha,\bm{r}) \leq \# \left\{ \bm{x} \in \Lambda \colon |\bm{x}|_\nu \ll \alpha^{1/s_K}, |Q(\bm{x})|_\nu \leq r_\nu \right\}
\]
Consider the set
\[
S(\alpha,(r_\nu)_\nu) = \prod_\nu \left\{\bm{x} \in K_\nu^3 \colon |\bm{x}|_\nu \ll \alpha^{1/s_K}, |Q(\bm{x})|_\nu \leq r_\nu  \right\}.
\]
 Writing $\Delta \colon K^3 \hookrightarrow \prod_{\nu} K_\nu^3$ for the diagonal embedding, we see that 
 \[N(\alpha,\bm{r}) \ll \# (\Delta^{-1}\left( S(\alpha,(r_\nu)_\nu) \right) \cap \Lambda).\]
 Given $\nu \mid \infty$ such that $K_\nu$ is not isomorphic to the complex numbers, consider the box
 \[
 \mathcal{B}_\nu = \left\{\bm{x} \in K_\nu^3 \colon |\bm{x}|_\nu \leq r_\nu/\alpha^{1/s_K} \right\}.
 \]
 If $K_\nu \cong \CC$ then write
 \[
 \mathcal{B}_\nu = \left\{\bm{x} \in K_\nu^3 \colon \max_i|(\Re(x_i),\Im(x_i)| \leq \sqrt{r_\nu/\alpha^{1/s_K}} \right\},
 \]
 where the absolute value in the definition of $\mathcal{B}_\nu$ is taken to be the usual absolute value on $\RR$. Note that for complex $\nu$ we have
 \[
 \{\bm{x} \in K_\nu^3 \colon |\bm{x}|_\nu \leq r_\nu/\alpha^{1/s_K} \} \subset \mathcal{B}_\nu  \subset \{\bm{x} \in K_\nu^3 \colon |\bm{x}|_\nu \leq 2 r_\nu/\alpha^{1/s_K} \}.
 \]
Define
    \[
    \mathcal{B} = \prod_\nu\mathcal{B}_\nu.
    \]
    Since $S(\alpha,(r_\nu)_\nu)$ defines a bounded set we may cover it with $M$, say, translates of $\mathcal{B}$,  denoted by $\mathcal{B}_i$. We pay choose the $\mathcal{B}_i$ such that the pairwise intersection of these boxes has trivial measure, and also such that $\mathcal{B}_i \cap S(\alpha,(r_\nu)_\nu) \neq \emptyset$.

  By translating $\Delta^{-1}(\mathcal{B}_i)$ by a point contained in $\Lambda \cap \Delta^{-1}(\mathcal{B}_i)$ if necessary, then via translation invariance of the lattice we find
    \[
    \#(\Lambda \cap \Delta^{-1}(\mathcal{B}_i)) \ll \# \left\{ \bm{x} \in \Lambda \colon |\bm{x}|_\nu \ll  \frac{r_\nu}{\alpha^{1/s_K}} \text{ for all } \nu \mid \infty \right\},
    \]
    for all $i = 1, \hdots, M$. Lemma~\ref{lem.lattice_in_box} therefore delivers
    \[
    \#(\Lambda \cap \Delta^{-1}(\mathcal{B}_i)) \ll 1 + \frac{\prod_\nu r_\nu}{\alpha\lambda_1^{d_K}}+ \frac{\left(\prod_\nu r_\nu\right)^2}{\alpha^2 (\lambda_1\lambda_2)^{d_K}} + \frac{\left(\prod_\nu r_\nu\right)^3}{\alpha^3 \det (\Lambda)},
    \]
    and thus
    \begin{equation} \label{eq.lattice_in_quadric_in_terms_of_M}
     {N}(\alpha,\bm{r}) \ll M \left( 1 + \frac{\prod_\nu r_\nu}{\alpha\lambda_1^{d_K}}+  \frac{\left(\prod_\nu r_\nu\right)^2}{\alpha^2 (\lambda_1\lambda_2)^{d_K}} + \frac{\left(\prod_\nu r_\nu\right)^3}{\alpha^3 \det (\Lambda)}\right).
    \end{equation}   
    To obtain an upper bound on the number of boxes $M$ that we need in order to cover this region, consider $(\bm{x}_\nu)_\nu \in \mathcal{B}_i \cap S(\alpha,(r_\nu)_\nu)$. Then all points inside $\mathcal{B}_i$ are of the form $(\bm{x}_\nu)_\nu + (\bm{y}_\nu)_\nu$ with $|\bm{y}_\nu|_\nu \ll r_\nu/\alpha^{1/s_K}$. Therefore we have that
    \[
    |\bm{x}_\nu + \bm{y}_\nu|_\nu \ll \alpha^{1/s_K} + \frac{r_\nu}{\alpha^{1/s_K}} \ll \alpha^{1/s_K},
    \]
    since we only consider $(r_\nu)_\nu$ such that  $r_\nu \ll \alpha^{2/s_K}$ holds. Further we find
    \begin{align*}
        |Q(\bm{x}_\nu + \bm{y}_\nu)|_\nu &= |Q(\bm{x}_\nu) + \bm{x}_\nu^T \nabla Q(\bm{y}_\nu) + Q(\bm{y_\nu})|_\nu \\
        &\ll r_\nu + \alpha^{1/s_K} \frac{r_\nu}{\alpha^{1/s_K}} + \left(\frac{r_\nu}{\alpha^{1/s_K}}\right)^2 \\
        &\ll r_\nu.
    \end{align*}
    Since $i$ was arbitrary, we deduce that there exists a constant $C>0$, only depending on the quadratic form $Q$ and $K$ such that 
    \[
    \bigcup_{i=1}^M \mathcal{B}_i \subset S(C\alpha, C(r_\nu)_\nu).
    \]
    Hence 
    \begin{equation} \label{eq.bound_for_M}
           M \ll \left(\frac{\alpha}{\prod_\nu r_\nu}\right)^3 \vol S(\alpha,(r_\nu)_\nu).
    \end{equation}
We will now compute the volume of $S(Q,\alpha,(r_\nu)_\nu)$.
 Note first that
\[
\vol S(\alpha,(r_\nu)_\nu) = \prod_{\nu \mid \infty} \vol S_\nu(\alpha,r_\nu),
\]
where
\[
S_\nu(\alpha,r_\nu) =  \left\{\bm{x} \in K_\nu^3 \colon |\bm{x}|_\nu \ll \alpha^{1/s_K}, |Q(\bm{x})|_\nu \ll r_\nu  \right\}.
\]
We consider two different cases. Firstly, if $Q$ is not isotropic over $K_\nu$ then there exists some constant $D >0$ depending only on $Q$ such that $|Q(\bm{x})|_\nu \geq D$ holds for all $\bm{x} \in K_\nu^3$ such that $|\bm{x}|_\nu = 1$. Therefore $|Q(\bm{x})|_\nu \geq A |\bm{x}|_\nu^2$ holds, and so $\bm{x} \in S_\nu(\alpha,R)$ implies that we must have $|\bm{x}|_\nu \ll r_\nu^{1/2}$. As a result we easily deduce
\[
\vol S_\nu(\alpha,R)  \ll r_\nu^{3/2} \ll \alpha^{1/s_K} r_\nu
\]
via recalling that we only consider $r_\nu$ such that $r_\nu \ll \alpha^{2/s_K}$.
If $Q$ is isotropic on the other hand, then by the classical theory of quadratic forms after a linear transformation of the coordinates in $K_\nu$ we may assume that it is of the shape
\[
Q(x,y,z) = yz-Dx^2,
\]
for some constant $D\in K_\nu$ depending on $Q$ that can be $0$ if the rank of $Q$ is 2. Note first that away from the nullset $z = 0$ we clearly have
\[
\vol\left\{ y \in K_\nu \colon |yz-Dx^2|_\nu \ll r_\nu, |y|_\nu \ll \alpha^{1/s_K} \right\} \ll \min\left\{ \alpha^{1/s_K}, \frac{r_\nu}{|z|_\nu} \right\},
\]
for any $x \in K_\nu$. Hence we find
\begin{align*}
\vol S_\nu(\alpha,r_\nu) \ll \int_{|x|_\nu \ll \alpha^{1/s_K}} \int_{|z|_\nu \ll \alpha^{1/s_K}} \min\left\{ \alpha^{1/s_K}, \frac{r_\nu}{|z|_\nu} \right\} dzdx \ll r_\nu \alpha^{1/s_K+\varepsilon}.
\end{align*}
We conclude that
\[
\vol S(\alpha,R) \ll \alpha^{1+\varepsilon} \prod_{\nu \mid \infty} r_\nu.
\]
From~\eqref{eq.bound_for_M} it thus follows that
\[
M \ll \frac{\alpha^{4+\varepsilon}}{(\prod_\nu r_\nu)^2}.
\]
Finally, recall $\prod_\nu r_\nu \asymp R$.
Using all of this along with~\eqref{eq.lattice_in_quadric_in_terms_of_M} and recalling the contribution from $Q(\bm{x}) = 0$ we have
\begin{align*}
    N(\alpha,R)
    &\ll \frac{\alpha^{4+\varepsilon}}{R^2} + \frac{\alpha^{3+\varepsilon}}{R \lambda_1^{d_K}}  + \frac{\alpha^{2+\varepsilon}}{(\lambda_1\lambda_2)^{d_K}} + \frac{\alpha^{1+\varepsilon} R}{\det \Lambda} + \frac{\alpha}{(\det \Lambda)^{1/3}}
\end{align*}
which completes the proof of this Lemma.
\end{proof}

\subsection{del Pezzo surfaces of degree 4} 
Our proof is very similar to the one given by Browning and Swarbrick-Jones \cite{browning_S-J} except for the fact that we avoid the Thue--Siegel--Roth theorem.

A del Pezzo surface of degree $4$ may be written as a complete intersection of two quadrics inside $\PP^4$.
After a change of variables if necessary, if $X$ contains a conic we may write the system of quadrics as
\[
x_0x_1-x_2x_3 = Q(x_0,x_1,x_2,x_3) + x_4^2 = 0
\]
inside $\PP^4$, where $Q$ is a quadratic form defined over $\O_K$. Taking $U \subset X$ to be the Zariski open set obtained after removing the exceptional locus (in this case it consists of $16$ lines) we find two conic fibrations $\pi_i \colon U \rightarrow \PP^1$ explicitly given by
\[
\pi_1(\bm{x}) = \begin{cases}
    [x_0,x_2], \quad &\text{if $(x_0,x_2) \neq \bm{0}$,} \\
    [x_3,x_1],\  &\text{if $(x_3,x_1) \neq \bm{0}$,}
\end{cases}
\]
and
\[
\pi_2(\bm{x}) = \begin{cases}
    [x_0,x_3], \quad &\text{if $(x_0,x_3) \neq \bm{0}$,} \\
    [x_2,x_1],  &\text{if $(x_2,x_1) \neq \bm{0}$.}
    \end{cases}
\]
In particular, this gives rise to a well-defined morphism $\pi \colon X \rightarrow \PP^1 \times \PP^1$. Denoting the Segre embedding by $\psi \colon \PP^1 \times \PP^1 \rightarrow \PP^3$ one can easily verify that we have
\[
H_{}(\pi_1(\bm{x})) H_{}(\pi_2(\bm{x})) = H(\psi(\pi(\bm{x})))
\]
for all $\bm{x} \in X(K)$.
One may further easily check that 
\[
\psi \circ \pi ([x_0,x_1,x_2,x_3,x_4]) = [x_0,x_3,x_2,x_1]
\]
holds, whenever $[x_0,x_1,x_2,x_3,x_4] \in U(K)$, and so the degree of $\psi \circ \pi \colon U \rightarrow \PP^3$ is $1$.
The functoriality of heights~\cite[Section 2.3]{SerreLecturesMordellWeil} therefore shows that we have
\[
H_{}(\pi_1(\bm{x})) H_{}(\pi_2(\bm{x})) \ll H(\bm{x}),
\]
for any $\bm{x} \in U(K)$. In particular it follows that given $\bm{x} \in U(K)$ we must have $H(\pi_i(\bm{x})) \ll H(\bm{x})^{1/2}$ for $i=1$ or $i=2$. We deduce that
\[
N_U(B) \leq n_1(B) + n_2(B),
\]
where
\[
n_i(B) = \# \left\{\bm{x}\in U(K) \colon H(\bm{x}) \leq B, \; H(\pi_i(\bm{x})) \ll B^{1/2} \right\}.
\]
It suffices to show $n_1(B) \ll B^{1+\varepsilon}$ since dealing with $n_2(B)$ is identical.
Given $(s,t) \in Z_1$ denote by $n_1(B,s,t)$ the cardinality of the points in the fibre $\pi_1^{-1}([s,t]) \cap U(k)$ of height at most $O(B^{1/2})$.
In order to show the desired bound for $n_1(B)$ via a dyadic decomposition argument it suffices to show the same bound for
\[
n_1(R,B) \coloneqq \sum_{\substack{(s,t) \in Z_1 \\ \norm{(s,t)}_\infty \sim R}} n_1(B,s,t).
\]
where $R \ll B^{1/2}$.
Given $(s,t) \in Z_1$ we have that a representative $(ux,yv,xv,yu,z) \in Z_4$ lies in the fibre $\pi_1^{-1}([s,t]) \cap U(k)$ precisely when 
\begin{equation} \label{eq.conic_dp4}
Q(sx,yt,xt,ys) + az^2 = 0.
\end{equation}
For fixed $(s,t)$ we may regard the above as a ternary quadratic form in $(x,y,z)$. Then the discriminant $\Delta(s,t)$ defines a separable, homogeneous polynomial with $\deg \Delta(s,t) = 4$. Note also that there are no points in $\pi_1([s,t])^{-1} \cap U(k)$ such that $\Delta(s,t) = 0$. We may bound $n_1(B,s,t)$ by the number of $(x,y,z) \in Z_2$ such that~\eqref{eq.conic_dp4} and
\[
|(sx,yt,xt,ys,z)|_\nu \ll B^{1/s_K}
\]
is satisfied for all $\nu \mid \infty$. Thus we may employ Corollary~\ref{Cor: ConicsDisc} in order to find
\[
n_1(R,B) \ll \sum_{\substack{(s,t) \in Z_1\\ \norm{(s,t)}_\infty \sim R \\ \Delta(s,t) \neq 0}} B^\varepsilon \left(1+ \frac{B \mathrm{N}(\Delta_0(s,t))^{1/2}}{R^{2/3}\mathrm{N}(\Delta(s,t))^{1/3}} \right).
\]
First note that the number of $(s,t) \in Z_1$ such that $\norm{(s,t)}_\infty \sim R$ is bounded by $R^2 \ll B$. Further, an argument of Broberg~\cite[Lemma 7]{broberg2001rational} shows that we have $\mathrm{N}(\Delta_0(s,t)) \ll 1$ for all $(s,t)\in Z_1$ such that $\Delta(s,t) \neq 0$. We are thus left with estimating
\begin{equation} \label{eq.dp4_last_estimate}
\sum_{\substack{(s,t) \in Z_1\\ \norm{(s,t)}_\infty \sim R \\ \Delta(s,t) \neq 0}}  \frac{B^{1+\varepsilon}}{R^{2/3}\mathrm{N}(\Delta(s,t))^{1/3}}. 
\end{equation}
We will proceed by further dividing the value of $\mathrm{N}(\Delta(s,t))$ into dyadic intervals. Say $\mathrm{N}(\Delta(s,t)) \sim S$, then we clearly have $1 \leq S \ll R^4$ and thus the number of dyadic intervals we have to consider is bounded by $O(B^\varepsilon)$.
We can apply Lemma~\ref{lem.binary_form_separable_bound} to find that there are at most $O(R^{1+\varepsilon}(1+S/R^{3}))$ many points $(s,t) \in Z_1$ such that $\norm{(s,t)}_\infty \sim R$ and $\mathrm{N}(\Delta(s,t)) \sim S$. Hence if $R \leq S$ then  we may bound~\eqref{eq.dp4_last_estimate} by
\[
\frac{B^{1+\varepsilon} R}{R^{2/3}S^{1/3}}\left(1+\frac{S}{R^3}\right) \ll \frac{B^{1+\varepsilon} R^{1/3}}{S^{1/3}} + \frac{B^{1+\varepsilon} S^{2/3}}{R^{8/3}} \ll B^{1+\varepsilon,} 
\]
since $S\ll R^4$, which is satisfactory. It therefore remains to bound
\[
\sum_{\substack{(s,t) \in Z_1 \\ \norm{(s,t)}_\infty \sim R \\ \mathrm{N}(\Delta(s,t)) \ll R}} n_1(B,s,t).
\]
We would like to apply Lemma~\ref{lem.quadratic_lemma_for_dp4_broberg} in order to bound $n_1(B,s,t)$ so we need to make sure the conditions are satisfied. Let $q(x,y,z) = Q(sx,ty,tx,sy) + az^2$. Then we have 
\[
q(0,y,z) = y^2 Q(0,t,0,s) + az^2,
\]
and
\[
q(x,0,z) = x^2 Q(s,0,t,0) + az^2.
\]
Given some $(s,t) \in K^2$ both of the above binary forms are singular only if $Q(s,0,t,0)=Q(0,t,0,s) = 0$. This can happen for at most $O(1)$ primitive points $(s,t) \in Z_1$ unless $Q(s,0,t,0)$ and $Q(0,t,0,s)$ are both identically zero. If this were the case, however, then it is straightforward to check that $X$ has a singular point.
We may bound the contribution to $\sum_{(s,t) \in Z_1} n_1(B,s,t)$ such that $Q(s,0,t,0)=Q(0,t,0,s) = 0$ clearly by $O(B^{1+\varepsilon})$. Otherwise, at least one of $q(0,y,z)$ or $q(x,0,z)$ is non-singular, whence it follows from Lemma~\ref{lem.quadratic_lemma_for_dp4_broberg} that for such $(s,t)$ we have 
\[
n_1(B,s,t) \ll \frac{B^{1+\varepsilon}}{R}.
\]
Finally, via Lemma~\ref{lem.binary_form_separable_bound} the number of  $(s,t) \in Z_1$ such that $\norm{(s,t)}_\infty \sim R$ and $\mathrm{N}(\Delta(s,t)) \ll R$ holds is bounded above by $R^{1+\varepsilon}$. We conclude that
\[
\sum_{\substack{(s,t) \in Z_1 \\ \norm{(s,t)}_\infty \sim R \\ \mathrm{N}(\Delta(s,t)) \ll R}} n_1(B,s,t) \ll R^{1+\varepsilon} \frac{B^{1+\varepsilon}}{R} \ll B^{1+\varepsilon},
\]
as desired.

\section{del Pezzo surfaces of degree 3--5}\label{Sec.dp345}
In this section we will prove Theorem~\ref{Th: TheTheorem} for the cases $d=3,4,5$. Let $X$ be a del Pezzo surface of degree $3\leq d\leq 5$ over $K$ and $U$ the complement of the exceptional curves. The anti-canonical divisor $-K_X$ induces an embedding $X\subset \PP^d$ realising $X$ as a smooth non-degenerate surface of degree $d$. Let $\bm{c}\in \PP^d(K)$ and denote by $H_{\bm{c}}\subset \PP^d(K)$ the hyperplane defined by $\bm{x}\cdot \bm{c}=0$. Moreover, we define $X_{\bm{c}}=H_{\bm{c}}\cap X$. By the adjunction formula, we have 
    \[
    2p_a(X_{\bm{c}})-2=0,
    \]
where $p_a$ denotes the arithmetic genus. It follows that if $X_{\bm{c}}$ is smooth and geometrically irreducible, then either $X(K)=\emptyset$ or $X_{\bm{c}}$ defines an elliptic curve over $K$.

From Lemma \ref{lem.siegel} we know that every $\bm{x}\in \PP^d(K)$ with $H(\bm{x})\leq B$ lies in $H_{\bm{c}}$ for some $\bm{c}\in \PP^d(K)$ with $H(\bm{c})\ll B^{1/d}$. We thus define
\[
N_{\bm{c}}(B)\coloneqq \{\bm{x}\in (X_{\bm{c}}\cap U)(K) \colon H(\bm{x})\leq B\} 
\]
and infer that
\begin{equation}\label{Eq:MainCount.dP345}
    N(B)\leq \sum_{\substack{\bm{c}\in \PP^d(K)\\ H(\bm{c})\ll B^{1/d}}}N_{\bm{c}}(B).
\end{equation}
Let $1\leq e\leq d$ denote the maximum of the degrees of the irreducible components of $X_{\bm{c}}$ defined over $K$. It will be convenient to consider separately the contribution from each $e$. Note that if $e\leq 2$, then $X_{\bm{c}}$ is either a union of lines, in which case we do not count its rational points, or it contains a conic defined over $K$. In particular, in this case $X$ admits a conic bundle structure over $K$ and we have the superior upper bound $N_U(B)\ll B^{1+\varepsilon}$ from Theorem \ref{Th: TheConicBundleTheorem} for $d=4,5$, so that it suffices to consider the contribution from $3\leq e \leq d$ when $d=4$ or $d=5$.
\subsection*{$\bm{e=d}$} Let us first consider the contribution from those $\bm{c}$ such that $X_{\bm{c}}$ is non-singular. In this case Proposition \ref{Prop.UpperBoundEllCurve} gives $N_{\bm{c}}(B) \ll  B^\varepsilon$. Moreover, the number of available $\bm{c}$ is $O(B^{(d+1)/d})$ by Lemma~\ref{Le: Number.OK.Points} so that we get an overall contribution of $O(B^{(d+1)/d+\varepsilon})$ to \eqref{Eq:MainCount.dP345}, which is sufficient. 

Next we assume that $X_{\bm{c}}$ is irreducible but singular. In this case $N_{\bm{c}}(B)\ll B^{2/d}$ by Proposition~\ref{lem.irreducible_curve_uniform}. In addition, $X_{\bm{c}}$ being singular implies that $\bm{c}\in X^*(K)$. We know that $\deg(X^*)=12$ by Proposition~\ref{Prop: DualReflexive}, so that by Proposition \ref{Prop.DimGrowth} the number of such $\bm{c}$ with $H(\bm{c})\ll B^{1/d}$ is $O(B^{(d-2)/d})$. Therefore, we get a contribution of $O(B^{(d+1)/d})$, which is again satisfactory. 

\subsection*{$\bm{e=d-1}$} 
This case turns out to be the most difficult one and we have to treat it individually for each value of $d$. Note that this case occurs precisely when $X_{\bm{c}}=L\cup D$, where $L\subset X$ is a line and $D\subset X_{\bm{c}}$ is an irreducible curve over $K$ of degree $d-1$.

\subsubsection*{$d=5$} We now treat the contribution from those hyperplanes for which $X_{\bm{c}}=D\cup L$, where $D$ is an irreducible quartic curve and $L$ a line contained in $X$. There are at most 10 lines in $X$ defined over $D$, and so we may restrict our attention to a specific one. After a suitable change of variables we can assume that $L$ is given by $x_2=\cdots = x_5=0$. Furthermore, any hyperplane containing $L$ takes the shape $H_{\bm{c}}=\{c_2x_2+\cdots +c_5x_5=0\}$ for some $[0, 0 , c_2, \hdots , c_5]\in \PP^5(K)$. We may take a representative $\bm{c}\in Z_5$ and assume without loss of generality that $\norm{c_5}=\norm{\bm{c}}\asymp H(\bm{c})^{1/s_K}$. 

Let $p_1=[1, 0, \hdots , 0]$ and consider the projection $X\setminus \{p_1\} \to \PP^4$. Since $p$ is a non-singular point of $X$, the closure of the image $Y\subset \PP^4$ is a surface of degree 4. In fact, the morphism corresponds to blowing up the point $p_1$ and then contracting the strict transforms of the resulting $(-2)$-curve and so $Y$ is a singular del Pezzo surface of degree 4. More explicitly, since $p_1$ lies on a line, the point $p_0=[1, 0 , 0, 0 , 0]$ will be a singularity.  Any (possibly singular) del Pezzo surface of degree 4 inside $\PP^4$ can be written as the intersection of two quadrics. Since $p_0\in Y$, we can write $Y=V(Q_1,Q_2)$ with 
\begin{align}\label{Eq: WeakdP4}
\begin{split}
    Q_1(x) &=x_1L_1(x_2,x_3,x_4,x_5)-q_1(x_2,x_3,x_4,x_5), \\
    Q_2(x)&=x_1L_2(x_2,x_3, x_4, x_5)-q_2(x_2,x_3,x_4,x_5)
    \end{split}
\end{align}
and where $L_i, q_i\in \O_K[x_2,\dots, x_5]$ are linear and quadratic forms respectively. The Jacobian criterion applied to the singular point $p_1$ implies that $L_1$ and $L_2$ are proportional, so that there exist $J\in \O_K[x_2,\dots, x_5]$ and $a\in \O_K$ such that $L_1=J$ and $L_2=aJ$. 

Let now $Z\subset \PP^3$ be the quadric surface defined by $aq_1(x_2,x_3,x_4,x_5)=q_2(x_2,x_3,x_4,x_5)$. The projection $[x_1, \hdots , x_5]\mapsto [x_2, \hdots , x_5]$ then defines a morphism $Y\setminus\{p_0\}\to Z$. Moreover, it is easily seen that \begin{align*}
\psi[(x_2,\dots, x_5])&=[q_1(x_2,\hdots, x_5), x_2J(x_2,\hdots, x_5), \hdots , x_5J(x_2,\hdots, x_5)]\\
    &=[aq_2(x_2,\dots, x_5), x_2J(x_2,\dots, x_5), \hdots , x_5J(x_2,\dots ,x_5)]
\end{align*}
is a well-defined inverse to the projection map from $Z\to Y\setminus\{p_0\}$ away from the locus $S\subset \PP^3$ defined by $J=q_1=q_2=0$.

As there are $O(1)$ rational points $[x_0, \hdots , x_5]$ in $D\setminus L$ above a point $[x_1, \hdots , x_5]$, it suffices to count $\bm{x}\in \PP^4(K)$ that lie on on the image of $D$ under the projection map with $H(\bm{x})\leq B$. Since $D\subset H_{\bm{c}}$, any such rational point takes the shape 
\[
[t_1, c_5t_2, c_5 t_3, c_5t_4, -(c_2t_2+c_3t_3+c_4t_4)]
\]
with $\bm{t}=[t_1, \hdots , t_4]\in \PP^3(K)$. 
For $g\in K[x_2,x_3,x_4,x_5]$, define 
\[
g_{\bm{c}}(t_2,t_3,t_4)= g(c_5t_2,c_5t_3,c_5t_4, -(c_2t_2+c_3t_3+c_4t_4)).
\]
Similarly, for $\bm{t}\in \PP^3(K)$, we define 
\[
H^{(\bm{c})}(\bm{t})=H(t_1,c_5t_2, c_5t_3, c_5t_4,-(c_2t_2+c_3t_3+c_4t_4)).
\]
Upon setting
\[
N'_{\bm{c}}(B)=\#\{\bm{t}\in\PP^3(K)\colon H^{(\bm{c})}(\bm{t})\leq B, aq_{1,\bm{c}}(t_2,t_3,t_4)=q_{2,\bm{c}}(t_2,t_3,t_4)=t_1J(t_2,t_3,t_4)\},
\]
it transpires from \eqref{Eq: WeakdP4} that $N_{\bm{c}}(B)\ll N'_{\bm{c}}(B)$.
Our goal is to show the following estimate, where by abuse of notation we denote by $H_{\bm{c}}$ also the hyperplane in $\PP^3$ defined by $c_2x_2+c_3x_3+c_4x_4+c_5x_5=0$.
\begin{prop}\label{Prop: degenerate.hyperplane.dp5}
    Assume that \begin{enumerate}[(i)]
        \item $H_{\bm{c}}\cap S=\emptyset$, where $S\subset \PP^3$ is the variety defined by $J=q_1=q_2=0$,
        \item $H_{\bm{c}}$ and $Z$ intersect transversally,
        \item $H_{\bm{c}}$ and $Z\cap \{J=0\}$ intersect transversally. \end{enumerate}
        Then 
        \[
        N_{\bm{c}}'(B)\ll B^{1/2+\varepsilon} \left(\norm{c_5}_\infty^{-1/2}+\norm{c_5}_\infty^{1/2}\mathrm{N}(\Delta(\bm{c}))^{-1/3}\right),
        \]
        where $\Delta(\bm{c})$ is the  discriminant of the quadratic form $aq_{1,\bm{c}}-q_{2,\bm{c}}$.
\end{prop}
Suppose for a moment that Proposition~\ref{Prop: degenerate.hyperplane.dp5} is proven already. After splitting $|c_5|_\nu$ for $\nu\mid\infty$ into dyadic intervals, one readily verifies that 
\[
\sum_{\substack{\bm{c}\in Z_3 \\ \norm{\bm{c}}_\infty \ll B^{1/5}}}\norm{c_5}_\infty^{-1/2}\ll B^{7/10}.
\]
Moreover, we can also split $\mathrm{N}(\Delta(\bm{c}))$ into dyadic intervals, say $\mathrm{N}(\Delta(\bm{c}))\sim \beta$ and $\norm{\bm{c}}_\infty \sim \alpha$ with $\alpha \ll B^{1/5}$ and $\beta \ll \alpha^6$. As $\Delta(\bm{c})$ is homogeneous of degree $6$, we can use Lemma~\ref{Le: Polysmall} to deduce that 
\[
\alpha^{1/2}\beta^{-1/3}\sum_{\substack{\bm{c}\in Z_3 \\ \norm{\bm{c}}_\infty\sim \alpha \\ \mathrm{N}(\Delta(\bm{c}))\sim \beta}} 1 \ll \alpha^{7/2}\beta^{-1/6}\ll \alpha^{7/2}.
\]
Summing over $\alpha$ gives $\sum_{\alpha \ll B^{1/5}}\alpha^{7/2}\ll B^{7/10}$. As there are $O(\alpha^{6\varepsilon})=O(B^\varepsilon)$ possibilities for $\beta$, in total we obtain a contribution of $B^{6/5+\varepsilon}$ under the assumption that the conditions of Proposition~\ref{Prop: degenerate.hyperplane.dp5} are satisfied. To deal with the remaining cases, we need the following lemma. 
\begin{lemma}\label{Le: Dim.BadHyper.dp5}
    There is a Zariski closed subset $W\subset \PP^3$ whose irreducible components have dimension at most 2 such that if $\bm{c}\in \PP^3$ fails (i), (ii) or (iii) in Proposition~\ref{Prop: degenerate.hyperplane.dp5}, then $\bm{c}$ lies in $W$. 
\end{lemma}
\begin{proof}
    We begin with a preliminary observation. We claim that the variety $S\subset \PP^3$ is 0-dimensional. Suppose for  a contradiction that there exists a positive dimensional irreducible component $E\subset S$. Then by looking at the defining equations of $Y$, it becomes clear that the closure $E'$ of $\A^1\times E$ is contained in $Y$. However, $E'$ has dimension 2 and $Y$ is an irreducible surface, so they must coincide. It is clear that $E'$ is contained in the hyperplane defined by $L$ inside $\PP^4$, but del Pezzo surfaces are  not contained in any hyperplane under the anti-canonical embedding.

   If $\bm{c}\in \PP^3$ fails (i), then $H_{\bm{c}}\cap S\neq\emptyset$. From what we have just shown, it follows that $S$ is a union of $O(1)$ points in $\PP^3$. In particular, $H_{\bm{c}}\cap S\neq\emptyset$ implies that $H_{\bm{c}}$ contains one of these points, which forces $\bm{c}$ to lie in one of $O(1)$ irreducible subvarieties of codimension at least 1 in $\PP^3$. 

    Next, assume that $\bm{c}$ does not satisfy (ii). This implies that $\bm{c}$ lies on the dual variety of $Z$ or intersects the singular locus of $Z$. As $Z$ is is birational to $X$, it is again a (possibly singular) del Pezzo surface and hence has at most isolated singularities. This again implies that $\bm{c}$ lies on $O(1)$ irreducible subvarieties of codimension at least 1 in $\PP^2$. 

    Finally, suppose that $\bm{c}$ fails (iii). Note that $Z\cap \{J=0\}$ cannot be a double line, as it corresponds birationally to a hyperplane section containing $L$ of $X$. As $X$ is smooth, this implies that hyperplane sections are reduced. It follows that $Z\cap \{J=0\}$ is isomorphic to a plane conic. If it is smooth, then the failure of (iii) implies that $\bm{c}$ lies on the dual variety of $Z\cap \{J=0\}$ inside $\PP^3$, which has codimension at least 1. When it is not smooth, then it is a union of two lines that intersect in a unique point $P\in \PP^3$ and $H_{\bm{c}}$ has tangential intersection if and only if $P\in H_{\bm{c}}$, which again forces $\bm{c}$ to lie on a linear subspace of $\PP^3$ of dimension 2. 
\end{proof}
Suppose now that $\bm{c}$ does not meet the requirements of Proposition~\ref{Prop: degenerate.hyperplane.dp5}. By Proposition~\ref{lem.irreducible_curve_uniform} we have $N_{\bm{c}}(B)\ll B^{1/2}$. Moreover, it follows from Lemma~\ref{Le: Dim.BadHyper.dp5} that $\bm{c}$ lies on $O(1)$ irreducible subvarieties of dimension at most 2 inside $\PP^3$. By \eqref{Eq:UnifVariety} the number of such $\bm{c}\in \PP^3(K)$ with $H(\bm{c})\ll B^{1/5}$ is $O(B^{3/5})$. Hence we get an overall contribution of $O(B^{11/10})$ to \eqref{Eq:MainCount.dP345} from this case, which is sufficient. To complete the case $e=4$ and $d=5$, we are therefore left with proving Proposition~\ref{Prop: degenerate.hyperplane.dp5}.
\begin{proof}[Proof of Proposition~\ref{Prop: degenerate.hyperplane.dp5}]

 Let $\bm{t}=(t_2,t_3,t_4)\in Z_2$ be a representative for $\bm{t}\in \PP^2(K)$ such $aq_{1,\bm{c}}(\bm{t})=q_{2,\bm{c}}(\bm{t})$ and write
\[
\mathfrak{a}=\langle q_{1,\bm{c}}(\bm{t}), c_5t_2J_{\bm{c}}(\bm{t}), c_5t_3J_{\bm{c}}(\bm{t}), c_5t_4J_{\bm{c}}(\bm{t}), -J_{\bm{c}}(\bm{t})(c_2t_2+c_3t_3+c_4t_4)\rangle.
\]
Let $\mathfrak{C}$ be the ideal generated by the resultant of the homogeneous forms $q_{1,\bm{c}}, q_{2,\bm{c}},J_{\bm{c}}$. Note that if $\mathfrak{C}=0$, then $J_{\bm{c}}, q_{1,\bm{c}}$ and $q_{2,\bm{c}}$ share a common root in $\PP^1$, which implies that $H_{\bm{c}}$ has non-empty intersection with $S$. As we assume that (i) holds, this is impossible. Since $\norm{\bm{c}}_\infty \ll B^{1/5}$, we immediately get $\mathrm{N}(\mathfrak{C})\ll B^A$. In addition, because $\bm{t}\in Z_2$, we clearly have $\mathfrak{a}\mid \mathfrak{A} \mathfrak{C}\langle c_5\rangle $, where $\mathfrak{A}=\mathfrak{a}_1\cdots \mathfrak{a}_h$ and $\mathfrak{a}_1,\dots, \mathfrak{a}_h$ are the fixed representatives of the class group of $K$.

We then have $H^{(\bm{c})}(\psi(\bm{t}))\leq B$ if and only if 
\begin{equation}\label{Eq: HeightCondition.dp5}
\norm{q_{1,\bm{c}}(\bm{t}), c_5t_2J_{\bm{c}}(\bm{t}), c_5t_3J_{\bm{c}}(\bm{t}), c_5t_4J_{\bm{c}}(\bm{t}), -J_{\bm{c}}(\bm{t})(c_2t_2+c_3t_3+c_4t_4)}_\infty\leq B \mathrm{N}(\mathfrak{a}).    
\end{equation}

Note that in the above display we may also replace $q_{1,\bm{c}}$ by $aq_{2,\bm{c}}$. As we assume that $q_{1,\bm{c}}, q_{2,\bm{c}}$ and $J_{\bm{c}}$ do not have a common root in $\overline{K}$, the inequality \eqref{Eq: HeightCondition.dp5} together with Lemma~\ref{Le: FunctorialityHeights} implies that 
\[
\norm{\bm{t}}_\infty \ll (B\mathrm{N}(\mathfrak{a}))^{A}. 
\]
Let us now write $\mathfrak{a}=\mathfrak{b}_1\mathfrak{b}_2$, where $\mathfrak{b}_1=\langle \mathfrak{a}, J_{\bm{c}}(\bm{t})\rangle$, so that in particular $\mathfrak{b}_1\mid \mathfrak{C}$ and $\mathfrak{b}_2\mid \mathfrak{A}\langle c_5\rangle $. 
From \eqref{Eq: HeightCondition.dp5} we get that
\[
\norm{J_{\bm{c}}(\bm{t})}_\infty \norm{c_5}_\infty \norm{(t_2,t_3,t_4)}_\infty \leq H^{(\bm{c})}(\psi(\bm{t}))\mathrm{N}(\mathfrak{a})\leq B\mathrm{N}(\mathfrak{a}),
\]
and so we must have that $J_{\bm{c}}(s,t)\leq B^{1/2}\mathrm{N}(\mathfrak{b}_1)\norm{c_5}_\infty^{-1/2}$ or $\norm{\bm{t}}_\infty \leq B^{1/2}\mathrm{N}(\mathfrak{b}_2)\norm{c_5}_\infty^{-1/2}\ll B^{1/2}\norm{c_5}_\infty^{1/2}$, where we used that $\mathfrak{b}_2\mid \mathfrak{A}\langle c_5\rangle$. In particular, it follows from our discussion so far that 
\[
N_{\bm{c}}'(B)\leq \sum_{\mathfrak{b}_1\mid \mathfrak{C}}\sum_{\mathfrak{b}_2\mid \mathfrak{A}\langle c_5\rangle }\left(n_1((B\mathrm{N}(\mathfrak{a}))^{A}, B^{1/2}\mathrm{N}(\mathfrak{b}_1)\norm{c_5}_\infty^{-1/2})+n_2(B^{1/2}\norm{c_5}_\infty^{1/2})\right),
\]
where for positive reals $R_1, R_2$ we have set
\[
n_1(R_1,R_2)=\#\{\bm{t}\in Z_2\colon \norm{\bm{t}}_\infty \ll R_1, \norm{J_{\bm{c}}(\bm{t})}_\infty \ll R_2, \mathfrak{b}_1\mid J_{\bm{c}}(\bm{t}), aq_{1,\bm{c}}(\bm{t})=q_{2,\bm{c}}(\bm{t})\}
\]
and 
\[
n_2(R_1)=\#\{\bm{t}\in Z_2\colon \norm{\bm{t}}_\infty \ll R_1, aq_{1,\bm{c}}(\bm{t})=q_{2,\bm{c}}(\bm{t})\}.
\]
Let us first focus on estimating $n_1(R_1,R_2)$. Let $C_{\bm{c}}$ be the conic defined by $aq_{1,\bm{c}}(\bm{t})=q_{2,\bm{c}}(\bm{t})$ inside $\PP^2(K)$ and note that since we assume that (ii) holds, $C_{\bm{c}}$ is geometrically irreducible. There are two possibilities: Either there is no point $\bm{t}\in C_{\bm{c}}(K)$ with $H(\bm{t})\ll R_1$, in which case $n_1(R_1,R_2)=0$, or such a point exists and we can use it to obtain a parameterisation $\PP^1(K)\to C_{\bm{c}}(K)$. Explicitly, this is done by sending a line through the point to its unique residual intersection point with the conic. In this way we obtain quadratic forms $g_1,g_2,g_2\in \O_K[u,v]$ without a common factor such that the map $\psi \colon \PP^1\to C_{\bm{c}}$ given by $[s, t]\mapsto [g_1(s,t), g_2(s,t), g_3(s,t)]$ gives a bijection of $\PP^1(K)$ with $C_{\bm{c}}(K)$. As we assume that the height of the initial point is bounded by $R_1$, it is clear that we can take the quadratic forms in such a way that $\norm{g_i}\ll R_1^A$.  

Let us  define $Q(s,t)=J_{\bm{c}}(g_1(s,t),g_2(s,t),g_3(s,t))$. We  claim that $Q(s,t)$ is square-free. To see this, note that if it has a double root $[s_0,t_0]$, then  $P_0=[g_1(s_0,t_0), g_2(s_0,t_0), g_3(s_0,t_0)]$ will satisfy $J_{\bm{c}}(P_0)=0$ and $P_0\in C_{\bm{c}}$. In particular, it will be a double point of $C_{\bm{c}}\cap \{J_{\bm{c}}=0\} = Z\cap \{J=0\}\cap H_{\bm{c}}$. As we assume that (iii) holds, this is impossible, which verifies the claim.

It is clear that $\norm{Q}\ll \norm{J_{\bm{c}}}\max_{i=1,2,3}\norm{g_i}\ll B^A R_1^A$, so that if $\norm{Q(s,t)}_\infty \ll R_2$ holds, then we must also have $\norm{(s,t)}_\infty \ll (BR_1R_2)^A$ for any $(s,t)\in Z_1$ by Lemma~\ref{Le: FunctorialityHeights}. In particular, we have 
\[
n_1(R_1,R_2) \leq \#\{(s,t)\in Z_1\colon \norm{(s,t)}_\infty \ll (BR_1R_2)^A, \mathfrak{b}_1\mid Q(s,t), \norm{Q(s,t)}_\infty \ll R_2\}.
\]
As $Q$ is square-free, we can invoke Corollary~\ref{Cor: QuadRepII} to bound this last quantity and deduce that
\begin{align*}
    n_1((B\mathrm{N}(\mathfrak{a}))^{A}, B^{1/2}\mathrm{N}(\mathfrak{b}_1)\norm{c_5}_\infty^{-1/2}) \ll (B^{A}\mathrm{N}(\mathfrak{a})^{A}\norm{Q})^\varepsilon  \frac{B^{1/2}}{\norm{c}_\infty^{1/2}}
    \ll B^{1/2+\varepsilon}\norm{c_5}_\infty^{-1/2},
\end{align*}
where we used that $\mathrm{N}(\mathfrak{a}), \mathrm{N}(\mathfrak{b}_1), \norm{Q}\ll B^A$.

Next we turn to $n_2(R_1)$. Let $\Delta(\bm{c})$ be the discriminant of the quadratic form $aq_{1,\bm{c}}-q_{2,\bm{c}}$ and let $\Delta_0(\bm{c})$ be the ideal generated by the $2\times 2$ minors of the matrix underlying $aq_{1,\bm{c}}-q_{2,\bm{c}}$. Note that as we assume that (ii) holds, we must have $\Delta(\bm{c})\neq 0$. Applying Corollary \ref{Cor: ConicsDisc} directly to $n_2(R_1)$ gives 
\[
n_2(B^{1/2}\norm{c_5}_\infty^{1/2})\ll \left(1+\frac{B^{1/2}\norm{c_5}_\infty^{1/2}\mathrm{N}(\Delta_0(\bm{c}))^{1/2}}{\mathrm{N}(\Delta(\bm{c}))^{1/3}}\right)\mathrm{N}(\Delta(\bm{c}))^\varepsilon.
\]
We clearly have $\mathrm{N}(\Delta(\bm{c}))^\varepsilon \ll B^\varepsilon$. Moreover, we claim that $\mathrm{N}(\Delta_0(\bm{c}))$ is in fact bounded. To see this, first assume that the $2\times 2$ minors $M_{ij}(\bm{c})$ have a common zero in $\overline{K}$. This would imply that a hyperplane section of the  quadric surface $Z$ is a double line. However, hyperplane sections of $Z$ correspond birationally to hyperplane sections containing $L$ of our original quintic del Pezzo surface $X$ and hence are reduced.
By Hilbert's Nullstellensatz we can find forms $f_{ijk}\in \O_K[x_1,\dots, x_4]$ and $a_k\in \O_K\setminus\{0\}$ such that 
\[
a_kx_k^d=\sum f_{ijk}M_{ij}(x_1,\dots, x_4)
\]
as an identity in $\O_K[x_1,\dots, x_4]$ for $k=1,\dots, 4$. It follows that \[
\mathrm{N}(\Delta_0(\bm{c}))\ll \mathrm{N}(a_1c_1^d,\dots, a_4c_4^d)\ll 1\]
for $\bm{c}\in Z_3$, as claimed. Using the usual divisor bound for ideals, it follows from our discussion so far that 
\begin{align*}
N_{\bm{c}}(B)&\ll \sum_{\mathfrak{b}_1\mid\mathfrak{C}}\sum_{\mathfrak{b}_2\mid \mathfrak{c_5}}\left(B^{1/2+\varepsilon}\norm{c_5}^{-1/2}_\infty + \frac{B^{1/2+\varepsilon }\norm{c_5}_\infty^{1/2}}{\mathrm{N}(\Delta(\bm{c}))^{1/3}}\right)\\
&\ll B^{1/2+\varepsilon} \left(\norm{c_5}_\infty^{-1/2}+\norm{c_5}_\infty^{1/2}\mathrm{N}(\Delta(\bm{c}))^{-1/3}\right),  
\end{align*}
which is what we wanted to show.
\end{proof}

\subsection*{$d=4$} Next we assume $X_{\bm{c}}=D\cup L$, where $L\subset X$ is a line and $D\subset X$ is an irreducible cubic curve, both defined over $K$. There are at most 16 lines in $X$, and so we may restrict our attention to a specific one. After a suitable change of variables $L$ is given by $x_2=x_3=x_4=0$ if we work with coordinates $x_0,\dots, x_4$ on $\PP^4$. It follows that $X$ is defined by 
\begin{equation}\label{Eq: Equation.dp4}
    x_0L_1 +x_1K_1 +Q_1 = x_0L_2+x_1K_2+Q_2=0,
\end{equation}
where $L_1, L_2, K_1, K_2 \in \O_K[x_2,x_3,x_4]$ are linear and $Q_1, Q_2 \in \O_K[x_2,x_3,x_3]$ are quadratic forms respectively. If $L_1$ and $L_2$ are proportional, then $[1,0, \hdots , 0 ]$ is a singular point of $X$, which is impossible. We can therefore eliminate $x_0$ from \eqref{Eq: Equation.dp4} to obtain the equation 
\begin{equation}\label{Eq: Equation.dp4.2}
    C(x_2,x_3, x_4)+x_1Q(x_2,x_3,x_4)=0
\end{equation}
for some cubic form $C\in \O_K[x_2,x_3,x_4]$. Any hyperplane containing $L$ is defined by $\bm{c}\in \PP^4(K)$ of the shape $\bm{c}=[0, 0, c_2 , c_3 , c_4]$. Without loss of generality we may assume that $(0,0, c_2,c_3,c_4)\in Z_4$ is a representative of $\bm{c}$ such that $\norm{c_4}=\norm{\bm{c}}$, so that $\norm{c_4}\asymp H(\bm{c})^{1/s_K}$. Any rational point $\bm{t}\in H_{\bm{c}}$ takes the shape $\bm{t}=[t_0, t_1, c_4 t_2 ,c_4 t_3, -(c_2t_2+c_3t_3)]$ with $[t_0, \hdots , t_3]\in \PP^3(K)$. In particular, the equation \eqref{Eq: Equation.dp4.2} transforms into 
\begin{equation}\label{Eq: PlaneCubic.dp4}
    C_{\bm{c}}(t_2,t_3)=t_1Q_{\bm{c}}(t_2,t_3),
\end{equation}
where we write $G_{\bm{c}}(t_2,t_3)=G(c_4t_2,c_4t_3, -c_2t_2-c_3t_3)$ for any $G\in\O_K[x,y,z]$. Moreover, it suffices to count $\bm{t}'=[t_1, t_2, t_3]\in \PP^2(K)$ such that \eqref{Eq: PlaneCubic.dp4} holds, since there are $O(1)$ available $\bm{t}$ lying above a given $\bm{t}'$, and we redefine $N_{\bm{c}}(B)$ to be this quantity. Our goal is to establish the following estimate.
\begin{prop}\label{Prop: Degenerate.hyperplane.dp4}
    If $Q_{\bm{c}}$ is square-free, then 
    \[
    N_{\bm{c}}(B) \ll \frac{B^{2/3+\varepsilon}}{\norm{{c}_4}_\infty^{2/3}}.
    \]
\end{prop}
Suppose for a moment that Proposition~\ref{Prop: Degenerate.hyperplane.dp4} holds. The contribution from such $\bm{c}$ to \eqref{Eq:MainCount.dP345} is then
\[
B^{2/3+\varepsilon}\sum_{\substack{\bm{c}\in Z^2 \\ \norm{\bm{c}}_\infty \ll B^{1/4}}}\norm{{c}_4}_\infty^{-2/3}\ll B^{5/4+\varepsilon},
\]
which is satisfactory. It remains to deal with the case when $Q_{\bm{c}}$ has a repeated root. Note that the cubic surface $Y$ defined by \eqref{Eq: Equation.dp4.2} is birational to $X$ and in particular again a (possibly singular) del Pezzo surface. If the curve $Z$ defined $Q=0$ defines a double line in $\PP^2$, then one can check via the Jacobian criterion that the singular locus of $Y$ has dimension at least $1$, which is impossible. Thus $Z$ has only isolated singularities. Now if $Q_{\bm{c}}$ has a double root, then $H_{\bm{c}}$ intersects the singular locus of $Z$, or $\bm{c}$ lies on the dual variety of $Z$ inside $\PP^2$. Either case defines a variety inside $\PP^2$ of dimension at most 1. In particular, there are most $O(B^{1/2})$ available $\bm{c}$ with $H(\bm{c})\ll B^{1/4}$ by \eqref{Eq:UnifVariety}. Moreover, Proposition \ref{lem.irreducible_curve_uniform} implies that $N_{\bm{c}}(B)\ll B^{2/3}$, so that we have a contribution of $O(B^{7/6})$, which is sufficient.

\begin{proof}[Proof of Proposition~\ref{Prop: Degenerate.hyperplane.dp4}]
If $D$ is irreducible, but not geometrically irreducible, then it contains $O(1)$ points by B\'ezout's theorem, which is sufficient. So we assume from now on that $D$ is geometrically irreducible. In particular, $Q_{\bm{c}}$ and $C_{\bm{c}}$ do not share a common factor. The equation \eqref{Eq: PlaneCubic.dp4} therefore defines a singular cubic plane curve $C'\subset \PP^2$ with a singularity at $t_0=[1, 0 , 0]$ and hence is rational over $K$. We obtain an explicit bijective morphism $\PP^1\to C'$ by sending a line to the unique residual intersection point with $C'$ and $t_0$. Explicitly, this is given by $[s, t]\mapsto [C_{\bm{c}}(s,t),  sQ_{\bm{c}}(s,t),  tQ_{\bm{c}}(s,t)]$. It follows that 
\[
N_{\bm{c}}(B)\ll \#\{ [s,t]\in \PP^1(K)\colon H([C_{\bm{c}}(s,t), c_4sQ_{\bm{c}}(s,t), c_4tQ_{\bm{c}}(s,t), -Q_{\bm{c}}(s,t)(c_2 s +c_3 t)])\ll B\}
\]
and we define the last quantity to be $N'_{\bm{c}}(B)$.  As $\norm{\bm{c}}\ll B^{1/4s_K}$, it is clear that $\norm{Q_{\bm{c}}}, \norm{C_{\bm{c}}}\ll B^A$. Thus if $(s,t)\in Z_1$ is counted by $N'_{\bm{c}}(B)$, then Lemma~\ref{Le: FunctorialityHeights} implies that $\norm{(s,t)}_\infty \ll B^A$. Let $\mathfrak{C}'$ be the ideal generated by the resultant of $Q_{\bm{c}}(s,t)$ and $C_{\bm{c}}(s,t)$ and define $\mathfrak{C}=\mathfrak{C}'\mathfrak{a}_1\cdots \mathfrak{a}_h$. By construction, as $(s,t)\in Z_1$, we must then have $\langle C_{\bm{c}}(s,t), sQ_{\bm{c}}(s,t), tQ_{\bm{c}}(s,t)\rangle \mid \mathfrak{C}$ for any $(s,t)\in Z_1$ and $\mathrm{N}(\mathfrak{C})\ll \max\{\norm{Q_{\bm{c}}}, \norm{C_{\bm{c}}}\}^A\ll B^A$. Suppose now that 
\[
\mathfrak{a}=\langle C_{\bm{c}}(s,t), c_4sQ_{\bm{c}}(s,t), c_4tQ_{\bm{c}}(s,t), Q_{\bm{c}}(s,t)(c_2s+c_3t)\rangle.
\]
If we write $\mathfrak{b}_1=\langle \mathfrak{a},Q_{\bm{c}}(s,t)\rangle$ and $\mathfrak{a}=\mathfrak{b}_1\mathfrak{b}_2$, then we must have $\mathfrak{b}_1\mid \mathfrak{C}$ and  $\mathfrak{b}_2\mid \langle c_4, c_2s+c_3t\rangle$. Moreover, for any $(s,t)\in Z_1$ we have 
\begin{align*}
H([C_{\bm{c}}(s,t), & c_4sQ_{\bm{c}}(s,t), c_4tQ_{\bm{c}}(s,t), -Q_{\bm{c}}(s,t)(c_2 s +c_3 t)])\\
&=\mathrm{N}(\mathfrak{a})^{-1}\norm{C_{\bm{c}}(s,t), c_4sQ_{\bm{c}}(s,t), c_4tQ_{\bm{c}}(s,t), -Q_{\bm{c}}(s,t)(c_2 s +c_3 t)}_\infty\\
&\geq \mathrm{N}(\mathfrak{a})^{-1}\norm{c_4}_\infty\norm{Q_{\bm{c}}(s,t)}_\infty \norm{(s,t)}_\infty
\end{align*}
In particular, if $(s,t)$ is counted by $N'_{\bm{c}}(B)$, then we must have 
\begin{equation}\label{Eq: twocases.dp4}
\norm{Q_{\bm{c}}(s,t)}_\infty\leq \mathrm{N}(\mathfrak{b}_1)B^{2/3}\norm{c_4}_\infty^{-2/3}\quad\text{or}\quad \norm{(s,t)}_\infty \leq \mathrm{N}(\mathfrak{b}_2)B^{1/3}\norm{c_4}_\infty^{-1/3}.
\end{equation}
Let us first bound the contribution from those $(s,t)$ for which the first alternative holds. As we have $\mathfrak{b}_1\mid Q_{\bm{c}}(s,t)$, by Corollary \ref{Cor: QuadRepII} we have that the contribution is up to a constant at most
\begin{align*}
     B^\varepsilon \norm{Q_{\bm{c}}}^\varepsilon\frac{B^{2/3}}{\norm{c_4}_\infty^{2/3}}\ll \frac{B^{2/3+\varepsilon}}{\norm{c_4}_\infty^{2/3}}.
\end{align*}
Let us now deal with the contribution form those $(s,t)$ for which the second alternative in \eqref{Eq: twocases.dp4} holds. For any such $(s,t)$, we must have $\mathfrak{b}_2\mid c_2s+c_3t$. As $\mathfrak{b}_2\mid c_4$ and $\mathrm{N}(\langle c_2,c_3,c_4\rangle) \ll 1$, it follows that the cardinality of the image of $Z_1$ in
\[
\#\{ (s,t) \in (\O_K/\mathfrak{b}_2)^{2}\colon c_2s+c_3t\equiv 0 \mod \mathfrak{b}_2\}
\]
is $O(1)$. Denoting the the image of $Z_1$ in the set above by $\mathcal{T}$, any $(s,t)$ must be congruent to some element in $\mathcal{T}$ modulo $\mathfrak{b}_2$, and hence the number of such $(s,t)$ with $\norm{(s,t)}_\infty \ll B^{1/3}\mathrm{N}(\mathfrak{b}_2)\norm{c_4}_\infty^{-1/3}$ is $O(B^{2/3}\norm{c_4}_\infty^{-2/3})$. As $\mathfrak{b}_1\mid \mathfrak{b}$, $\mathfrak{b}_2\mid c_4$ and $\mathrm{N}(\mathfrak{C})\ll B^A$, $\mathrm{N}(c_4)\ll B^{1/4}$, the familiar divisor bound for ideals shows that number of available $\mathfrak{b}_1$ and $\mathfrak{b}_2$ is $O(B^\varepsilon)$. In summary, we have
\[
N_{\bm{c}}'(B)\ll \frac{B^{2/3+\varepsilon}}{\norm{c_4}_\infty^{2/3}},
\]
which is what we wanted to show.
\end{proof}
\subsubsection*{$d=3$} Suppose that $X_{\bm{c}}$ contains a line $L\subset X$. Since there are at most 27 lines in $X$ defined over $K$, it suffices to bound $N_{{\bm{c}}}(B)$ when $X_{\bm{c}}$ contains a fixed line $L\subset X$. After a suitable change of variables $L$ is given by $x_1=x_2=0$, in which case we may assume that $F$ is of the shape 
\[
F(\bm{x})=x_1Q_1(\bm{x})+x_2Q_2(\bm{x})
\]
for some quadratic forms $Q_1,Q_2\in \O_K[x_1,\dots, x_4]$. Moreover, the fact that $L\subset V_{\bm{c}}$ also implies that $\bm{c}=[s, t, 0, 0]$. Any $\bm{x}\in \PP^3(k)$ that lies on $L$ is necessarily of the form $\bm{x}=[tx_1, -sx_1, x_2,x_3]$. In particular, if $\bm{x}$ satisfies $F(\bm{x})=0$, then either $x_1=0$, in which case $\bm{x}\in L$, or $\bm{x}$ is a root of 
\[
Q_{s,t}(\bm{x})\coloneqq tQ_1(tx_1,-sx_1,x_2,x_3)-sQ_2(tx_1,-sx_1,x_2,x_3).
\]
It is clear that if $(x_1,x_2,x_3)\in Z_2$ satisfies $H(tx_1, -sx_1, x_2, x_3)\leq B$, then we must have $\norm{x_1}_\infty \ll B \norm{(s,t)}_\infty^{-1}$. Let $\Delta(s,t)$ be the discriminant of $Q_{s,t}$ considered as a ternary quadratic form in $(x_1,x_2,x_3)$, so that $\Delta(s,t)$ is homogeneous of degree 5. In addition, define $\Delta_0(s,t)$ to be the ideal generated by the $2\times 2$ minors of the matrix underlying $Q_{s,t}$. Then Lemma 7 of \cite{broberg2001rational} shows that $\mathrm{N}(\Delta_0(s,t))\ll 1$, and hence Proposition \ref{Prop: UniformConics} implies 
\[
N_{s,t}(B)\ll B^\varepsilon+\frac{B^{1+\varepsilon}}{\norm{(s,t)}^{1/3}_\infty \mathrm{N}(\Delta(s,t))^{1/3}},
\]
where we used the divisor estimate $\tau(\mathrm{N}(\Delta(s,t)))\ll B^\varepsilon$, since $\mathrm{N}(\Delta(s,t))\ll B^{5/3}$. There are $O(B^{2/3})$ available $(s,t)\in Z_1$ with $H(s,t)\ll B^{1/3}$ and hence if $B^\varepsilon$ dominates in the estimate above, we get a contribution of $O(B^{2/3+\varepsilon})$, which is sufficient. Let us now put $\norm{(s,t)}_\infty$ and $\mathrm{N}(\Delta(s,t))$ into dyadic intervals, say $\norm{(s,t)}_\infty \asymp \alpha $ and $\mathrm{N}(\Delta(s,t))\asymp \beta$, where $\alpha \ll B^{1/3}$ and $\beta \ll B^{5/3}$. By Lemma \ref{Le: Polysmall} the number of available $(s,t)$ is $O(\alpha \beta^{1/5})$, and hence if the second term dominates we get a contribution of 
\[
B^{1+\varepsilon} \alpha^{2/3}\beta^{-2/15}.
\]
Summing over dyadic intervals shows that the overall contribution is $O(B^{11/9+\varepsilon})$, which is sufficient and therefore completes our treatment of cubic surfaces.
\subsection*{$\bm{e=d-2}$} Note that this case is only relevant when $e=5$ and hence $d=3$. Suppose now that $X_{\bm{c}}= B\cup C$, where $C$ is irreducible of degree 3. If $B$ is a conic, we are in the context of Theorem \ref{Th: TheConicBundleTheorem}, so that we may assume that $B$ is the union of two skew lines $L_1, L_2$. There are $O(1)$ pairs of skew lines in $X$, and so we may restrict our attention to a fixed pair. Since $L_1$ and $L_2$ are skew, any hyperplane containing them must contain the three-dimensional linear space they span. This forces $\bm{c}$ to lie on a line in $\PP^5$, and hence the number of such hyperplanes of height $O(B^{1/5})$ is $O(B^{2/5})$. Moreover, by Proposition \ref{lem.irreducible_curve_uniform} the curve $C$ contains $O(B^{2/3})$ rational points of height $B$, so that the contribution is $O(B^{2/5+2/3})=O(B^{16/15})$, which is again satisfactory and thus completes our proof of Theorem \ref{Th: TheTheorem} for $d=3,4,5$.

\section{del Pezzo surfaces of degree 2}
The anticanonical model of a smooth del Pezzo surface of degree $2$ is given by a hypersurface in weighted projective space $\PP(2,1,1,1)$ in the variables $(y,u,v,w)$ of the form
\[
X:  y^2 = g(u,v,w),
\]
where $g\in \O_K[u,v,w]$ is a non-singular quartic form. 
For $[y, u, v, w]\in \PP(2,1,1,1)(K)$, define 
\[
H(y,u,v,w)=\prod_{\nu \in \Omega_K}\max\{|y|_\nu^{1/2}, |u|_\nu, |v|_\nu, |w|_\nu\}.
\]
Let $U\subset X$ be the complement of the exceptional curves of $X$. The counting function with respect to the anti-canonical height function of $X$ is then given by
\[
N(B) = \# \left\{ [y,u,v,w]\in U(K)\colon   H([y,u,v,w]) \leq B \right\}.
\]
Let $\pi \colon X\to \PP^2$ be the map induced by the anti-canonical divisor, that is $\pi([y, u, v, w])=[u, v, w]$. Given $\bm{c} \in \PP^2(K)$ we let $H_{\bm{c}}\subset \PP^2$ be the line defined by $\bm{x}\cdot \bm{c}=0$. We may consider the pullback of $H_{\bm{c}}$ under $\pi$. Writing $X_{\bm{c}}=\pi^{-1}(H_{\bm{c}})$, the defining curve for $X_{\bm{c}}$ can be written as
\[
y^2 = g_{\bm{c}}(x,t)
\]
for a suitable binary quartic form $g_{\bm{c}}$. If we write
\[
N_{\bm{c}}(B) = \# \left\{ \bm{t}\in (X_{\bm{c}}\cap U)(K)\colon H(\bm{t})\leq B\right\},
\]
then using Lemma~\ref{lem.siegel} we find
\[
N(B) \leq \sum_{\substack{\bm{c}\in \PP^2(K) \\ H(\bm{c})\ll B^{1/2}}} N_{\bm{c}}(B).
\]
We will split the study of $N_{\bm{c}}(B)$ into three cases. 

First consider the case where $g_{\bm{c}}(x,t)$ has no multiple factors. If $X_{\bm{c}}$ has no rational points of height at most $B$ then trivially $N_{\bm{c}}(B) = 0$ for such $\bm{c}$. Otherwise, let $\bm{x}_0$ be a rational point of $X_{\bm{c}}$ with $H( \bm{x}_0 ) \leq B$. Then $X_{\bm{c}}$ is non-singular and defines an elliptic curve, as is shown in~\cite[Chapter 8]{cassels_elliptic_book}. In particular, it is shown there that there exists an isomorphism (which is independent of $\bm{c}$) that birationally maps $X_{\bm{c}}$ to an elliptic curve in Weierstrass normal form
\[
E_{\bm{c}} \colon ZY^2 = X^3 + A_{\bm{c}}XZ^2 + B_{\bm{c}}Z^3.
\]
Note that any rational point $\bm{x}$ of $X_{\bm{c}}$ is mapped to a rational solution $P(\bm{x})$, say. Since the birational transformation only depends polynomially on the coefficients of $g_{\bm{c}}$ there exist absolute constants $A, C>0$ such that
\[
\norm{P(\bm{x})} \leq C \norm{g_{\bm{c}}}^A \norm{\bm{x}_0}^A \norm{\bm{x}}^A.
\]
Finally, since different rational points are mapped to different rational points under this birational transformation we find
\[
N_{\bm{c}}(B) \ll N(E_{\bm{c}}, CB^A),
\]
where $N(E,R)$ denotes the number of rational points of an elliptic curve $E\subset \PP^2$ up to height $R$ with respect to the usual height on $\PP^2$. Thus, by Proposition~\ref{Prop.UpperBoundEllCurve} we find that in this case
\[
N_{\bm{c}}(B) \ll B^\varepsilon.
\]
There are $O(B^{3/2})$ rational points $\bm{c}\in \PP^2(K)$ with $H(\bm{c})\ll B^{1/2}$, so that the total contribution in this case is bounded by $O_\varepsilon(B^{3/2+\varepsilon})$.

Next, we consider the case when $g_{\bm{c}}(x,t)$ has a multiple factor but $X_{\bm{c}}$ is geometrically irreducible. This implies that $g_{\bm{c}}=L^2Q$, where $L,Q\in \overline{K}[x,t]$. Note that in fact $L$ and hence also $Q$ must have coefficients in $\O_K$, because we assume $\cha(K)\neq 2$. Indeed, if not then $Q$ must also be square, which is impossible as then $y^2-g_{\bm{c}}$ is reducible. Hence we may assume that $L$ and $Q$ are both defined over $K$. There are most $O(1)$ possible choices $(x,t)\in \PP^1(K)$ such that $L(x,t)Q(x,t)=0$, which also forces $y=0$. It thus suffices to bound the contribution from those $(y,x,t)$ for which $L(x,t)Q(x,t)\neq 0$. If we write $z=y/L(x,t)$, then the equation $y^2=L^2(x,t)Q(x,t)$ implies that $z^2=Q(x,t)$. Moreover, if $H(y,x,t)\leq B$, then $z$ will have a representative in $\O_K$ with $\norm{z}_\infty \ll B^2$ and $(x,t)$ a representative in $\O_K^2$ with $\norm{(x,t)}_\infty \ll B$. It follows that
\[
N_{\bm{c}}(B) \leq \#\{(z,x,t)\in \PP^2(K)\colon \norm{(x,t)}_\infty \ll B, \norm{z}_\infty \ll B^{2}, z^2=Q(x,t)\}.
\]
The binary form $Q(x,t)$ must be square-free, as else $y^2-L^2Q$ would be reducible, a case that we excluded by assumption. In particular, $z^2=Q(x,t)$ defines an irreducible conic in $\PP^2$ and we can use Proposition \ref{Prop: UniformConics} to deduce that 
\[
N_{\bm{c}}\ll B^{4/3}.
\]
We will now bound the number of $\bm{c}\in \PP^2(K)$ such that $g_{\bm{c}}$ has a multiple factor. This happens precisely if the hyperplane $H_{\bm{c}}$ has singular intersection with the quartic curve $V(g)\subset \PP^2$ and hence $\bm{c}$ must lie on the dual curve $V(g)^*\subset \PP^2$ of $V(g)$. As $V(g)$ is a smooth quartic curve and $\cha(K)\neq 2,3$, the dual curve is an irreducible curve of degree 12. By Proposition~\ref{lem.irreducible_curve_uniform} it follows that the number of $\bm{c}\in V(g)^*$ with $H(\bm{c})\ll B^{1/2}$ is $O(B^{1/12})$. Therefore in total we get a contribution of $O(B^{17/12})$ from this case.

Finally and lastly we assume that $y^2-g_{\bm{c}}(x,t)$ is geometrically reducible. This happens precisely if $g_{\bm{c}}$ is a perfect square as a polynomial in $\overline{K}[x,t]$. In particular, any such $H_{\bm{c}}$ is a bitangent vector of $V(g)$. The 56 exceptional curves in $X$ correspond precisely to the preimages of the 28 bitangents of $V(g)$ under $\pi$, and so any such point on $X_{\bm{c}}$ is excluded from our count.

We conclude that in total we obtain $N(B) \ll_\varepsilon B^{3/2 + \varepsilon}$, which suffices for Theorem~\ref{Th: TheTheorem} when $d=2$.

\section{del Pezzo surfaces of degree 1}
Via the anticanonial embedding, a non-singular del Pezzo surface $X$ of degree $1$ may be realised as a hypersurface in weighted projective space $\PP(3,2,1,1)$ in the variables $(y,x,u,v)$ given by an equation of the form
\[
X: y^2 = x^3 + g(u,v)x + h(u,v),
\]
where $g$ and $h$ are binary forms of degrees $4$ and $6$, respectively. An anti-canonical height function is given by 
\[
H(\bm{t})=\prod_{\nu \in \Omega_K}\max\{|y|^{1/3}_\nu, |x|^{1/2}_\nu, |u|_\nu, |v|_\nu\}
\]
for $\bm{t}=[y,x,u,v]\in \PP(3,2,1,1)(K)$.
The counting function of interest is then given by
\[
N(B) = \# \{\bm{t}\in X(K)\colon H(\bm{t})\leq B\}.
\]
In this case we are aiming for an upper bound of the form $O(B^{2+\varepsilon})$, and so we do not have to remove any of the exceptional curves. Moreover, the point $[1, 1, 0, 0]$ is the unique base point of the anti-canonical linear system, which induces the rational map $[y, x, u , v]\mapsto [u, v]\in \PP^1$ and so we may restrict  our attention to counting those points for which $(u, v)\neq \bm{0}$. 

Note that upon replacing $[y, x,  u, v]\in \PP(3,2,1,1)(K)$ with $[\mu^3y, \mu^2x, \mu u, \mu v]$ for a suitable unit $\mu \in \O_K^\times$, we see from Lemma~\ref{lem.good_units} that every element of $\PP(3,2,1,1)(K)$ with height at most $B$ has a representative $(y,x,u,v)\in\O_K^4$ with $\norm{y}^{1/3}, \norm{x}^{1/2}, \norm{(u,v)} \ll B^{1/s_K}$. It thus follows that 
\[ 
N(B) \leq \sum_{\substack{(u,v)\in \O_K^2 \\ 0<\norm{(u,v)}\ll B^{1/s_K}}} N_{u,v}(B),
\]
where 
\[
N_{u,v}(B)\coloneqq \#\{(y,x)\in \O_K^2\colon \norm{y}\ll B^{3/s_K}, \norm{x}\ll B^{2/s_K}, y^2=x^3+g(u,v)x+h(u,v)\}.
\]
Define $\Delta(u,v) = 4g(u,v)^2+27h(u,v)^3$. Note that this is a homogeneous polynomial of degree $12$. If $(u,v)\in \O_K^2$ is such that $\Delta(u,v) \neq 0$ then $ZY^2 = X^3 + g(u,v)XZ^2 + h(u,v)Z^3$ defines an elliptic curve in $\PP^2$ and any integral point $(y,x)\in \O_K^2$ gives rise to a unique rational point on the elliptic curve. By Proposition \ref{Prop.UpperBoundEllCurve} we thus find $N_{(u,v)}(B) \ll B^\varepsilon$ in this case. Since the number of $(u,v)\in \O_K^2$ with $\norm{u,v}\ll B^{1/s_K}$ is $O(B^2)$ by Lemma~\ref{Le: Number.OK.Points}, we get a total contribution of $O_\varepsilon(B^{2+\varepsilon})$, which is sufficient.

It remains to estimate $\sum_{(u,v)} N_{u,v}(B)$ where the sum runs over $(u, v)\in \O_K^2$ that satisfy $\Delta(u,v) = 0$. Note first that there are at most twelve solutions $[u_i,v_i]$, $i = 1, \hdots, 12$ to  $\Delta(u,v) = 0$ in $\PP^1(K)$. Further, any $(u,v)\in \O_K^2$ such that $[u, v]=[u_i, v_i]$ in $\PP^1(K)$ must be an $\O_K$-multiple of $(u_i,v_i)$. By Lemma~\ref{Le: Number.OK.Points} we know that there are most $O(B\norm{(u_i,v_i)}_\infty^{-1})=O(B)$ pairs $(u,v)= \lambda(u_i,v_i)$ with $\norm{\lambda}\ll B^{1/s_K}$ and $\mathrm{N}(\lambda)\ll B\norm{(u,v)}_\infty^{-1}$. We thus find that the number of summands that we currently consider is bounded by $O(B)$, where the implied constant is independent of $X$. Moreover, it is clear that 
\[
N_{u,v}(B) \leq \#\{(y,x)\in \O_K^2\colon \norm{(y,x)}\ll B^{3/s_K}, y^2=x^3+g(u,v)x+h(u,v)\}
\]
and that the affine curve defined by $y^2=x^3+g(u,v)x+h(u,v)\subset \mathbb{A}^2$ is irreducible. The affine version of the Bombieri-Pila bound generalised to global fields by Paredes and Sasyk \cite[Theorem 1.9]{paredes2021uniform}, implies that the last quantity is bounded by $O(B)$, where the implied constant is independent of $X$. Hence we obtain a total contribution of $O(B^{2})$, which completes the proof of Theorem~\ref{Th: TheTheorem} for $d=1$.

\printbibliography
\end{document}